\RequirePackage{fix-cm}
\documentclass[smallextended]{svjour3}
\usepackage{amsmath,amsfonts}
\usepackage{booktabs}
\usepackage{cancel}
\usepackage{cases}
\usepackage{comment}
\usepackage{hyperref}
\usepackage{multirow}
\usepackage{nicefrac}
\usepackage{paralist}
\usepackage{tikz-cd}

\usepackage{pgfplots,pgfplotstable}
\usepackage[font=small]{subcaption}

\usepackage{cabin} 
\usepackage{newtxmath}

\usepackage{makecell}

\usetikzlibrary{patterns}

\numberwithin{equation}{section}

\newcommand{\st}{\,:\,}
\newcommand{\Real}{\mathbb{R}}
\newcommand{\Natural}{\mathbb{N}}

\DeclareRobustCommand{\bvec}[1]{\boldsymbol{#1}}
\newcommand{\uvec}[1]{\underline{\bvec{#1}}}
\newcommand{\cvec}[1]{\bvec{\mathcal{#1}}}

\newcommand{\rotation}[1]{\varrho_{#1}}

\DeclareMathOperator{\GRAD}{\bf grad}
\DeclareMathOperator{\CURL}{\bf curl}
\DeclareMathOperator{\DIV}{div}
\DeclareMathOperator{\ROT}{rot}
\DeclareMathOperator{\VROT}{\bf rot}

\newcommand{\compl}{{\rm c}}
 
\newcommand{\Hcurl}[1]{\bvec{\mathrm{H}}(\CURL;#1)}
\newcommand{\Hrot}[1]{\bvec{\mathrm{H}}(\ROT;#1)}
\newcommand{\Hdiv}[1]{\bvec{\mathrm{H}}(\DIV;#1)}

\newcommand{\HDcurl}[1]{\bvec{\mathrm{H}}_0(\CURL;#1)}
\newcommand{\HDdiv}[1]{\bvec{\mathrm{H}}_0(\DIV;#1)}

\newcommand{\Xgrad}[1]{\underline{X}_{\GRAD,#1}^k}

\newcommand{\Xcurl}[1]{\underline{\bvec{X}}_{\CURL,#1}^k}
\newcommand{\Xdiv}[1]{\underline{\bvec{X}}_{\DIV,#1}^k}
\newcommand{\Xbullet}[1]{\underline{X}_{\bullet,#1}^k}

\newcommand{\Igrad}[1]{\underline{I}_{\GRAD,#1}^k}
\newcommand{\Icurl}[1]{\uvec{I}_{\CURL,#1}^k}
\newcommand{\Idiv}[1]{\uvec{I}_{\DIV,#1}^{k}}
\newcommand{\Icurlz}[1]{\uvec{I}_{\CURL,#1}^0}

\newcommand{\NE}[1]{\boldsymbol{\mathcal{N}}^{#1}}
\newcommand{\RT}[1]{\boldsymbol{\mathcal{RT}}^{#1}}

\newcommand{\lproj}[2]{\pi_{\mathcal{P},#2}^{#1}}
\newcommand{\vlproj}[2]{\boldsymbol{\pi}_{\cvec{P},#2}^{#1}}

\newcommand{\Rproj}[2]{\bvec{\pi}_{\cvec{R},#2}^{#1}}
\newcommand{\Rcproj}[2]{\bvec{\pi}_{\cvec{R},#2}^{\compl,#1}}

\newcommand{\Gproj}[2]{\bvec{\pi}_{\cvec{G},#2}^{#1}}
\newcommand{\Gcproj}[2]{\bvec{\pi}_{\cvec{G},#2}^{\compl,#1}}

\newcommand{\Xproj}[2]{\bvec{\pi}_{\cvec{X},#2}^{#1}}
\newcommand{\Xcproj}[2]{\bvec{\pi}_{\cvec{X},#2}^{\compl,#1}}

\newcommand{\uGT}[1][k]{\uvec{G}_T^{#1}}
\newcommand{\uGF}[1][k]{\uvec{G}_F^{#1}}

\newcommand{\uCT}[1][k]{\uvec{C}_T^{#1}}

\newcommand{\uGh}[1][k]{\uvec{G}_h^{#1}}
\newcommand{\uCh}[1][]{\uvec{C}_h^k}
\newcommand{\Dh}[1][]{D_h^k}

\newcommand{\GE}[1][k]{G_E^{#1}}
\newcommand{\cGF}[1][k]{\boldsymbol{\mathsf G}_F^{#1}}
\newcommand{\cGT}[1][k]{\boldsymbol{\mathsf G}_T^{#1}}

\newcommand{\CF}{C_F^k}
\newcommand{\cCT}{\boldsymbol{\mathsf C}_T^k} 

\newcommand{\DT}{D_T^k}

\newcommand{\trF}{\gamma_F^{k+1}}
\newcommand{\trFt}{\bvec{\gamma}_{{\rm t},F}^k}
\newcommand{\trFtR}{\bvec{\gamma}_{{\rm t},\cvec{R},F}^k}

\newcommand{\Pgrad}[1][T]{P_{\GRAD,#1}^{k+1}}
\newcommand{\Pcurl}[1][T]{\bvec{P}_{\CURL,#1}^k}
\newcommand{\Pdiv}[1][T]{\bvec{P}_{\DIV,#1}^k}
\newcommand{\Pcurlz}[1][T]{\bvec{P}_{\CURL,#1}^0}

\newcommand{\PcurlR}{\bvec{P}_{\CURL,\cvec{R},T}^k}
\newcommand{\PdivG}{\bvec{P}_{\DIV,\cvec{G},T}^k}

\newcommand{\RcurlT}{\bvec{R}_{\CURL,T}}
\newcommand{\RcurlF}[1][F]{\bvec{R}_{\CURL,#1}}

\newcommand{\faces}[1]{\mathcal{F}_{#1}}
\newcommand{\edges}[1]{\mathcal{E}_{#1}}
\newcommand{\vertices}[1]{\mathcal{V}_{#1}}

\newcommand{\FT}{\faces{T}}
\newcommand{\ET}{\edges{T}}
\newcommand{\EF}{\edges{F}}
\newcommand{\FE}{\mathcal{F}_E}

\newcommand{\VE}{\vertices{E}}

\newcommand{\normal}{\bvec{n}}
\newcommand{\tangent}{\bvec{t}}

\newcommand{\Poly}[2][]{\mathcal{P}_{#1}^{#2}}
\newcommand{\vPoly}[2][]{\cvec{P}_{#1}^{#2}}
\newcommand{\Roly}[1]{\cvec{R}^{#1}}
\newcommand{\Goly}[1]{\cvec{G}^{#1}}
\newcommand{\cRoly}[1]{\cvec{R}^{\compl,#1}}
\newcommand{\cGoly}[1]{\cvec{G}^{\compl,#1}}
\newcommand{\Qoly}[2][]{\mathcal{Q}_{#1}^{#2}}

\newcommand{\Xoly}[1]{\cvec{X}^{#1}}
\newcommand{\cXoly}[1]{\cvec{X}^{\compl,#1}}

\newcommand{\Leb}{\mathrm{L}^2}
\newcommand{\Sob}[1]{\mathrm{H}^{#1}}
\newcommand{\rC}[1]{\mathrm{C}^{#1}}
\newcommand{\vSob}[1]{\boldsymbol{\mathrm{H}}^{#1}}
\newcommand{\vLeb}{\boldsymbol{\mathrm{L}}^2}
\newcommand{\vC}[1]{\boldsymbol{\mathrm{C}}^{#1}}

\newcommand{\rec}[3]{\mathfrak{R}_{#1}(#2,#3)}
\newcommand{\Xrec}[5]{\mathfrak{R}_{\cvec{#1},#2}^{#3}(#4,#5)}

\newcommand{\recRT}[3]{\Xrec{\cvec{R}}{T}{#1}{#2}{#3}}
\newcommand{\recGT}[3]{\Xrec{\cvec{G}}{T}{#1}{#2}{#3}}

\newcommand{\norm}[2][]{\|#2\|_{#1}}
\newcommand{\seminorm}[2][]{|#2|_{#1}}
\newcommand{\vvvert}{\vert\kern-0.25ex\vert\kern-0.25ex\vert}
\newcommand{\tnorm}[2][]{\vvvert #2\vvvert_{#1}}

\newcommand{\term}{\mathfrak{T}}

\DeclareMathOperator{\Ker}{Ker}
\DeclareMathOperator{\Image}{Im}

\DeclareMathOperator{\card}{card}

\newcommand{\Mh}[1][h]{\mathcal{M}_{#1}}
\newcommand{\Th}[1][h]{\mathcal{T}_{#1}}
\newcommand{\Fh}[1][h]{\mathcal{F}_{#1}}

\newcommand{\Eh}[1][h]{\mathcal{E}_{#1}}
\newcommand{\Vh}{\mathcal{V}_h}

\newcommand{\dEgrad}{\tilde{\mathcal{E}}_{\GRAD,h}}
\newcommand{\dEcurl}{\tilde{\mathcal{E}}_{\CURL,h}}
\newcommand{\dEdiv}{\tilde{\mathcal{E}}_{\DIV,h}}

\newcommand{\CurlCorr}{\bvec{\delta}_T}

\graphicspath{{figures/pdf/}}

\newcommand{\logLogSlopeTriangle}[5]
{
    \pgfplotsextra
    {
        \pgfkeysgetvalue{/pgfplots/xmin}{\xmin}
        \pgfkeysgetvalue{/pgfplots/xmax}{\xmax}
        \pgfkeysgetvalue{/pgfplots/ymin}{\ymin}
        \pgfkeysgetvalue{/pgfplots/ymax}{\ymax}

        \pgfmathsetmacro{\xArel}{#1}
        \pgfmathsetmacro{\yArel}{#3}
        \pgfmathsetmacro{\xBrel}{#1-#2}
        \pgfmathsetmacro{\yBrel}{\yArel}
        \pgfmathsetmacro{\xCrel}{\xArel}

        \pgfmathsetmacro{\lnxB}{\xmin*(1-(#1-#2))+\xmax*(#1-#2)} 
        \pgfmathsetmacro{\lnxA}{\xmin*(1-#1)+\xmax*#1} 
        \pgfmathsetmacro{\lnyA}{\ymin*(1-#3)+\ymax*#3} 
        \pgfmathsetmacro{\lnyC}{\lnyA+#4*(\lnxA-\lnxB)}
        \pgfmathsetmacro{\yCrel}{\lnyC-\ymin)/(\ymax-\ymin)}

        \coordinate (A) at (rel axis cs:\xArel,\yArel);
        \coordinate (B) at (rel axis cs:\xBrel,\yBrel);
        \coordinate (C) at (rel axis cs:\xCrel,\yCrel);

        \draw[#5]   (A)-- node[pos=0.5,anchor=north] {\scriptsize{1}}
                    (B)-- 
                    (C)-- node[pos=0.,anchor=west] {\scriptsize{#4}} 
                    cycle;
    }
}

\begin{document}

\title{An arbitrary-order discrete de Rham complex on polyhedral meshes: Exactness, Poincar\'e inequalities, and consistency\thanks{Communicated by Douglas Arnold.\\
The authors acknowledge the support of \emph{Agence Nationale de la Recherche} through the grant NEMESIS (ANR-20-MRS2-0004-01).
Daniele Di Pietro's work was also partially supported by the fast4hho grant (ANR-17-CE23-0019).
J\'er\^ome Droniou was partially supported by the Australian Government through the \emph{Australian Research Council}'s Discovery Projects funding scheme (grant number DP170100605). The authors also wish to thank Daniel Matthews for suggesting an explicit basis for $\cGoly{\ell}(T)$, which was used in the implementation of DDR in versions 4.0+ of the \texttt{HArDCore3D} library.}}
\author{Daniele A. Di Pietro \and J\'er\^ome Droniou}

\authorrunning{D. A. Di Pietro \and J. Droniou}
\titlerunning{An arbitrary-order discrete de Rham complex on polyhedral meshes}

\institute{%
  D. A. Di Pietro \at
  IMAG, Univ Montpellier, CNRS, Montpellier, France \\
  \email{daniele.di-pietro@umontpellier.fr}
  \and
  J. Droniou \at
  School of Mathematics, Monash University, Melbourne, Australia \\
  \email{jerome.droniou@monash.edu}
}


\date{Received: 28 January 2021 / Revised: 06 August 2021 / Accepted: 07 August 2021}

\maketitle

\begin{abstract}
  In this paper we present a novel arbitrary-order discrete de Rham (DDR) complex on general polyhedral meshes based on the decomposition of polynomial spaces into ranges of vector calculus operators and complements linked to the spaces in the Koszul complex.
  The DDR complex is fully discrete, meaning that both the spaces and discrete calculus operators are replaced by discrete counterparts, and satisfies suitable exactness properties depending on the topology of the domain.
  In conjunction with bespoke discrete counterparts of $\Leb$-products, it can be used to design schemes for partial differential equations that benefit from the exactness of the sequence but, unlike classical (e.g., Raviart--Thomas--N\'ed\'elec) finite elements, are nonconforming.
  We prove a complete panel of results for the analysis of such schemes: exactness properties, uniform Poincar\'e inequalities, as well as primal and adjoint consistency.
  We also show how this DDR complex enables the design of a numerical scheme for a magnetostatics problem, and use the aforementioned results to prove stability and optimal error estimates for this scheme.
  
\end{abstract}

\setcounter{tocdepth}{3}
\tableofcontents


\section{Introduction}

The design of stable and convergent schemes for the numerical approximation of certain classes of partial differential equations (PDEs) requires to reproduce, at the discrete level, the underlying geometric, topological, and algebraic structures.
This leads to the notion of \emph{compatibility}, which can be achieved either in a conforming or non-conforming setting.
Relevant examples include PDEs that relate to the de Rham complex.
For an open connected polyhedral domain $\Omega\subset\Real^3$, this complex reads
\begin{equation}\label{eq:continuous.de.rham}
  \begin{tikzcd}
    \Real\arrow{r}{i_\Omega} & \Sob{1}(\Omega)\arrow{r}{\GRAD} & \Hcurl{\Omega}\arrow{r}{\CURL} & \Hdiv{\Omega}\arrow{r}{\DIV} & \Leb(\Omega)\arrow{r}{0} & \{0\},
  \end{tikzcd}
\end{equation}
where $i_\Omega$ denotes the operator that maps a real value to a constant function over $\Omega$,
$\Sob{1}(\Omega)$ the space of scalar-valued functions over $\Omega$ that are square integrable along with their gradient,
$\Hcurl{\Omega}$ (resp.\ $\Hdiv{\Omega}$) the space of vector-valued functions over $\Omega$ that are square integrable along with their curl (resp.\ divergence).
In order to serve as a basis for the numerical approximation of PDEs, discrete counterparts of this sequence of spaces and operators should enjoy the following key properties:
\begin{enumerate}[\bf (P1)]
\item \emph{Complex and exactness properties.} For the sequence to form a complex, the image of each discrete vector calculus operator should be contained in the kernel of the next one.
  Moreover, the following exactness properties should be reproduced at the discrete level: $\Image i_\Omega = \Ker\GRAD$ (since $\Omega$ is connected); $\Image\GRAD = \Ker\CURL$ if the first Betti number of $\Omega$ is zero; $\Image\CURL=\Ker\DIV$ if the second Betti number of $\Omega$ is zero; $\Image\DIV = \Leb(\Omega)$ (since we are in dimension three).
\item \emph{Uniform Poincar\'e inequalities.} Whenever a function from a space in the sequence lies in some orthogonal complement of the kernel of the vector calculus operator defined on this space, its (discrete) $\Leb$-norm should be controlled by the (discrete) $\Leb$-norm of the operator up to a multiplicative constant independent of the mesh size.
\item \emph{Primal and adjoint consistency.}
  The discrete vector calculus operators should satisfy appropriate commutation properties with the interpolators and their continuous counterparts.
  Additionally, these operators along with the corresponding (scalar or vector) potentials should approximate smooth fields with sufficient accuracy.
  Finally, whenever a formal integration by parts is used in the weak formulation of the problem at hand, the vector calculus operators should also enjoy suitable adjoint consistency properties.
  The notion of adjoint consistency accounts for the failure, in non-conforming settings, to verify global integration by parts formulas exactly.
\end{enumerate}

In the context of Finite Element (FE) approximations, discrete counterparts of the de Rham complex are obtained replacing each space in the sequence with a finite-dimensional subspace.
These subspaces are built upon a conforming mesh of the domain, whose elements are restricted to a small number of shapes and, in practice, are most often tetrahedra; see \cite{Arnold:18} for a complete and extremely general exposition including an exhaustive bibliography, and also \cite{Christiansen.Rapetti:16} on the link between Raviart--Thomas--N\'ed\'elec differential forms and FE systems.
The restriction to conforming meshes made of standard elements can be a major shortcoming in advanced applications, limiting, for example, the capacity for local refinement or mesh agglomeration; see, e.g., the preface of \cite{Di-Pietro.Droniou:20}.
The extension of the FE approach to more general meshes including, e.g., polyhedral elements and non-matching interfaces, is not straightforward.
Recent efforts in this direction have been made in \cite{Gillette.Rand.ea:16,Chen.Wang:17} (see also references therein), focusing mainly on the lowest-order case and with some limitations on the element shapes in three dimensions.
The extension to specific element shapes has also been considered in \cite{Duran.Devloo.ea:19,Devloo.Duran.ea:19}.
  A recent generalisation of FE methods is provided by the Isogeometric Analysis, which is designed to facilitate exchanges with Computer Assisted Design software.
  In this framework, spline spaces and projection operators that verify a de Rham diagram have been developed in \cite{Buffa.Rivas.ea:11}; see also \cite{Buffa.Sangalli.ea:14}.

General polytopal meshes can be handled by several lowest-order methods grounded, to a different extent, in the seminal work of Whitney on geometric integration \cite{Whitney:57}.
These methods share the common feature that discrete de Rham complexes are obtained by replacing both the spaces and operators with discrete counterparts.
Specifically, the spaces consist of vectors of real numbers attached to mesh entities of dimension equal to the index of the space in the sequence (vertices for $\Sob{1}(\Omega)$, edges for $\Hcurl{\Omega}$, faces for $\Hdiv{\Omega}$, and elements for $\Leb(\Omega)$).
In Mimetic Finite Differences, discrete vector calculus operators and $\Leb$-products are obtained by mimicking the Stokes theorem; see \cite{Beirao-da-Veiga.Lipnikov.ea:14} for a complete exposition.
Their extension to polytopal meshes has first been carried out in \cite{Kuznetsov.Lipnikov.ea:04,Lipnikov.Shashkov.ea:06}, then analysed in \cite{Brezzi.Lipnikov.ea:05,Brezzi.Buffa.ea:09}; see also \cite{Droniou.Eymard.ea:10} for a link with the Mixed Hybrid Finite Volume methods of \cite{Droniou.Eymard:06,Eymard.Gallouet.ea:10} and \cite[Section 2.5]{Di-Pietro.Ern.ea:14} along with \cite[Section 3.5]{Di-Pietro.Ern:17} and \cite{Aghili.Boyaval.ea:15} for links with Hybrid High-Order methods.
In the Discrete Geometric Approach, originally introduced in \cite{Codecasa.Specogna.ea:07} and extended to polyhedral meshes in \cite{Codecasa.Specogna.ea:09,Codecasa.Specogna.ea:10}, as well as in Compatible Discrete Operators \cite{Bonelle.Ern:14,Bonelle.Di-Pietro.ea:15}, the key notions are topological vector calculus operators (expressed in terms of incidence matrices) along with the Hodge operator.
The role of the latter is to establish a link, through the introduction of physical parameters, between quantities defined on primal and dual mesh entities.
All of the above schemes are limited to the lowest-order, and their analysis often relies on an interplay of functional and topological arguments that is not required in our approach.

Discretisation methods that provide arbitrary-order approximations on general polyhedral meshes have only recently appeared in the literature.
A first example is provided by the Virtual Element Method, which can be described as a FE method where explicit expressions for the basis functions are not available at each point.
A de Rham complex of virtual spaces on polyhedra has been recently proposed in \cite{Beirao-da-Veiga.Brezzi.ea:16};
important evolutions of this original virtual complex are contained in \cite{Beirao-da-Veiga.Brezzi.ea:18*1,Beirao-da-Veiga.Brezzi.ea:18*2}, which also include applications to the Kikuchi formulation of magnetostatics, and in \cite{Beirao-da-Veiga.Mascotto:20}, which contains a detailed study of the interpolation and stability properties of the low-order VEM spaces.
In order to derive an actual discretisation scheme starting from the sequence of virtual spaces, a variational crime involving projections is required.
A different approach is pursued in \cite{Di-Pietro.Droniou.ea:20,Di-Pietro.Droniou:20*1}, where a discrete de Rham (DDR) complex is presented, based on decompositions of full polynomial spaces into the range of vector calculus operators and their $\Leb$-orthogonal complements.
This complex involves discrete spaces and operators that appear, through discrete $\Leb$-products, in the formulation of discretisation methods.
The analysis in \cite{Di-Pietro.Droniou.ea:20,Di-Pietro.Droniou:20*1} focuses on a subset of properties {\bf (P1)}--{\bf (P2)} involved in the stability analysis of numerical schemes: local exactness (\cite[Theorems 4.1 and 5.1]{Di-Pietro.Droniou.ea:20}), global complex property, discrete counterparts of $\Image\GRAD = \Ker\CURL$ for domains that do not enclose voids and $\Image\DIV = \Leb(\Omega)$ (\cite[Theorem 3]{Di-Pietro.Droniou:20*1}), as well as Poincar\'e inequalities for the divergence and the curl (\cite[Theorems 18 and 20, respectively]{Di-Pietro.Droniou:20*1}).
This approach completely avoids, both in the construction and in the analysis, the use of (virtual or piecewise polynomial) functions with global regularity, and is closer in spirit to Mimetic Finite Differences and Mixed Hybrid Finite Volume methods.

Regarding consistency properties \textbf{(P3)} for polytopal methods, and starting from low-order methods, results for Compatible Discrete Operator approximations of the Poisson problem based on nodal unknowns can be found in \cite{Bonelle.Ern:14}; see in particular the proof of Theorem 3.3 therein, which contains an adjoint consistency result for a gradient reconstructed from vertex values.
In the same framework, an adjoint consistency estimate for a discrete curl constructed from edge values can be found in \cite[Lemma 2.3]{Bonelle.Ern:15}.
A rather complete set of consistency results for Mimetic Finite Difference operators can be found in \cite{Beirao-da-Veiga.Lipnikov.ea:14}, where they appear as intermediate steps in the error analyses of Chapters 5--7.
A notable exception is provided by the adjoint consistency of the curl operator, which is not needed in the error estimate of \cite[Theorem 7.3]{Beirao-da-Veiga.Lipnikov.ea:14} since the authors consider an approximation of the current density based on the knowledge of a vector potential.

Moving to consistency properties for arbitrary-order polytopal methods, error estimates that involve the adjoint consistency of a gradient and the consistency of the corresponding potential have been recently derived in \cite{Brenner.Guan.ea:17} in the framework of the $\Sob{1}$-conforming Virtual Element method.
The same method is considered in \cite[Section 3.2]{Di-Pietro.Droniou:18}, where a different analysis is proposed based on the third Strang lemma. The estimate of the consistency error in \cite[Theorem 19]{Di-Pietro.Droniou:18} involves, in particular, the adjoint consistency of a discrete gradient reconstructed as the gradient of a scalar polynomial rather than a vector-valued polynomial.
We note, in passing, that the concept of adjoint consistency for (discrete) gradients is directly related to the notion of limit-conformity in the Gradient Discretisation Method \cite{Droniou.Eymard.ea:18}, a generic framework which encompasses several polytopal
methods.
Primal and dual consistency estimates for a discrete divergence and the corresponding vector potential similar (but not identical) to the ones considered here have been established in \cite{Di-Pietro.Ern:17} in the framework of Mixed High-Order methods.
Note that these methods (the $\Sob{1}$-conforming Virtual Element method and the Mixed High-Order method) do not lead to a discrete de Rham complex. In the framework of arbitrary-order compatible discretisations, on the other hand, primal consistency results for the curl appear as intermediate results in \cite{Beirao-da-Veiga.Brezzi.ea:18*2}, where an error analysis for a Virtual Element approximation of magnetostatics is carried out assuming interpolation estimates for three-dimensional vector valued virtual spaces; see Remark 4.4 therein. However, \cite{Beirao-da-Veiga.Brezzi.ea:18*2} does not establish any adjoint consistency property of the discrete curl (the formulation of magnetostatics considered in this reference does not require it).
\medskip

\paragraph{Content of the paper.}

We present a new DDR sequence based, contrary to \cite{Di-Pietro.Droniou.ea:20,Di-Pietro.Droniou:20*1}, on explicit complements of the ranges of vector calculus operators inspired by the ones used in \cite{Beirao-da-Veiga.Brezzi.ea:18*2}; these complements are easier to implement, and enable a complete proof of the full set of properties {\bf (P1)}--{\bf (P3)}.
To the best of our knowledge, this is the first time that such a complete panel of results is available for an arbitrary-order polyhedral method compatible with the de Rham complex.
The complements considered here are linked to the spaces appearing in the Koszul complex (see, e.g., \cite[Chapter 7]{Arnold:18}) and enjoy two key properties on general polyhedral meshes: they are hierarchical (see Remark \ref{rem:hierarchical.complements} below) and their traces on polyhedral faces or edges lie in appropriate polynomial spaces (cf.\ Proposition \ref{prop:traces.NE.RT}).
These properties make it possible to prove discrete integration by parts formulas for the discrete potentials (see Remarks \ref{rem:validity.trF}, \ref{rem:validity.Pgrad}, \ref{rem:ibp.PcurlT}, and \ref{rem:ibp.PdivT} below) which, in turn, are essential to the proof of the adjoint consistency properties.

The key ingredients to establish primal consistency are the polynomial consistency of discrete vector calculus operators along with the corresponding potentials, and their boundedness when applied to the interpolates of smooth functions.
The proofs of adjoint consistency, on the other hand, rely on operator-specific techniques and are all grounded in the above-mentioned discrete integration by parts formulas for the corresponding potential reconstructions.
Specifically, the key point for the adjoint consistency of the gradient are estimates for local $\Sob{1}$-like seminorms of the scalar potentials.
The adjoint consistency of curl requires, on the other hand, the construction of liftings of the discrete face potentials that satisfy an orthogonality and a boundedness condition.
These reconstructions are inspired by the minimal reconstruction operators of \cite[Chapter 3]{Beirao-da-Veiga.Lipnikov.ea:14}, with a key novelty provided by a curl correction which ensures the well-posedness of the reconstruction inside mesh elements and relies on fine results from \cite{Dauge:88,Assous.Ciarlet.ea:18}.

In order to showcase the theoretical results derived here, we carry out a full convergence analysis for a DDR approximation of magnetostatics.
This is, to the best of our knowledge, the first full theoretical result of this kind for arbitrary-order polytopal methods.

The key innovation of the DDR complex presented here, compared to the one in \cite{Di-Pietro.Droniou.ea:20,Di-Pietro.Droniou:20*1}, precisely lies in the fact that it enables all mathematical results required to prove error estimates for schemes built from this sequence. The only analytical results available in \cite{Di-Pietro.Droniou.ea:20,Di-Pietro.Droniou:20*1} are Poincar\'e inequalities for the curl and the divergence and, as a matter of fact, it seems that the sequence in these references does not satisfy the critical discrete integration by parts formulas mentioned above, and is therefore not amenable to an adjoint consistency analysis.

\medskip

The rest of the paper is organised as follows.
In Section \ref{sec:setting} we establish the general setting.
Section \ref{sec:ddr} contains the definition of the DDR sequence along with key intermediate results for the discrete vector calculus operators (including the commutation property in {\bf (P3)}) and the proof of {\bf (P1)}.
In Section \ref{sec:potentials.L2products}, we introduce tools for the design and analysis of schemes based on the DDR sequence: polynomial potential reconstructions and $\Leb$-products on the discrete spaces.
Discrete Poincar\'e inequalities corresponding to {\bf (P2)} are covered in Section \ref{sec:poincare}.
Section \ref{sec:consistency.results} contains the statement and proofs of the primal and adjoint consistency results corresponding to {\bf (P3)}.
The application of the theoretical tools to the error analysis of a DDR approximation of magnetostatics is considered in Section \ref{sec:application}, where numerical evidence supporting the error estimates is also provided.
The paper is completed by three appendices.
Appendix \ref{sec:results} contains results on local polynomial spaces including those on the traces of the trimmed spaces constructed from the Koszul complements.
Appendix \ref{appen:RcurlT} contains an in-depth and novel study of the div-curl problems defining the curl liftings on polytopal elements: well-posedness, orthogonality and boundedness properties.
Finally, Appendix \ref{appen:notations} details the conventions of notation adopted throughout the paper, and lists the main spaces and operators of the DDR complex.


\section{Setting}\label{sec:setting}

\subsection{Domain and mesh}

For any (measurable) set $Y\subset\Real^3$, we denote by $h_Y\coloneq\sup\{|\bvec{x}-\bvec{y}|\st \bvec{x},\bvec{y}\in Y\}$ its diameter and by $|Y|$ its Hausdorff measure.
We consider meshes $\Mh\coloneq\Th\cup\Fh\cup\Eh\cup\Vh$, where:
$\Th$ is a finite collection of open disjoint polyhedral elements such that $\overline{\Omega} = \bigcup_{T\in\Th}\overline{T}$ and $h=\max_{T\in\Th}h_T>0$;
$\Fh$ is a finite collection of open planar faces;
$\Eh$ is the set collecting the open polygonal edges (line segments) of the faces;
$\Vh$ is the set collecting the edge endpoints.
It is assumed, in what follows, that $(\Th,\Fh)$ matches the conditions in \cite[Definition 1.4]{Di-Pietro.Droniou:20}, so that the faces form a partition of the mesh skeleton $\bigcup_{T\in\Th}\partial T$.
We additionally assume that the polytopes in $\Th\cup\Fh$ are simply connected and have connected Lipschitz-continuous boundaries.
This notion of mesh is related to that of cellular (or CW) complex from algebraic topology; see, e.g., \cite[Chapter 7]{Spanier:94}.

The set collecting the mesh faces that lie on the boundary of a mesh element $T\in\Th$ is denoted by $\FT$.
For any mesh element or face $Y\in\Th\cup\Fh$, we denote, respectively, by $\edges{Y}$ and $\vertices{Y}$ the set of edges and vertices of $Y$.

Throughout the paper, unless otherwise specified, we write $a\lesssim b$ in place of $a\le Cb$ with $C$ depending only on $\Omega$, the mesh regularity parameter $\rho$ of \cite[Definition 1.9]{Di-Pietro.Droniou:20}, and the considered polynomial degree.
We note that this mesh regularity parameter $\rho\in (0,1)$ is bounded away from $0$ when a shape-regular matching simplicial submesh of $\Th$ exists such that each $T\in\Th$ is partitioned into simplices of size uniformly comparable to $T$.
We also use $a\simeq b$ as a shorthand for ``$a\lesssim b$ and $b\lesssim a$''.

\subsection{Orientation of mesh entities and vector calculus operators on faces}

For any face $F\in\Fh$, an orientation is set by prescribing a unit normal vector $\normal_F$ and, for any mesh element $T\in\Th$ sharing $F$, we denote by $\omega_{TF}\in\{-1,1\}$ the orientation of $F$ relative to $T$, that is, $\omega_{TF}=1$ if $\normal_F$ points out of $T$, $-1$ otherwise.
With this choice, $\omega_{TF}\normal_F$ is the unit vector normal to $F$ that points out of $T$.
For any edge $E\in\Eh$, an orientation is set by prescribing the unit tangent vector $\tangent_E$.
Denoting by $F\in\Fh$ a face such that $E\in\EF$, its boundary $\partial F$ is oriented counter-clockwise with respect to $\normal_F$, and we denote by $\omega_{FE}\in\{-1,1\}$ the (opposite of the) orientation of $E$ relative to that $\partial F$: $\omega_{FE}=1$ if $\tangent_E$ points on $E$ in the opposite orientation to $\partial F$, $\omega_{FE}=-1$ otherwise.
We also denote by $\normal_{FE}$ the unit vector normal to $E$ lying in the plane of $F$ such that $(\tangent_E,\normal_{FE})$ forms a system of right-handed coordinates in the plane of $F$, so that the system of coordinates $(\tangent_E,\normal_{FE},\normal_F)$ is right-handed in $\Real^3$.
It can be checked that $\omega_{FE}\normal_{FE}$ is the normal to $E$, in the plane where $F$ lies, pointing out of $F$.

For any mesh face $F\in\Fh$, we denote by $\GRAD_F$ and $\DIV_F$ the tangent gradient and divergence operators acting on smooth enough functions.
Moreover, for any $r:F\to\Real$ and $\bvec{z}:F\to\Real^2$ smooth enough, we define the two-dimensional vector and scalar curl operators such that
\begin{equation} \label{eq:def:VROTF:ROTF}
  \VROT_F r\coloneq \rotation{-\nicefrac\pi2}(\GRAD_F r)\quad\mbox{ and }\quad \ROT_F\bvec{z}=\DIV_F(\rotation{-\nicefrac\pi2}\bvec{z}),
\end{equation}
where $\rotation{-\nicefrac\pi2}$ is the rotation of angle $-\frac\pi2$ in the oriented tangent space to $F$.

\subsection{Lebesgue and Sobolev spaces}

Let $Y$ be a measurable subset of $\Real^3$.
We denote by $\Leb(Y)$ the Lebesgue space spanned by functions that are square-integrable over $Y$.
When $Y$ is a subset of an $n$-dimensional variety, we will use the boldface notation $\vLeb(Y)\coloneq\Leb(Y)^n$ for the space of vector-valued fields over $Y$ with square-integrable components.
Given an integer $l$ and $Y\in\{\Omega\}\cup\Th\cup\Fh$, $\Sob{l}(Y)$ will denote the Sobolev space spanned by square-integrable functions whose partial derivatives of order up to $l$ are also square-integrable.
Denoting again by $n$ the dimension of $Y$, we let $\vSob{l}(Y)\coloneq\Sob{l}(Y)^n$ and $\vC{l}(Y)\coloneq\rC{l}(Y)^n$.
For all $F\in\Fh$, we let $\Hrot{F}\coloneq\left\{\bvec{v}\in\vLeb(F)\st\ROT_F\bvec{v}\in \Leb(F)\right\}$.
Similarly, for all $Y\in\{\Omega\}\cup\Th$, $\Hcurl{Y}\coloneq\left\{\bvec{v}\in\vLeb(Y)\st\CURL\bvec{v}\in\vLeb(Y)\right\}$ and
$\Hdiv{Y}\coloneq\left\{\bvec{w}\in\vLeb(Y)\st\DIV\bvec{w}\in \Leb(Y)\right\}$.

\subsection{Polynomial spaces and decompositions}\label{sec:polynomial.spaces}

For a given integer $\ell\ge 0$, $\mathbb{P}_n^\ell$ denotes the space of $n$-variate polynomials of total degree $\le\ell$, with the convention that $\mathbb{P}_0^\ell = \Real$ for any $\ell$ and that $\mathbb{P}_n^{-1} \coloneq \{ 0 \}$ for any $n$.
For any $Y\in\Th\cup\Fh\cup\Eh$, we denote by $\Poly{\ell}(Y)$ the space spanned by the restriction to $Y$ of the functions in $\mathbb{P}_3^\ell$.
Denoting by $1\le n\le 3$ the dimension of $Y$, $\Poly{\ell}(Y)$ is isomorphic to $\mathbb{P}_n^\ell$ (see \cite[Proposition 1.23]{Di-Pietro.Droniou:20}).
In what follows, with a little abuse of notation, both spaces are denoted by $\Poly{\ell}(Y)$.
We additionally denote by $\lproj{\ell}{Y}$ the corresponding $\Leb$-orthogonal projector and let $\Poly{0,\ell}(Y)$ denote the subspace of $\Poly{\ell}(Y)$ made of polynomials with zero average over $Y$.
For the sake of brevity, we also introduce the boldface notations $\vPoly{\ell}(T)\coloneq\Poly{\ell}(T)^3$ for all $T\in\Th$ and $\vPoly{\ell}(F)\coloneq\Poly{\ell}(F)^2$ for all $F\in\Fh$.

Let again an integer $\ell\ge 1$ be given, and denote by $\mathfrak{E}\subset\Eh$ a collection of edges such that $S_{\mathfrak{E}}\coloneq\bigcup_{E\in\mathfrak{E}}\overline{E}$ forms a connected set.
We denote by $\Poly[\rm c]{\ell}(\mathfrak{E})\coloneq\left\{ q_{\mathfrak{E}}\in \rC{0}(S_{\mathfrak{E}})\st\text{$(q_{\mathfrak{E}})_{|E}\in\Poly{\ell}(E)$ for all $E\in\mathfrak{E}$}\right\}$ the space of functions over $S_\mathfrak{E}$ whose restriction to each edge $E\in\mathfrak{E}$ is a polynomial of total degree $\le\ell$ and that are continuous at the edges endpoints; these endpoints are collected in the set $\vertices{\mathfrak{E}}\subset\Vh$.
Denoting by $\bvec{x}_V$ the coordinates vector of a vertex $V\in\Vh$, it can be easily checked that the following mapping is an isomorphism:
\begin{equation}\label{eq:Poly.c.ell:isomorphism}
  \Poly[\rm c]{\ell}(\mathfrak{E})\ni q_{\mathfrak{E}}
  \mapsto
  \big( (\lproj{\ell-2}{E} (q_{\mathfrak{E}})_{|E})_{E\in\mathfrak{E}}, (q_{\mathfrak{E}}(\bvec{x}_V))_{V\in\vertices{\mathfrak{E}}}\big)
  \in\left(
  \bigtimes_{E\in\mathfrak{E}}\Poly{\ell-2}(E)
  \right)\times\Real^{\vertices{\mathfrak{E}}}.
\end{equation}

For all $Y\in\Th\cup\Fh$, denote by $\bvec{x}_Y$ a point inside $Y$ such that $Y$ contains a ball centered at $\bvec{x}_Y$ of radius $\rho h_Y$, where $\rho$ is the mesh regularity parameter in \cite[Definition 1.9]{Di-Pietro.Droniou:20}.
For any mesh face $F\in\Fh$ and any integer $\ell\ge 0$, we define the following relevant subspaces of $\vPoly{\ell}(F)$:
\begin{subequations}\label{eq:spaces.F}
  \begin{alignat}{2}\label{eq:Goly.cGoly.F}
    \Goly{\ell}(F)&\coloneq\GRAD_F\Poly{\ell+1}(F),
    &\qquad
    \cGoly{\ell}(F)&\coloneq(\bvec{x}-\bvec{x}_F)^\perp\Poly{\ell-1}(F),
    \\ \label{eq:Roly.cRoly.F}
    \Roly{\ell}(F)&\coloneq\VROT_F\Poly{\ell+1}(F),
    &\qquad
    \cRoly{\ell}(F)&\coloneq(\bvec{x}-\bvec{x}_F)\Poly{\ell-1}(F),
  \end{alignat}
\end{subequations}
(where $\bvec{y}^\perp$ is a shorthand for the rotated vector $\varrho_{-\pi/2}\bvec{y}$) so that
\begin{equation}\label{eq:decomposition:Poly.ell.F}
  \vPoly{\ell}(F)
  = \Goly{\ell}(F) \oplus \cGoly{\ell}(F)
  = \Roly{\ell}(F) \oplus \cRoly{\ell}(F).
\end{equation}
These decompositions of $\vPoly{\ell}(F)$ (as well as those of $\vPoly{\ell}(T)$ in \eqref{eq:decomposition:Poly.ell.T} below) result from \cite[Corollary 7.4]{Arnold:18}.
Notice that the direct sums in the above expression are not $\Leb$-orthogonal in general.
The $\Leb$-orthogonal projectors on the spaces \eqref{eq:spaces.F} are, with obvious notation, $\Gproj{\ell}{F}$, $\Gcproj{\ell}{F}$, $\Rproj{\ell}{F}$, and $\Rcproj{\ell}{F}$.
Similarly, for any mesh element $T\in\Th$ and any integer $\ell\ge 0$ we introduce the following subspaces of $\vPoly{\ell}(T)$:
\begin{subequations}\label{eq:spaces.T}
  \begin{alignat}{2} \label{eq:Goly.cGoly}
    \Goly{\ell}(T)&\coloneq\GRAD\Poly{\ell+1}(T),
    &\qquad 
    \cGoly{\ell}(T)&\coloneq(\bvec{x}-\bvec{x}_T)\times \vPoly{\ell-1}(T),
    \\ \label{eq:Roly.cRoly}
    \Roly{\ell}(T)&\coloneq\CURL\vPoly{\ell+1}(T),
    &\qquad
    \cRoly{\ell}(T)&\coloneq(\bvec{x}-\bvec{x}_T)\Poly{\ell-1}(T),
  \end{alignat}
\end{subequations}
so that
\begin{equation}\label{eq:decomposition:Poly.ell.T}
  \vPoly{\ell}(T)
  = \Goly{\ell}(T) \oplus \cGoly{\ell}(T)
  = \Roly{\ell}(T) \oplus \cRoly{\ell}(T).
\end{equation}
Also in this case, the direct sums above are not $\Leb$-orthogonal in general.
The $\Leb$-orthogonal projectors on the spaces \eqref{eq:spaces.T} are $\Gproj{\ell}{T}$, $\Gcproj{\ell}{T}$, $\Rproj{\ell}{T}$, and $\Rcproj{\ell}{T}$.

\begin{remark}[Hierarchical complements]\label{rem:hierarchical.complements}
  Unlike the $\Leb$-orthogonal complements considered in \cite{Di-Pietro.Droniou.ea:20}, the Koszul complements in \eqref{eq:decomposition:Poly.ell.F} and \eqref{eq:decomposition:Poly.ell.T} satisfy, for all $Y\in\Th\cup\Fh$ and all $\ell\ge 1$,
  \begin{equation}\label{eq:hierarchical.complements}
    \text{
      $\cGoly{\ell-1}(Y)\subset\cGoly{\ell}(Y)$\quad and \quad
      $\cRoly{\ell-1}(Y)\subset\cRoly{\ell}(Y)$.
    }
  \end{equation}
\end{remark}

\begin{remark}[Vector calculus isomorphisms on local polynomial spaces]
  For any polygon $F$, polyhedron $T$, and polynomial degree $\ell\ge 0$, a consequence of the polynomial exactness \cite[Corollary 7.3]{Arnold:18} is that the following mappings are isomorphisms: 
  \begin{alignat}{2}
    \VROT_F&:\Poly{0,\ell}(F)\xrightarrow{\cong}\Roly{\ell-1}(F)
    \label{eq:iso:VROTF.GRAD}\\
    \DIV_F &:\cRoly{\ell}(F)\xrightarrow{\cong} \Poly{\ell-1}(F)\,,\qquad&\DIV&:\cRoly{\ell}(T)\xrightarrow{\cong}\Poly{\ell-1}(T),
    \label{eq:iso:DIV}\\
    &&\CURL &:\cGoly{\ell}(T)\xrightarrow{\cong}\Roly{\ell-1}(T).
    \label{eq:iso:CURL}
  \end{alignat}
  An estimate of the norms of the inverses of these differential isomorphisms is provided in Lemma \ref{lem:norm.isomorphisms} in Appendix \ref{sec:results}.
\end{remark}

\begin{remark}[Composition of $\Leb$-orthogonal projectors]
  Let $\cvec{X}\in\{\cvec{G},\cvec{R}\}$, $\ell\ge-1$, and $Y\in\Th\cup\Fh$.
  Using the definition of the $\Leb$-orthogonal projectors, and denoting by $\vlproj{\ell}{Y}$ the $\Leb$-orthogonal projector on $\vPoly{\ell}(Y)$, it holds
  \begin{equation}\label{eq:projector.composition}
    \text{
      $\Xproj{\ell}{Y} = \Xproj{\ell}{Y}\circ\vlproj{\ell}{Y}$\quad and\quad
      $\Xcproj{\ell}{Y} = \Xcproj{\ell}{Y}\circ\vlproj{\ell}{Y}$.
    }
  \end{equation}
\end{remark}

In what follows, we will need the local N\'ed\'elec and Raviart--Thomas spaces: For $Y\in\Th\cup\Fh$,
\begin{equation}\label{eq:NE.RT}
  \NE{\ell}(Y)\coloneq\Goly{\ell-1}(Y) \oplus \cGoly{\ell}(Y),\qquad
  \RT{\ell}(Y)\coloneq\Roly{\ell-1}(Y) \oplus \cRoly{\ell}(Y).
\end{equation}
These spaces sit between $\vPoly{\ell-1}(Y)$ and $\vPoly{\ell}(Y)$ and are therefore referred to as \emph{trimmed} in the FE literature.
Notice that we have selected the index in \eqref{eq:NE.RT} so as to reflect the maximum polynomial degrees of functions in each space and, as a result, it is shifted by $+1$ with respect to \cite{Di-Pietro.Droniou.ea:20,Di-Pietro.Droniou:20*1}.

\subsection{Recovery operator}

As mentioned above, the direct sums in \eqref{eq:decomposition:Poly.ell.F} and \eqref{eq:decomposition:Poly.ell.T} are not $\Leb$-orthogonal. The following lemma however shows that, for any of these decompositions, a given polynomial can be recovered from its orthogonal projections on each space in the sum.

\begin{lemma}[Recovery operator]\label{lem:recovery.operator}
Let $E$ be a Euclidean space, $S$ be a subspace of $E$, and $S^\compl$ be a complement (not necessarily orthogonal) of $S$ in $E$. Let $\pi_S$ and $\pi_S^\compl$ be, respectively, the orthogonal projections on $S$ and $S^\compl$.
Then, the mappings ${\rm Id}-\pi_S\pi_S^\compl:E\to E$ and ${\rm Id}-\pi_S^\compl\pi_S:E\to E$ are isomorphisms.

We can therefore define the \emph{recovery operator} $\rec{S,S^\compl}{\cdot}{\cdot}:S\times S^\compl\to E$ such that
\begin{multline}\label{eq:recovery.operator}
  \rec{S,S^\compl}{\bvec{b}}{\bvec{c}}
  \coloneq
  ({\rm Id}-\pi_S\pi_S^\compl)^{-1}(\bvec{b}-\pi_S\bvec{c})
  +({\rm Id}-\pi_S^\compl\pi_S)^{-1}(\bvec{c}-\pi_S^\compl\bvec{b})
  \\
  \forall(\bvec{b},\bvec{c})\in S\times S^\compl.
\end{multline}
This operator satisfies the following properties:
\begin{equation}\label{eq:recovery.proj}
  \pi_S \big( \rec{S,S^\compl}{\bvec{b}}{\bvec{c}}\big) = \bvec{b}\quad\mbox{ and }\quad
  \pi_S^\compl \big( \rec{S,S^\compl}{\bvec{b}}{\bvec{c}}\big) = \bvec{c}\qquad\forall (\bvec{b},\bvec{c})\in S\times S^\compl,
\end{equation}
  \begin{equation}\label{eq:recovery.rec}
    \bvec{a}=\rec{S,S^\compl}{\pi_S\bvec{a}}{\pi_S^\compl\bvec{a}}\qquad\forall \bvec{a}\in E.
  \end{equation}
\end{lemma}

\begin{proof}
Let us denote by $\norm{{\cdot}}$ the norm in $E$. To prove that ${\rm Id}-\pi_S\pi_S^\compl$ is invertible, we show that the mapping $\pi_S\pi_S^\compl$ has a norm $<1$, which implies
\begin{equation}\label{eq:neumann.series}
({\rm Id}-\pi_S\pi_S^\compl)^{-1}=\sum_{n\ge 0}(\pi_S\pi_S^\compl)^n.
\end{equation}
The space $E$ being finite dimensional, it suffices to see that, for any $\bvec{x}\in E$ with $\norm{\bvec{x}}=1$, we have $\norm{\pi_S(\pi_S^\compl\bvec{x})}<1$.
Since $\pi_S$ is an orthogonal projector, by Pythagoras' theorem we have $\norm{\pi_S(\pi_S^\compl\bvec{x})}\le \norm{\pi_S^\compl\bvec{x}}$, with equality only if $\pi_S^\compl\bvec{x}\in S$, that is, only
if $\pi_S^\compl\bvec{x}=\bvec{0}$ since $\pi_S^\compl\bvec{x}\in S^\compl$. In this case, $\norm{\pi_S(\pi_S^\compl\bvec{x})}=0<1$.
Otherwise, $\norm{\pi_S(\pi_S^\compl\bvec{x})}<\norm{\pi_S^\compl\bvec{x}}\le\norm{\bvec{x}}=1$, where the second inequality is a consequence of the fact that $\pi_S^\compl$ is an orthogonal projection.
This concludes the proof that ${\rm Id}-\pi_S\pi_S^\compl$ is an isomorphism. The invertibility of ${\rm Id}-\pi_S^\compl\pi_S$ is obtained similarly, exchanging the roles of $S$ and $S^\compl$.
\smallskip

Let us prove the first relation in \eqref{eq:recovery.proj}. The second follows using the same arguments. We expand $({\rm Id}-\pi_S\pi_S^\compl)^{-1}$ in \eqref{eq:recovery.operator} using the series \eqref{eq:neumann.series} (and similarly for $({\rm Id}-\pi_S^\compl\pi_S)^{-1}$) to write
  \[
  \begin{aligned}
    \pi_S\big(\rec{S,S^\compl}{\bvec{b}}{\bvec{c}}\big)
    &=\pi_S\sum_{n\ge 0}(\pi_S\pi_S^\compl)^n(\bvec{b}-\pi_S\bvec{c})
    +\pi_S\sum_{n\ge 0}(\pi_S^\compl\pi_S)^n(\bvec{c}-\pi_S^\compl\bvec{b})
    \\
    &=\left[\pi_S\sum_{n\ge 0}(\pi_S\pi_S^\compl)^n-\pi_S\sum_{n\ge 0}(\pi_S^\compl\pi_S)^n\pi_S^\compl\right]\bvec{b}
    \\
    &\quad 
    +\left[\pi_S\sum_{n\ge 0}(\pi_S^\compl\pi_S)^n-\pi_S\sum_{n\ge 0}(\pi_S\pi_S^\compl)^n\pi_S\right]\bvec{c}.
  \end{aligned}
  \]
We have $\pi_S\sum_{n\ge 0}(\pi_S^\compl\pi_S)^n\pi_S^\compl=\pi_S\sum_{n\ge 1}(\pi_S\pi_S^\compl)^n$ (we have used $\pi_S\pi_S=\pi_S$ to introduce the pre-factor $\pi_S$) and the operator acting on $\bvec{b}$ above therefore reduces to $\pi_S$, and returns $\bvec{b}$ since $\bvec{b}\in S$. As for the operator acting on $\bvec{c}$, using again $\pi_S\pi_S=\pi_S$ shows that it is equal to 0. This concludes the proof of the first relation in \eqref{eq:recovery.proj}. 
\smallskip

Fix now $\bvec{a}\in E$ and set $\bvec{z}\coloneq\bvec{a}-\rec{S,S^\compl}{\pi_S\bvec{a}}{\pi_S^\compl\bvec{a}}$. Applying \eqref{eq:recovery.proj} to $\bvec{b}=\pi_S\bvec{a}$ and $\bvec{c}=\pi_S^\compl\bvec{a}$ shows that $\pi_S\bvec{z}=\pi_S^\compl\bvec{z}=\bvec{0}$. Since $E=S\oplus S^\compl$, we can write $\bvec{z}=\bvec{z}_S+\bvec{z}_S^c$ with $\bvec{z}_S\in S$
and $\bvec{z}_S^\compl\in S^\compl$, and the definition of the orthogonal projectors on $S$ and $S^\compl$ therefore yields, with $(\cdot,\cdot)_E$ the scalar product on $E$,
\[
\norm{\bvec{z}}^2=(\bvec{z},\bvec{z})_E=(\bvec{z},\bvec{z}_S)_E+(\bvec{z},\bvec{z}_S^\compl)_E=
(\pi_S\bvec{z},\bvec{z}_S)_E+(\pi_S^\compl\bvec{z},\bvec{z}_S^\compl)_E=0.
\]
Hence, $\bvec{z}=\bvec{0}$ and \eqref{eq:recovery.rec} is established.\qed
\end{proof}

The following lemma shows that the norm of the recovery operator for the decompositions \eqref{eq:decomposition:Poly.ell.F} and \eqref{eq:decomposition:Poly.ell.T} is equivalent to the sum of the norms of its arguments, uniformly in $h$.
In other words, it states that the decompositions are not just algebraic but also topological (uniformly in $h$).
Since the recovery operator will mostly be of interest to us for these pairs of spaces, to alleviate the notations from here on we will write
\begin{equation}\label{eq:def.Xrec}
  \Xrec{X}{Y}{\ell}{\cdot}{\cdot}\coloneq\rec{\Xoly{\ell}(Y),\cXoly{\ell}(Y)}{\cdot}{\cdot}\quad\forall\cvec{X}\in\{\cvec{R},\cvec{G}\}\,,\;\forall Y\in\Th\cup\Fh.
\end{equation}

\begin{lemma}[Estimate on the norm of the recovery operator]\label{lem:piS.piSc}
  For all $\ell\ge 0$, there exists $\alpha<1$ depending only on the mesh regularity parameter in \cite[Definition 1.9]{Di-Pietro.Droniou:20} such that, for all $\cvec{X}\in\{\cvec{R},\cvec{G}\}$ and all $Y\in\Th\cup\Fh$,
  \begin{equation}\label{eq:estimates.projections}
    \norm[Y]{\Xproj{\ell}{Y}\Xcproj{\ell}{Y}}\le \alpha\quad\mbox{ and }\quad\norm[Y]{\Xcproj{\ell}{Y}\Xproj{\ell}{Y}}\le \alpha,
  \end{equation}
  where $\norm[Y]{{\cdot}}$ denotes the norm induced by $\norm[\vLeb(Y)]{{\cdot}}$ on the space of endomorphisms of $\vPoly{\ell}(Y)$.
  As a result,
  \begin{equation}\label{eq:estimate.norm.rec}
    \norm[\vLeb(Y)]{\Xrec{X}{Y}{\ell}{\bvec{v}}{\bvec{w}}}
    \simeq \norm[\vLeb(Y)]{\bvec{v}} +  \norm[\vLeb(Y)]{\bvec{w}}
    \qquad\forall(\bvec{v},\bvec{w})\in \Xoly{\ell}(Y)\times\cXoly{\ell}(Y).
  \end{equation}
\end{lemma}

\begin{remark}[Recovery operator and $\Leb$-orthogonal complements]
  When working with $\Leb$-orthogonal complements to $\Goly{\ell}(Y)$ and $\Roly{\ell}(Y)$, instead of the Koszul complements in \eqref{eq:spaces.F} and \eqref{eq:spaces.T}, the recovery operator is trivial since it consists in the sum of its two arguments (its topological property \eqref{eq:estimate.norm.rec} is also obvious).
  As mentioned in the introduction, however, the Koszul complements enable proofs of commutation and consistency properties that do not seem straightforward with orthogonal complements; the trade-off lies in having to deal with a less trivial recovery operator (although it remains a purely theoretical tool, see Remark \ref{rem:validity:trFt:Roly.k}), whose topological properties are more complex to establish.
\end{remark}

\begin{proof}
  \underline{1. \emph{Proof of \eqref{eq:estimates.projections}.}}
  We estimate $\norm[T]{\Gproj{\ell}{T}\Gcproj{\ell}{T}}$ for an element $T\in\Th$, the other cases being identical.
  The linear mapping $\Real^3\ni\bvec{x}\mapsto h_T^{-1}(\bvec{x}-\bvec{x}_T)\in\Real^3$ maps $T$ onto a polyhedron $\widehat{T}$ of diameter 1, transports the spaces $\vPoly{\ell}(T)$, $\Goly{\ell}(T)$ and $\cGoly{\ell}(T)$ on their equivalent over $\widehat{T}$, and simply scales the $\Leb$-norm of functions.
  As a consequence, $\norm[T]{\Gproj{\ell}{T}\Gcproj{\ell}{T}}=\norm[\widehat{T}]{\Gproj{\ell}{\widehat{T}}\Gcproj{\ell}{\widehat{T}}}$, and we only have to estimate the latter quantity.
  
  Assume that we establish the existence of $\alpha <1$, depending only on the mesh regularity parameter, such that
  \begin{equation}\label{eq:est.proj.todo}
    \int_{\widehat{T}}\bvec{v}\cdot\Gproj{\ell}{\widehat{T}}\bvec{v}
    \le \alpha^2 \norm[\vLeb(\widehat{T})]{\bvec{v}}\norm[\vLeb(\widehat{T})]{\Gproj{\ell}{\widehat{T}}\bvec{v}}
    \quad\forall \bvec{v}\in \cGoly{\ell}(\widehat{T})=\bvec{x}\times\vPoly{\ell-1}(\widehat{T}).
  \end{equation}
  Notice that, with the selected mapping, $\bvec{x}_T$ is mapped onto $\bvec{0}\in\widehat{T}$.
  Then, for all $\bvec{w}\in \vPoly{\ell}(T)$,
  \begin{align}
    \int_{\widehat{T}}\Gproj{\ell}{\widehat{T}}(\Gcproj{\ell}{\widehat{T}}\bvec{w})\cdot\Gproj{\ell}{\widehat{T}}(\Gcproj{\ell}{\widehat{T}}\bvec{w})
    ={}&\int_{\widehat{T}}\Gcproj{\ell}{\widehat{T}}\bvec{w}\cdot\Gproj{\ell}{\widehat{T}}(\Gcproj{\ell}{\widehat{T}}\bvec{w})
    \nonumber\\
    \le{}& \alpha^2\norm[\vLeb(\widehat{T})]{\Gcproj{\ell}{\widehat{T}}\bvec{w}}\norm[\vLeb(\widehat{T})]{\Gproj{\ell}{\widehat{T}}\Gcproj{\ell}{\widehat{T}}\bvec{w}}
    \nonumber\\
    \le{}& \alpha^2\norm[\vLeb(\widehat{T})]{\bvec{w}}^2,
  \label{eq:est.proj.done}
  \end{align}
  where the equality comes from the definition of $\Gproj{\ell}{\widehat{T}}$,
  the first inequality is obtained applying \eqref{eq:est.proj.todo} to $\bvec{v}=\Gcproj{\ell}{\widehat{T}}\bvec{w}$, and the conclusion is obtained using the fact that $\Gproj{\ell}{\widehat{T}}$ and $\Gcproj{\ell}{\widehat{T}}$ are both $\vLeb(\widehat{T})$-orthogonal projectors and have thus norm $1$. The bound \eqref{eq:est.proj.done} shows that $\norm[\widehat{T}]{\Gproj{\ell}{\widehat{T}}\Gcproj{\ell}{\widehat{T}}}\le\alpha$ and concludes the proof.

  We therefore only have to establish \eqref{eq:est.proj.todo}. Note that, in the rest of the proof, polynomials are indifferently considered over $\Real^3$ or some of its open subsets.
  We also remark that, by choice of $\bvec{x}_T$ in $T$ and of the mapping $T\mapsto\widehat{T}$, we have $B(\rho)\subset \widehat{T}\subset B(1)$, where $B(r)$ is the ball in $\Real^d$ centered at $\bvec{0}$ and of radius $r$. The proof of \eqref{eq:est.proj.todo} is done by contradiction: if this relation does not hold, there exists a sequence $(\widehat{T}_n)_{n\in\Natural}$ of open sets between $B(\rho)$ and $B(1)$, a sequence $(\alpha_n)_{n\in\Natural}$ converging to $1$, and a sequence $(\bvec{v}_n)_{n\in\Natural}$ in $\bvec{x}\times\Poly{\ell-1}(\Real^3)$ such that
  \begin{equation}\label{eq:est.proj.contradiction}
    \int_{\widehat{T}_n}\bvec{v}_n\cdot\Gproj{\ell}{\widehat{T}_n}\bvec{v}_n> \alpha_n^2 \norm[\vLeb(\widehat{T}_n)]{\bvec{v}_n}\norm[\vLeb(\widehat{T}_n)]{\Gproj{\ell}{\widehat{T}_n}\bvec{v}_n}.
  \end{equation}
  Upon replacing $\bvec{v}_n$ by $\bvec{v}_n/\norm[\vLeb(\widehat{T}_n)]{\bvec{v}_n}$, we can assume that $\norm[\vLeb(\widehat{T}_n)]{\bvec{v}_n}=1$. Since $B(\rho)\subset\widehat{T}_n$, we infer that $\norm[\vLeb(B(\rho))]{\bvec{v}_n}\le\norm[\vLeb(\widehat{T}_n)]{\bvec{v}_n}=1$; hence, $(\bvec{v}_n)_{n\in\Natural}$ is bounded for the $\vLeb(B(\rho))$-norm in the finite-dimensional space $\bvec{x}\times\Poly{\ell-1}(\Real^3)$, and converges up to a subsequence to some $\bvec{v}\in \bvec{x}\times\Poly{\ell-1}(\Real^3)$. Likewise, we can assume that $\Gproj{\ell}{\widehat{T}_n}\bvec{v}_n\to \bvec{w}$ in $\Goly{\ell}(\Real^3)$. The characteristic function $\mathbf{1}_{\widehat{T}_n}$ satisfies $\mathbf{1}_{B(\rho)}\le \mathbf{1}_{\widehat{T}_n}\le \mathbf{1}_{B(1)}$ and converges therefore, up to a subsequence, in $L^\infty(B(1))$ weak-$\star$ towards some function $\theta$ satisfying $\mathbf{1}_{B(\rho)}\le \theta\le \mathbf{1}_{B(1)}$. Noting that
  \[
  \begin{gathered}
    \int_{\widehat{T}_n}\bvec{v}_n\cdot\Gproj{\ell}{\widehat{T}_n}\bvec{v}_n=\int_{B(1)}\mathbf{1}_{\widehat{T}_n}\bvec{v}_n\cdot\Gproj{\ell}{\widehat{T}_n}\bvec{v}_n\,,
    \\
    \norm[\vLeb(\widehat{T}_n)]{\bvec{v}_n}^2=\int_{B(1)}\mathbf{1}_{\widehat{T}_n}|\bvec{v}_n|^2,
    \quad\mbox{ and }\quad
    \norm[\vLeb(\widehat{T}_n)]{\Gproj{\ell}{\widehat{T}_n}\bvec{v}_n}^2=\int_{B(1)}\mathbf{1}_{\widehat{T}_n}|\Gproj{\ell}{\widehat{T}_n}\bvec{v}_n|^2,
  \end{gathered}
  \]
  the aforementioned convergences enable us to take the limit $n\to\infty$ of \eqref{eq:est.proj.contradiction} and find
  \begin{equation}\label{eq:piS.piSc:intermediate}
    \int_{B(1)}\theta \bvec{v}\cdot\bvec{w}\ge \norm[\vLeb(B(1))]{\sqrt{\theta} \bvec{v}}\norm[\vLeb(B(1))]{\sqrt{\theta} \bvec{w}}.
  \end{equation}
  The Cauchy--Schwarz inequality, on the other hand, gives
  \[
   \int_{B(1)}\theta \bvec{v}\cdot\bvec{w}= \int_{B(1)}\sqrt{\theta} \bvec{v}\cdot\sqrt{\theta}\bvec{w}\le\norm[\vLeb(B(1))]{\sqrt{\theta} \bvec{v}}\norm[\vLeb(B(1))]{\sqrt{\theta} \bvec{w}},
  \]
  which, combined with \eqref{eq:piS.piSc:intermediate}, shows that,
  \[
   \int_{B(1)}\theta \bvec{v}\cdot\bvec{w}=\norm[\vLeb(B(1))]{\sqrt{\theta} \bvec{v}}\norm[\vLeb(B(1))]{\sqrt{\theta} \bvec{w}}.
  \]
  Hence, $\sqrt{\theta} \bvec{v}$ and $\sqrt{\theta} \bvec{w}$ are co-linear. Restricted to $B(\rho)$, over which $\theta=1$, this proves that $\bvec{v}$ and $\bvec{w}$ are co-linear. Since $\bvec{v}\in \cGoly{\ell}(B(\rho))$ and $\bvec{w}\in\Goly{\ell}(B(\rho))$, we infer that $\bvec{v}=\bvec{w}=\bvec{0}$ on $B(\rho)$, and thus on $\Real^3$. This leads to $0=\norm[\vLeb(B(1))]{\sqrt{\theta} \bvec{v}}=\lim_{n\to\infty}\norm[\vLeb(\widehat{T}_n)]{\bvec{v}_n}=1$, which yields the sought contradiction.
  \medskip\\
  \underline{2. \emph{Proof of \eqref{eq:estimate.norm.rec}.}}
  By \eqref{eq:recovery.operator}, recalling the abridged notation \eqref{eq:def.Xrec}, we have
  \begin{multline*}
    \norm[\vLeb(Y)]{\Xrec{X}{Y}{\ell}{\bvec{v}}{\bvec{w}}}
    \le\norm[Y]{({\rm Id}-\Xproj{\ell}{Y}\Xcproj{\ell}{Y})^{-1}}\left(\norm[\vLeb(Y)]{\bvec{v}} + \norm[\vLeb(Y)]{\Xproj{\ell}{Y}\bvec{w}}\right)
    \\
    + \norm[Y]{({\rm Id}-\Xcproj{\ell}{Y}\Xproj{\ell}{Y})^{-1}}\left(\norm[\vLeb(Y)]{\Xcproj{\ell}{Y}\bvec{v}} + \norm[\vLeb(Y)]{\bvec{w}}\right).
  \end{multline*}
  The expansion \eqref{eq:neumann.series} and the estimates \eqref{eq:estimates.projections} show that 
  \[
  {\norm[Y]{({\rm Id}-\Xproj{\ell}{Y}\Xcproj{\ell}{Y})^{-1}}}\le\sum_{n\ge 0}\norm[Y]{\Xproj{\ell}{Y}\Xcproj{\ell}{Y}}^n\le
  \sum_{n\ge 0}\alpha^n=\frac{1}{1-\alpha}
  \]
  and, similarly, ${\norm[Y]{({\rm Id}-\Xcproj{\ell}{Y}\Xproj{\ell}{Y})^{-1}}}\le\frac{1}{1-\alpha}$. Since $\norm[\vLeb(Y)]{\Xproj{\ell}{Y}\bvec{w}}\le\norm[\vLeb(Y)]{\bvec{w}}$ and $\norm[\vLeb(Y)]{\Xcproj{\ell}{Y}\bvec{v}}\le\norm[\vLeb(Y)]{\bvec{v}}$ as both $\Xproj{\ell}{Y}$ and $\Xcproj{\ell}{Y}$ are $\Leb$-orthogonal projectors, we conclude that
  \[
  \norm[\vLeb(Y)]{\Xrec{X}{Y}{\ell}{\bvec{v}}{\bvec{w}}}
  \le\frac{2}{1-\alpha}\left(\norm[\vLeb(Y)]{\bvec{v}} + \norm[\vLeb(Y)]{\bvec{w}}\right).
  \]

  To prove the converse inequality, we use \eqref{eq:recovery.proj} along with the $\Leb$-bounded\-ness of $\Xproj{\ell}{Y}$ and $\Xcproj{\ell}{Y}$ to write
  \[
  \begin{aligned}
    \norm[\vLeb(Y)]{\bvec{v}}
    + \norm[\vLeb(Y)]{\bvec{w}}
    &= \norm[\vLeb(Y)]{\Xproj{\ell}{Y}\Xrec{X}{Y}{\ell}{\bvec{v}}{\bvec{w}}}    
    + \norm[\vLeb(Y)]{\Xcproj{\ell}{Y}\Xrec{X}{Y}{\ell}{\bvec{v}}{\bvec{w}}}
    \\
    &\le 2\norm[\vLeb(Y)]{\Xrec{X}{Y}{\ell}{\bvec{v}}{\bvec{w}}}.
  \end{aligned}
  \]
  This concludes the proof of the norm equivalence \eqref{eq:estimate.norm.rec}.\qed
\end{proof}

\section{Discrete de Rham complex}\label{sec:ddr}

We define a discrete counterpart of the de Rham complex \eqref{eq:continuous.de.rham}.
Throughout the rest of this section, we fix an integer $k\ge 0$ corresponding to the polynomial degree of the discrete sequence.
The rules used in the notations are detailed in Appendix \ref{appen:notations}, and the main DDR-related notations are summarised in Table \ref{tab:notations}.

\subsection{Discrete spaces}

The DDR spaces are spanned by vectors of polynomials whose components, each attached to a mesh entity, are selected in order to:
\begin{enumerate}[1)]
\item enable the reconstruction of consistent local discrete vector calculus operators and (scalar or vector) potentials in full polynomial spaces of total degree $\le k$ (or $\le k+1$ for the potentials associated with the gradient);
\item give rise to exact local sequences on mesh elements and faces.
\end{enumerate}
Specifically, the discrete counterparts of $\Sob{1}(\Omega)$, $\Hcurl{\Omega}$, $\Hdiv{\Omega}$ and $\Leb(\Omega)$ are respectively defined as follows:
\begin{equation}\label{eq:Xgrad.h}
\Xgrad{h}\coloneq\Big\{
\begin{aligned}[t]
  \underline{q}_h&=\big((q_T)_{T\in\Th},(q_F)_{F\in\Fh},q_{\Eh}\big)\st
  \\
  &\text{$q_T\in \Poly{k-1}(T)$ for all $T\in\Th$,}
  \\
  &\text{$q_F\in\Poly{k-1}(F)$ for all $F\in\Fh$,}
  \\
  &\text{and $q_{\Eh}\in\Poly[\rm c]{k+1}(\Eh)$}\Big\},
\end{aligned}
\end{equation}
\begin{equation}\label{eq:Xcurl.h}
  \Xcurl{h}\coloneq\Big\{
  \begin{aligned}[t]
    \uvec{v}_h
    &=\big(
    (\bvec{v}_{\cvec{R},T},\bvec{v}_{\cvec{R},T}^\compl)_{T\in\Th},(\bvec{v}_{\cvec{R},F},\bvec{v}_{\cvec{R},F}^\compl)_{F\in\Fh}, (v_E)_{E\in\Eh}
    \big)\st
    \\
    &\qquad\text{$\bvec{v}_{\cvec{R},T}\in\Roly{k-1}(T)$ and $\bvec{v}_{\cvec{R},T}^\compl\in\cRoly{k}(T)$ for all $T\in\Th$,}
    \\
    &\qquad\text{$\bvec{v}_{\cvec{R},F}\in\Roly{k-1}(F)$ and $\bvec{v}_{\cvec{R},F}^\compl\in\cRoly{k}(F)$ for all $F\in\Fh$,}
    \\
    &\qquad\text{and $v_E\in\Poly{k}(E)$ for all $E\in\Eh$}\Big\},
  \end{aligned}
\end{equation}
\begin{equation}\label{eq:Xdiv.h}
  \Xdiv{h}\coloneq\Big\{
  \begin{aligned}[t]
    \uvec{w}_h
    &=\big((\bvec{w}_{\cvec{G},T},\bvec{w}_{\cvec{G},T}^\compl)_{T\in\Th}, (w_F)_{F\in\Fh}\big)\st
    \\
    &\qquad\text{$\bvec{w}_{\cvec{G},T}\in\Goly{k-1}(T)$ and $\bvec{w}_{\cvec{G},T}^\compl\in\cGoly{k}(T)$ for all $T\in\Th$,}
    \\
    &\qquad\text{and $w_F\in\Poly{k}(F)$ for all $F\in\Fh$}
    \Big\},
  \end{aligned}
\end{equation}
and
\[
\Poly{k}(\Th)\coloneq\left\{
q_h\in \Leb(\Omega)\st\text{$(q_h)_{|T}\in\Poly{k}(T)$ for all $T\in\Th$}
\right\}.
\]

\begin{remark}[{Component of $\Xgrad{h}$ on the mesh edge skeleton}]\label{rem:Poly.c.k+1:isomorphism}
  By the isomorphism \eqref{eq:Poly.c.ell:isomorphism} with $\mathfrak{E}=\Eh$, we can replace the space $\Poly[\rm c]{k+1}(\Eh)$ in the definition of $\Xgrad{h}$ by the Cartesian product space $\left(\bigtimes_{E\in\Eh}\Poly{k-1}(E)\right)\times\Real^{\Vh}$.
  This product space is easier to manipulate in practical implementations of the DDR complex.
\end{remark}

\begin{remark}[{Components of $\Xcurl{h}$ and $\Xdiv{h}$}]
  For each mesh element or face $Y\in\Th\cup\Fh$, the pair of components $(\bvec{v}_{\cvec{R},Y},\bvec{v}_{\cvec{R},Y}^\compl)$ of a vector in $\Xcurl{h}$ defines an element in $\RT{k}(Y)$.
  Similarly, for any $T\in\Th$, each pair of element components $(\bvec{w}_{\cvec{G},T},\bvec{w}_{\cvec{G},T}^\compl)$ of a vector in $\Xdiv{h}$ defines an element in $\NE{k}(T)$.
  In the exposition, we prefer to distinguish these components as they play very different roles in the construction.
\end{remark}

\begin{table}\centering
  \renewcommand{\arraystretch}{1.3}  
  \begin{tabular}{cc|cccc}
    \toprule
    Index & Space & $V$ & $E$ & $F$ & $T$ \\
    \midrule
    0 & $\Xgrad{h}$ & $\Real = \Poly{k}(V)$ & $\Poly{k-1}(E)$ & $\Poly{k-1}(F)$ & $\Poly{k-1}(T)$ \\
    1 & $\Xcurl{h}$ & & $\Poly{k}(E)$ & $\Roly{k-1}(F)\times\cRoly{k}(F)$ & $\Roly{k-1}(T)\times\cRoly{k}(T)$ \\
    2 & $\Xdiv{h}$  & & & $\Poly{k}(F)$ & $\Goly{k-1}(T)\times\cGoly{k}(T)$ \\
    3 & $\Poly{k}(\Th)$ & & & & $\Poly{k}(T)$ \\
    \bottomrule
  \end{tabular}  
  \caption{Polynomial components attached to each mesh vertex $V\in\Vh$, edge $E\in\Eh$, face $F\in\Fh$, and element $T\in\Th$ for each of the DDR spaces. The space $\Poly[\rm c]{k+1}(\Eh)$ in the definition of $\Xgrad{h}$ has been replaced by $\left(\bigtimes_{E\in\Eh}\Poly{k-1}(E)\right)\times\Real^{\Vh}$, see Remark \ref{rem:Poly.c.k+1:isomorphism}.
  \label{tab:local.spaces}}    
\end{table}
\begin{table}\centering
    \begin{tabular}{c|cc|cc|cc|cc}
      \toprule
      \multirow{2}{*}{$k$} & \multicolumn{2}{c|}{$\Xgrad{T}$} & \multicolumn{2}{c|}{$\Xcurl{T}$} & \multicolumn{2}{c|}{$\Xdiv{T}$} & \multicolumn{2}{c}{$\mathrm{L}^k(T)$} \\
      & Tetra & Hexa & Tetra & Hexa & Tetra & Hexa & Tetra & Hexa \\
      \midrule
      0 & 4 (4) & 8 (8) & 6 (6) & 12 (12) & 4 (4) & 6 (6) & 1 (1) & 1 (1) \\
      1 & 15 (10) & 27 (27) & 28 (20) & 46 (54) & 18 (15) & 24 (36) & 4 (4) & 4 (8) \\
      2 & 32 (20) & 54 (64) & 65 (45) & 99 (144) & 44 (36) & 56 (108) & 10 (10) & 10 (27) \\
      \bottomrule    
    \end{tabular}
    \caption{Number of degrees of freedom of the local DDR spaces for tetrahedral and hexahedral elements, and comparison with Raviart--Thomas--N\'ed\'elec (RTN) finite element spaces (in parentheses). For the discrete $\Leb$-space, we have $\mathrm{L}^k(T) = \Qoly{k}(T)$ ($d$-variate polynomials of degree $\le k$ in each variable) for hexahedral RTN finite elements, $\mathrm{L}^k(T) = \Poly{k}(T)$ otherwise.}\label{tab:comparison.fem}  
\end{table}
The polynomial components attached to mesh vertices, edges, faces, and elements for each of the DDR spaces are summarised in Table \ref{tab:local.spaces} (notice that we have accounted for Remark \ref{rem:Poly.c.k+1:isomorphism} for $\Xgrad{h}$).
An inspection of Table \ref{tab:local.spaces} reveals that its diagonal contains full polynomial spaces on the mesh entities of dimension corresponding to the index of the space in the sequence (with the convection that $\Poly{k}(V)\coloneq\Real$ for any vertex $V\in\Vh$).
The components collected in the upper triangular portion of the table are non-zero only for $k\ge 1$, and encode additional information required for the reconstruction of high-order discrete vector calculus operators and potentials.
In particular, the complements $\cRoly{k}(F)$, $\cRoly{k}(T)$, and $\cGoly{k}(T)$ complete the information contained, respectively, in the face curl, element curl and tangential trace, and element divergence to construct the corresponding face or element vector potentials; see Sections \ref{sec:curl}, \ref{sec:Pcurl}, and \ref{sec:Pdiv}.

In what follows, given $\bullet\in\{\GRAD,\CURL,\DIV\}$ and a mesh entity $Y$ of dimension greater than or equal to the index of $\underline{X}_{\bullet,h}^k$, we denote by $\underline{X}_{\bullet,Y}^k$ the restriction of this space to $Y$, i.e., $\underline{X}_{\bullet,Y}^k$ contains the polynomial components attached to $Y$ and to all the mesh entities that lie on its boundary.

\begin{remark}[Comparison with Raviart--Thomas--N\'ed\'elec finite elements]
  When $T$ is a tetrahedron or a hexahedron, the local spaces in the DDR sequence can be compared to classical (Raviart--Thomas--N\'ed\'elec) FE spaces.
  The number of degrees of freedom in each case for polynomial degrees $k\in\{0,1,2\}$ (the most commonly used) is reported in Table \ref{tab:comparison.fem}.
  For $k\ge 1$, the DDR construction leads to slightly larger spaces on tetrahedra and to significantly smaller spaces on hexahedra.
  The number of degrees of freedom for the DDR spaces could be further reduced adapting the serendipity techniques of Virtual Elements \cite{Beirao-da-Veiga.Brezzi.ea:18}; this topic is left for a future work.
  \smallskip

  For codes aiming at general meshes, the implementation of the DDR spaces requires the local (element-by-element) computation of discrete vector operators and potentials, which is an additional cost with respect to traditional FE codes.
  It should be noticed, however, that:
  1) these computations are an embarrassingly parallel task that scales with the number of mesh elements, and are therefore asymptotically less expensive than the resolution of the algebraic systems (see, e.g., Figure \ref{fig:times.vs.elements});
  2) this cost can be substantially reduced when dealing with meshes composed of a finite number of element shapes using standard reference element techniques;
  3) it possible to combine the FE and DDR approaches on a given mesh (using the former on elements of standard shape and the latter on elements of more general shape, possibly resulting from local mesh refinement).
\end{remark}

\subsection{Interpolators}

In the following, for all $\underline{q}_h\in\Xgrad{h}$, we set
  \begin{equation}\label{eq:qE}
    q_E\coloneq(q_{\Eh})_{|E}\in\Poly{k+1}(E).
\end{equation}
The interpolators on the DDR spaces are defined collecting component-wise $\Leb$-projections.
Specifically $\Igrad{h}:\rC{0}(\overline{\Omega})\to\Xgrad{h}$ is such that, for all $q\in \rC{0}(\overline{\Omega})$,
\begin{equation}\label{eq:Igradh}
  \begin{gathered}
    \Igrad{h} q \coloneq \big((\lproj{k-1}{T} q_{|T})_{T\in\Th},(\lproj{k-1}{F} q_{|F})_{F\in\Fh},q_{\Eh}\big)\in\Xgrad{h}
    \\
    \text{
      where $\lproj{k-1}{E} (q_{\Eh})_{|E}=\lproj{k-1}{E} q_{|E}$ for all $E\in\Eh$
    }
    \\
    \text{
      and $q_{\Eh}(\bvec{x}_V)=q(\bvec{x}_V)$ for all $V\in\Vh$.
    }
  \end{gathered}
\end{equation}
$\Icurl{h}:\vC{0}(\overline{\Omega})\to\Xcurl{h}$ is defined setting, for all $\bvec{v}\in\vC{0}(\overline{\Omega})$,
\begin{equation}\label{eq:Icurlh}
  \begin{aligned}
  \Icurl{h}\bvec{v}\coloneq\big({}&
  (\Rproj{k-1}{T}\bvec{v}_{|T},\Rcproj{k}{T}\bvec{v}_{|T})_{T\in\Th},\\
  &(\Rproj{k-1}{F}\bvec{v}_{{\rm t},F},\Rcproj{k}{F}\bvec{v}_{{\rm t},F})_{F\in\Fh},\\
  &(\lproj{k}{E}(\bvec{v}_{|E}\cdot\tangent_E)_{E\in\Eh}
  \big),
  \end{aligned}
\end{equation}
where $\bvec{v}_{{\rm t},F}\coloneq\normal_F\times(\bvec{v}_{|F}\times\normal_F)$ denotes the tangent trace of $\bvec{v}$ over $F$. Finally, $\Idiv{h}:\vSob{1}(\Omega)\to\Xdiv{h}$ is such that, for all $\bvec{w}\in\vSob{1}(\Omega)$,
\begin{equation}\label{eq:Idivh}
\Idiv{h}\bvec{w}\coloneq\big(
(\Gproj{k-1}{T}\bvec{w}_{|T},\Gcproj{k}{T}\bvec{w}_{|T})_{T\in\Th},
(\lproj{k}{F}(\bvec{w}_{|F}\cdot\normal_F)_{F\in\Fh}
\big).
\end{equation}
The restriction of the above interpolators to a mesh entity $Y$ of dimension larger than or equal to the index of the corresponding space in the sequence (see Table \ref{tab:local.spaces}) is denoted replacing the subscript $h$ by $Y$.
Finally, we let $\lproj{k}{h}:\Leb(\Omega)\to\Poly{k}(\Th)$ denote the global $\Leb$-orthogonal projector such that, for all $q\in \Leb(\Omega)$, $(\lproj{k}{h}q)_{|T} = \lproj{k}{T} q_{|T}$ for all $T\in\Th$.

\subsection{Discrete vector calculus operators}\label{sec:ddr:operators}

We define in this section the discrete vector calculus operators that appear in the DDR sequence, obtained collecting the $\Leb$-orthogonal projections of local discrete operators mapping on full polynomial spaces.
In what follows, the operators that only appear in the discrete sequence \eqref{eq:global.sequence.3D} through projections are denoted in sans serif font, while those appearing verbatim (without projection) in the sequence are in standard font.

\subsubsection{Gradient}

The discrete counterpart of the gradient operator in the DDR sequence maps on $\Xcurl{h}$, and therefore requires to define local gradients on mesh edges, faces, and elements.

For any $E\in\Eh$, the \emph{edge gradient} $\GE:\Xgrad{E}\to\Poly{k}(E)$ is defined as: For all $q_E\in\Xgrad{E}=\Poly{k+1}(E)$,
\begin{equation}\label{eq:GE}
  \GE q_E \coloneq q_E',
\end{equation}
where the derivative is taken along $E$ according to the orientation of $\tangent_E$.

For any $F\in\Fh$, the \emph{face gradient} $\cGF:\Xgrad{F}\to\vPoly{k}(F)$ is such that, for all $\underline{q}_F = (q_F, q_{\EF})\in\Xgrad{F}$ and all $\bvec{w}_F\in\vPoly{k}(F)$,
\begin{align}\label{eq:cGF}
  \int_F\cGF\underline{q}_F\cdot\bvec{w}_F
  &= -\int_F q_F\DIV_F\bvec{w}_F
  + \sum_{E\in\EF}\omega_{FE}\int_E q_{\EF}(\bvec{w}_F\cdot\normal_{FE})
  \\ 
  &= \int_F \GRAD_F q_F\cdot\bvec{w}_F
  + \sum_{E\in\EF}\omega_{FE}\int_E (q_{\EF} - q_F) (\bvec{w}_F\cdot\normal_{FE}).\nonumber
\end{align}
The existence and uniqueness of $\cGF\underline{q}_F$ in $\vPoly{k}(F)$ follow from the Riesz representation theorem applied to this space equipped with the usual $\Leb$-product.
Similar considerations hold for the other discrete vector calculus operators defined below, and will not be repeated.

The \emph{scalar trace} $\trF:\Xgrad{F}\to\Poly{k+1}(F)$ is such that, for all $\underline{q}_F\in\Xgrad{F}$,
\begin{multline}\label{eq:trF}
\int_F\trF\underline{q}_F\DIV_F\bvec{v}_F
= -\int_F\cGF\underline{q}_F\cdot\bvec{v}_F
+ \sum_{E\in\EF}\omega_{FE}\int_E q_{\EF}(\bvec{v}_F\cdot\normal_{FE})
\\
\forall\bvec{v}_F\in\cRoly{k+2}(F).
\end{multline}
This relation defines $\trF \underline{q}_F$ uniquely in $\Poly{k+1}(F)$ owing to the isomorphism property \eqref{eq:iso:DIV} with $\ell=k+2$.

\begin{remark}[Validity of \eqref{eq:trF}]\label{rem:validity.trF}
    The relation \eqref{eq:trF} holds, in fact, for any $\bvec{v}_F\in\Roly{k}(F)\oplus\cRoly{k+2}(F)$.
    To check it, take $\bvec{v}_F\in\Roly{k}(F)$ and notice that the left-hand side vanishes owing to $\DIV_F\bvec{v}_F=0$, while the right-hand side vanishes owing to the definition \eqref{eq:cGF} of $\cGF\underline{q}_F$ and again $\DIV_F\bvec{v}_F=0$.
    This means, in particular, that \eqref{eq:trF} holds for any $\bvec{v}_F\in\RT{k+1}(F)\subset\Roly{k}(F)\oplus\cRoly{k+2}(F)$ (see Remark \ref{rem:hierarchical.complements}).
\end{remark}

For all $T\in\Th$, the \emph{element gradient} $\cGT:\Xgrad{T}\to\vPoly{k}(T)$ is defined such that, for all $\underline{q}_T\in\Xgrad{T}$ and all $\bvec{w}_T\in\vPoly{k}(T)$,
\begin{align}\label{eq:cGT}
  \int_T\cGT\underline{q}_T\cdot\bvec{w}_T
  &= -\int_T q_T\DIV\bvec{w}_T
  + \sum_{F\in\FT}\omega_{TF}\int_F\trF\underline{q}_F(\bvec{w}_T\cdot\normal_F)
  \\ \nonumber
  &= \int_T \GRAD q_T\cdot\bvec{w}_T
  + \sum_{F\in\FT}\omega_{TF}\int_F(\trF\underline{q}_F-q_T)(\bvec{w}_T\cdot\normal_F),
\end{align}
where we have performed an integration by parts on the first term in the right-hand side to pass to the second line.

\begin{lemma}[Consistency properties]\label{lem:consistency.grad.trace}
  The edge, face, and element gradients, and scalar trace satisfy the following consistency properties:
  \begin{alignat}{6} \label{eq:GE:consistency}
    &\forall E\in\Eh &\qquad \GE\big(\Igrad{E} q\big) &= \lproj{k}{E} (q') &\qquad& \forall q\in \Sob{1}(E),
    \\ \label{eq:cGF:consistency}
    &\forall F\in\Fh &\qquad \cGF\big(\Igrad{F} q\big) &= \GRAD_F q &\qquad& \forall q\in\Poly{k+1}(F),
    \\ \label{eq:trF:consistency}
    &\forall F\in\Fh &\qquad \trF\big(\Igrad{F} q\big) &= q &\qquad& \forall q\in\Poly{k+1}(F),
    \\ \label{eq:lproj.k-1.trF}
    &\forall F\in\Fh &\qquad \lproj{k-1}{F}\big(\trF\underline{q}_F\big) &= q_F &\qquad& \forall\underline{q}_F\in\Xgrad{F},
    \\ \label{eq:cGT:consistency}
    &\forall T\in\Th &\qquad \cGT\big(\Igrad{T} q\big) &= \GRAD q &\qquad& \forall q\in\Poly{k+1}(T).
  \end{alignat}
\end{lemma}

\begin{proof}
Let us prove \eqref{eq:GE:consistency}. Take $q\in \Sob{1}(E)$. For all $r_E\in\Poly{k}(E)$, denoting by $\bvec{x}_{V_1}$ and $\bvec{x}_{V_2}$ the coordinates of the vertices $V_1$ and $V_2$ of $E$, oriented so that $\tangent_E$ points from $V_1$ to $V_2$, we have
\begin{align*}
  \int_E (\Igrad{E}q)'r_E
  &=(\Igrad{E}q\, r_E)(\bvec{x}_{V_2}) - (\Igrad{E}q\, r_E)(\bvec{x}_{V_1}) - \int_E (\Igrad{E}q) r_E'
  \\
  &=(q\, r_E)(\bvec{x}_{V_2})-(q\, r_E)(\bvec{x}_{V_1}) - \int_E q r_E'
  = \int_E q' r_E,
\end{align*}
where we have used an integration by parts in the first line, obtained the second equality applying the definition of $\Igrad{E}q\in\Poly{k+1}(E)$ (which satisfies $(\Igrad{E}q)(\bvec{x}_V)=q(\bvec{x}_V)$ for all $V\in\VE$ and $\lproj{k-1}{E}(\Igrad{E}q)=\lproj{k-1}{E}q$) together with $r_E'\in\Poly{k-1}(E)$, and used another integration by parts to conclude. This proves that $(\Igrad{E}q)'=\lproj{k}{E}(q')$.

Relation \eqref{eq:cGF:consistency} can be deduced as in \cite[Proposition 4.1]{Di-Pietro.Droniou.ea:20}. To prove \eqref{eq:trF:consistency}, we write \eqref{eq:trF} for $\underline{q}_F = \Igrad{F} q$ with $q\in\Poly{k+1}(T)$, use \eqref{eq:cGF:consistency}, and notice that $q_{\EF} = q_{|\partial F}$ (since $q_{|\partial F}\in\Poly[c]{k+1}(\EF)$)
to get, for all $\bvec{v}_F\in\cRoly{k+2}(F)$,
\[
\begin{aligned}
  \int_F\trF(\Igrad{F} q)\DIV_F\bvec{v}_F
  &=-\int_F\GRAD_F q\cdot\bvec{v}_F
  + \sum_{E\in\EF}\omega_{FE}\int_E q_{|\partial F}(\bvec{v}_F\cdot\normal_{FE})
  \\
  &=\int_F q\DIV_F\bvec{v}_F.
\end{aligned}
\]
The isomorphism property \eqref{eq:iso:DIV} with $\ell=k+2$ then concludes the proof that $\trF(\Igrad{F} q)=q$.

The equality \eqref{eq:lproj.k-1.trF} follows from \eqref{eq:trF} written for $\bvec{v}_F\in\cRoly{k}(F)$ (this choice is made possible by \eqref{eq:hierarchical.complements}) after replacing the full face gradient $\cGF$ by its definition \eqref{eq:cGF}, simplifying the boundary terms, and invoking again the isomorphism property \eqref{eq:iso:DIV}, this time with $\ell=k$.

Finally, \eqref{eq:cGT:consistency} can be established from \eqref{eq:trF:consistency} following the ideas in \cite[Lemma 5.1]{Di-Pietro.Droniou.ea:20}.\qed
\end{proof}

The following proposition contains a stronger version of \cite[Eq. (5.16)]{Di-Pietro.Droniou.ea:20}, with test function taken in the N\'ed\'elec space $\NE{k+1}(T)$ instead of $\vPoly{k}(T)$.

\begin{proposition}[Link between element and face gradients]\label{prop:link.GT.GF}
  For all $T\in\Th$ and all $(\underline{q}_T,\bvec{z}_T)\in\Xgrad{T}\times\NE{k+1}(T)$,
  \begin{equation}\label{eq:link.GT.GF}
    \int_T\cGT\underline{q}_T\cdot\CURL\bvec{z}_T
    = -\sum_{F\in\FT}\omega_{TF}\int_F\cGF\underline{q}_F\cdot(\bvec{z}_T\times\normal_F).
  \end{equation}
\end{proposition}

\begin{proof}
  Writing \eqref{eq:cGT} with $\bvec{w}_T = \CURL\bvec{z}_T\in\vPoly{k}(T)$ and recalling the relation $\DIV\CURL\bvec{z}_T=0$, we have
  \[
  \begin{aligned}
    \int_T\cGT\underline{q}_T\cdot\CURL\bvec{z}_T
    &= \sum_{F\in\FT}\omega_{TF}\int_F\trF\underline{q}_F(\CURL\bvec{z}_T\cdot\normal_F)
    \\
    &= \sum_{F\in\FT}\omega_{TF}\int_F\trF\underline{q}_F\DIV_F(\bvec{z}_T\times\normal_F),
  \end{aligned}
  \]
  the last equality being a consequence of \cite[Eq. (3.7)]{Di-Pietro.Droniou.ea:20}.
  To conclude, we invoke \eqref{eq:trF} with $\bvec{v}_F=(\bvec{z}_T)_{|F}\times\normal_F\in\RT{k+1}(F)$ (cf.\ \eqref{eq:NE.T.trace} and Remark \ref{rem:validity.trF}) and cancel the edge terms using \cite[Eqs. (5.13) and (5.14)]{Di-Pietro.Droniou.ea:20}.\qed
\end{proof}

The \emph{global discrete gradient} $\uGh:\Xgrad{h}\to\Xcurl{h}$ is obtained collecting the projections of each local gradient on the space attached to the corresponding mesh entity: For all $\underline{q}_h\in\Xgrad{h}$,
\begin{equation}\label{eq:uGh}
  \uGh\underline{q}_h\coloneq
  \begin{aligned}[t]
    \Big(
    &\big( \Rproj{k-1}{T}\big(\cGT\underline{q}_T\big),\Rcproj{k}{T}\big(\cGT\underline{q}_T\big) \big)_{T\in\Th},
    \\
    &\big( \Rproj{k-1}{F}\big(\cGF\underline{q}_F\big),\Rcproj{k}{F}\big(\cGF\underline{q}_F\big) \big)_{F\in\Fh},
    \\
    &( \GE q_E )_{E\in\Eh}
    \Big).
  \end{aligned}
\end{equation}

\begin{remark}[Practical implementation]
  In schemes based on the DDR sequence, the discrete gradient \eqref{eq:uGh} only appears as an argument of the discrete $\Leb$-product on $\Xcurl{h}$ (see \eqref{eq:Xcurl:l2.prod} below), that is, composed with the scalar trace and potential reconstruction on this space.
  Thus, leveraging \eqref{eq:Pcurl.uGT=cGT} below, one never has to implement $\uGh$, as only the full element gradients $(\cGT)_{T\in\Th}$ are required.
  Similar considerations hold for the discrete curl defined by \eqref{eq:uCh} below (see also \cite[Remark 7]{Di-Pietro.Droniou:20*1} on this matter).

  Notice that this strategy differs from the one often pursued in the context of Virtual Elements, which consists in directly taking the appropriate components of $\uGh$ as degrees of freedom.
  This difference is linked to the fact that the present construction embeds what could be interpreted in Virtual Element terms as an \emph{enhancement}, enabling us to reduce the degree of certain internal polynomial components.
\end{remark}

\subsubsection{Curl}\label{sec:curl}

We next consider the DDR counterpart of the curl operator, which maps on $\Xdiv{h}$ and therefore has components at mesh faces and inside mesh elements.
For all $F\in\Fh$, the \emph{face curl} $\CF:\Xcurl{F}\to\Poly{k}(F)$ is such that, for all $\uvec{v}_F=\big(\bvec{v}_{\cvec{R},F},\bvec{v}_{\cvec{R},F}^\compl, (v_E)_{E\in\EF}\big)\in\Xcurl{F}$,
\begin{equation}\label{eq:CF}
  \int_F\CF\uvec{v}_F~r_F
  = \int_F\bvec{v}_{\cvec{R},F}\cdot\VROT_F r_F
  - \sum_{E\in\EF}\omega_{FE}\int_E v_Er_F\qquad
  \forall r_F\in\Poly{k}(F).
\end{equation}
Reasoning as in \cite[Proposition 4.3]{Di-Pietro.Droniou.ea:20}, we get
\begin{equation}\label{eq:CF:consistency}
  \CF\big(\Icurl{F}\bvec{v}\big) = \lproj{k}{F}\big(\ROT_F\bvec{v}\big)\qquad
  \forall\bvec{v}\in\vSob{1}(F).
\end{equation}

\begin{proposition}[Local complex property]\label{prop:2d.complex}
  Let $F\in\Fh$ and denote by $\uGF:\Xgrad{F}\to\Xcurl{F}$ the restriction to $F$ of the global gradient $\uGh$ defined by \eqref{eq:uGh}.
  Then, it holds
  \begin{equation}\label{eq:Im.uGF.subset.Ker.CF}
    \Image\uGF \subset \Ker\CF \qquad \forall F\in\Fh.
  \end{equation}
\end{proposition}

\begin{remark}[Two-dimensional complex]
  The relations \eqref{eq:cGF:consistency} and \eqref{eq:Im.uGF.subset.Ker.CF} show that the following two-dimensional sequence forms a complex:
  \[
  \begin{tikzcd}
    \Real\arrow{r}{\Igrad{F}} & \Xgrad{F}\arrow{r}{\uGF} & \Xcurl{F}\arrow{r}{\CF} & \Poly{k}(F)\arrow{r}{0} & \{0\}.
  \end{tikzcd}
  \]
  Having assumed $F$ simply connected, adapting the arguments of \cite[Theorem 4.1]{Di-Pietro.Droniou.ea:20}, one can additionally prove that this complex is exact, that is, $\Ker\uGF = \Igrad{F}\Real$, $\Image\uGF = \Ker \CF$, and $\Image\CF = \Poly{k}(F)$.
\end{remark}

\begin{proof}[Proposition \ref{prop:2d.complex}]
  Let $\underline{q}_F\in\Xgrad{F}$.
  Using the definition \eqref{eq:CF} of $\CF$ and \eqref{eq:uGh} of $\uGh$ we have, for all $r_F\in\Poly{k}(F)$,
  \begin{align*}
    \int_F\CF\big(\uGF\underline{q}_F\big) r_F
    &=
    \int_F \Rproj{k-1}{F}\big(\cGF\underline{q}_F\big)\cdot\VROT_F r_F - \sum_{E\in\EF}\omega_{FE}\int_E \GE\underline{q}_F r_F
    \\
    &=
    \int_F \cGF\underline{q}_F\cdot\VROT_F r_F - \sum_{E\in\EF}\omega_{FE}\int_E \GE\underline{q}_F r_F
    \\
    &=
    \sum_{E\in\EF}\omega_{FE}\int_E\left[
      q_{\EF}(\VROT_F r_F \cdot\normal_{FE})-q_E' r_F
      \right]
    =0,
  \end{align*}
  where the suppression of $\Rproj{k-1}{F}$ in the second line is possible since $\VROT_F r_F\in\Roly{k-1}(F)$,
  the third line is obtained using the definitions \eqref{eq:cGF} of $\cGF$ with $\bvec{w}_F=\VROT_F r_F$ (additionally noticing that $\DIV_F(\VROT_F r_F)=0$) and \eqref{eq:GE} of $\GE$,
  while the conclusion is obtained reasoning as in \cite[Point 2.\ of Proposition 4.4]{Di-Pietro.Droniou.ea:20} (see in particular Eq. (4.19) therein).\qed
\end{proof}

The \emph{tangential trace} $\trFt:\Xcurl{F}\to\vPoly{k}(F)$ is such that, for all $\uvec{v}_F\in\Xcurl{F}$, recalling the notation \eqref{eq:def.Xrec},
\begin{equation}\label{eq:trFt}
  \trFt\uvec{v}_F \coloneq \Xrec{\cvec{R}}{\cvec{F}}{k}{\trFtR\uvec{v}_F}{\bvec{v}_{\cvec{R},F}^\compl},
\end{equation}
where $\trFtR\uvec{v}_F\in\Roly{k}(F)$ is defined, using the isomorphism property \eqref{eq:iso:VROTF.GRAD} with $\ell=k+1$, by
\begin{multline}\label{eq:trFt:Roly.k}
  \int_F\trFtR\uvec{v}_F\cdot\VROT_F r_F
  = \int_F\CF\uvec{v}_F~r_F
  + \sum_{E\in\EF}\omega_{FE}\int_E v_E r_F
  \\
  \forall r_F\in\Poly{0,k+1}(F).
\end{multline}

\begin{remark}[Validity of \eqref{eq:trFt:Roly.k}]\label{rem:validity:trFt:Roly.k}
  Observing that both sides of \eqref{eq:trFt:Roly.k} vanish when $r_F\in\Poly{0}(F)$, it is inferred that this relation holds in fact for any $r_F\in\Poly{k+1}(F)$.
  We also notice that, since $\Rproj{k}{F}\big(\trFt\uvec{v}_F\big)=\trFtR\uvec{v}_F$ (by virtue of \eqref{eq:trFt} and \eqref{eq:recovery.proj}), $\trFtR$ can be replaced by $\trFt$ in the left-hand side of \eqref{eq:trFt:Roly.k}.

    The actual computation of $\trFt$ does not require the implementation of the recovery operator in the right-hand side of \eqref{eq:trFt}, but rather hinges on the solution of the following equation:
    For all $(r_F,\bvec{w}_F)\in\Poly{0,k+1}(F)\times\cRoly{k}(F)$,
    \[
    \int_F\trFt\uvec{v}_F\cdot(\VROT_F r_F + \bvec{w}_F)
    = \int_F\CF\uvec{v}_F~r_F
    + \sum_{E\in\EF}\omega_{FE}\int_E v_E r_F
    + \int_F\bvec{v}_{\cvec{R},F}^\compl\cdot\bvec{w}_F.
    \]
    Indeed, the test functions of the form $(r_F,\bvec{0})$ with $r_F$ spanning $\Poly{0,k+1}(F)$ enforce that $\Rproj{k}{F}\big(\trFt\uvec{v}_F\big) = \trFtR\uvec{v}_F$ satisfies \eqref{eq:trFt:Roly.k},
    while the test functions of the form $(0,\bvec{w}_F)$ with $\bvec{w}_F$ spanning $\cRoly{k}(F)$ enforce that $\Rcproj{k}{F}\big(\trFt\uvec{v}_F\big) = \bvec{v}_{\cvec{R},F}^\compl$.
    These two conditions combined yield \eqref{eq:trFt}.
    Similar considerations hold for the three-dimensional potential reconstructions defined in Sections \ref{sec:Pcurl} and \ref{sec:Pdiv} below.
\end{remark}

\begin{proposition}[Properties of the tangential trace]\label{prop:trFt}
  It holds
  \begin{alignat}{2}\label{eq:Rproj.k-1.trFt}
    \Rproj{k-1}{F}\big(\trFt\uvec{v}_F\big) = \bvec{v}_{\cvec{R},F} \quad&\mbox{and}\quad \Rcproj{k}{F}\big(\trFt\uvec{v}_F\big) = \bvec{v}_{\cvec{R},F}^\compl &\qquad& \forall\uvec{v}_F\in\Xcurl{F},
    \\
    \label{eq:trFt.cons}
    \trFt(\Icurl{F}\bvec{v})&=\vlproj{k}{F}\bvec{v}&\qquad&\forall \bvec{v}\in\NE{k+1}(F),
    \\ \label{eq:trFt:GF}
    \trFt\big(\uGF\underline{q}_F\big) &= \cGF\underline{q}_F &\qquad& \forall\underline{q}_F\in\Xgrad{F}.
  \end{alignat}
\end{proposition}

\begin{proof}
  \underline{1. \emph{Proof of \eqref{eq:Rproj.k-1.trFt}.}}
  Since $\Roly{k-1}(F)\subset\Roly{k}(F)$, we have $\Rproj{k-1}{F}=\Rproj{k-1}{F}\Rproj{k}{F}$ and thus, using \eqref{eq:recovery.proj} and Remark \ref{rem:validity:trFt:Roly.k}, we obtain 
  \[
  \Rproj{k-1}{F}\big(\trFt\uvec{v}_F\big)
  =\Rproj{k-1}{F}\big(\Rproj{k}{F}\trFt\uvec{v}_F\big)
  =\Rproj{k-1}{F}\big(\trFtR\uvec{v}_F\big).
  \]
  Applying the definitions \eqref{eq:trFt:Roly.k} of $\trFtR$ and \eqref{eq:CF} of $\CF$ with a generic $r_F\in\Poly{0,k}(F)$ leads to
  $\int_F{\trFtR\uvec{v}_F\cdot\VROT_F r_F} = \int_F{\bvec{v}_{\cvec{R},F}\cdot\VROT_F r_F}$, hence 
  \[
  \Rproj{k-1}{F}\big(\trFtR\uvec{v}_F\big)=\bvec{v}_{\cvec{R},F}.
  \]
  This proves the first relation in \eqref{eq:Rproj.k-1.trFt}. The second relation is a straightforward consequence of \eqref{eq:trFt} and \eqref{eq:recovery.proj}.
  \medskip\\
  \underline{2. \emph{Proof of \eqref{eq:trFt.cons}.}}
  Let $\bvec{v}\in\NE{k+1}(F)$.
  Writing \eqref{eq:trFt:Roly.k} for $\uvec{v}_F = \Icurl{F}\bvec{v}$, observing that $\CF\big(\Icurl{F}\bvec{v}\big) = \ROT_F\bvec{v}\in\Poly{k}(F)$ by \eqref{eq:CF:consistency} and that $v_E = \bvec{v}_{|E}\cdot\tangent_E$ for all $E\in\EF$ by \eqref{eq:NE.F.trace} with $\ell=k+1$, and integrating by parts the right-hand side, it is inferred that $\trFtR\big(\Icurl{F}\bvec{v}\big) = \Rproj{k}{F}\bvec{v}$. Thus, by \eqref{eq:trFt}, $\trFt\big(\Icurl{F}\bvec{v}\big) = \Xrec{\cvec{R}}{F}{k}{\Rproj{k}{F}\bvec{v}}{\Rcproj{k}{F}\bvec{v}} = \vlproj{k}{F}\bvec{v}$, where the conclusion results from \eqref{eq:projector.composition} with $(\cvec{X},Y,\ell) = (\cvec{R},F,k)$ followed by \eqref{eq:recovery.rec}.
  \medskip\\
  \underline{3. \emph{Proof of \eqref{eq:trFt:GF}.}}
  Let $\underline{q}_F\in\Xgrad{F}$.
  For all $r_F\in\Poly{k+1}(F)$, it holds
  \begin{equation}\label{eq:exact.CF.middle}
    \begin{aligned}
      \sum_{E\in\EF}\omega_{FE}\int_E \GE q_E~r_F
      &= \sum_{E\in\EF}\omega_{FE}\int_E q_{\EF}(\VROT_F r_F\cdot\normal_{FE}) \\
      &= \int_F \cGF\underline{q}_F \cdot\VROT_F r_F,
    \end{aligned}
  \end{equation}
  where the first equality follows recalling that $\GE q_E = q_E'$ on $E$, integrating by parts on each edge, noting that $(r_F)_{|E}'=-\VROT_F r_F\cdot\normal_{FE}$ (see \cite[Eq. (4.20)]{Di-Pietro.Droniou.ea:20}), and cancelling out the vertex values that appear twice with opposite sign,
  while the conclusion is obtained recalling the definition \eqref{eq:cGF} of $\cGF$ and observing that $\DIV_F(\VROT_F r_F) = 0$.
  Writing \eqref{eq:trFt:Roly.k} for $\uvec{v}_F=\uGF\underline{q}_F$, we obtain
  \[
  \begin{aligned}
    \int_F\trFtR\big(\uGF\underline{q}_F\big)\cdot\VROT_F r_F
    &= \int_F\cancel{\CF\big(\uGF\underline{q}_F\big)}~r_F
    + \sum_{E\in\EF}\omega_{FE}\int_E \GE q_E~r_F
    \\
    &= \int_F\cGF\underline{q}_F\cdot\VROT_F r_F,
  \end{aligned}
  \]  
  where we have used the inclusion \eqref{eq:Im.uGF.subset.Ker.CF} in the cancellation, while the conclusion follows from \eqref{eq:exact.CF.middle}.
  This implies $\trFtR\big(\uGF\underline{q}_F\big)=\Rproj{k}{F}\big(\cGF\underline{q}_F\big)$.
  By definition, the component of $\uGF\underline{q}_F$ on $\cRoly{k}(F)$ is $\Rcproj{k}{F}\big(\cGF\underline{q}_F\big)$.
  Plugging the above results into \eqref{eq:trFt} with $\uvec{v}_F=\uGF\underline{q}_F$ and using the recovery formula \eqref{eq:recovery.rec} with $(S,S^\compl) = (\Roly{k}(F),\cRoly{k}(F))$ and $\bvec{a} = \cGF\underline{q}_F$ concludes the proof.\qed
\end{proof}

For all $T\in\Th$, the \emph{element curl} $\cCT:\Xcurl{T}\to\vPoly{k}(T)$ is defined such that, for all $\uvec{v}_T=\big(\bvec{v}_{\cvec{R},T}, \bvec{v}_{\cvec{R},T}^\compl, (\bvec{v}_{\cvec{R},F}, \bvec{v}_{\cvec{R},F}^\compl)_{F\in\FT},(v_E)_{E\in\ET}\big)\in\Xcurl{T}$,
\begin{multline}\label{eq:cCT}
  \int_T\cCT\uvec{v}_T\cdot\bvec{w}_T
  = \int_T\bvec{v}_{\cvec{R},T}\cdot\CURL\bvec{w}_T
  + \sum_{F\in\FT}\omega_{TF}\int_F\trFt\uvec{v}_F\cdot(\bvec{w}_T\times\normal_F)
  \\
  \forall\bvec{w}_T\in\vPoly{k}(T).
\end{multline}
The following polynomial consistency property is proved as in \cite[Lemma 5.2]{Di-Pietro.Droniou.ea:20} (recall the shift of exponent in the notation of the N\'ed\'elec space with respect to this reference):
\begin{equation}\label{eq:CT:consistency}
\forall T\in\Th \qquad \cCT\big(\Icurl{T}\bvec{v}\big) = \CURL\bvec{v} \qquad \forall\bvec{v}\in\NE{k+1}(T).
\end{equation}

\begin{proposition}[Link between element and face curls]\label{prop:cCT.CF}
  For all $(\uvec{v}_T,r_T)\in\Xcurl{T}\times\Poly{k+1}(T)$, it holds
  \begin{equation}\label{eq:cCT.CF}
    \int_T \cCT\uvec{v}_T\cdot\GRAD r_T
    = \sum_{F\in\FT}\omega_{TF}\int_F \CF\uvec{v}_F~r_T.
  \end{equation} 
\end{proposition}
\begin{proof}
  For any $r_T\in\Poly{k+1}(T)$, writing \eqref{eq:cCT} for $\bvec{w}_T=\GRAD r_T\in\vPoly{k}(T)$ and using the fact that $\CURL(\GRAD r_T)=\bvec{0}$ and that $(\GRAD r_T)_{|F}\times\normal_F=\VROT_F(r_{T|F})$ for all $F\in\FT$ (see \cite[Eq. (3.6)]{Di-Pietro.Droniou.ea:20}), we infer 
  that 
  \[
  \int_T{\cCT\uvec{v}_T\cdot\GRAD r_T} = \sum_{F\in\FT}\omega_{TF}\int_F{\trFt\uvec{v}_F\cdot \VROT_F(r_{T|F})}.
  \]
  Using Remark \ref{rem:validity:trFt:Roly.k}, we arrive at
  \[
  \int_T \cCT\uvec{v}_T\cdot\GRAD r_T
  =\sum_{F\in\FT}\omega_{TF}\bigg[
    \int_F \CF\uvec{v}_F r_T + \sum_{E\in\EF}\omega_{FE}\int_E v_E r_T
    \bigg].
  \]
  By \cite[Eq. (5.13)]{Di-Pietro.Droniou.ea:20}, the edge terms in the above expression can be cancelled, thereby proving \eqref{eq:cCT.CF}.\qed
\end{proof}

The \emph{global discrete curl} $\uCh:\Xcurl{h}\to\Xdiv{h}$ is such that, for all $\uvec{v}_h\in\Xcurl{h}$,
\begin{equation}\label{eq:uCh}
  \uCh\uvec{v}_h\coloneq\big(
  \big( \Gproj{k-1}{T}\big(\cCT\uvec{v}_T\big),\Gcproj{k}{T}\big(\cCT\uvec{v}_T\big) \big)_{T\in\Th},
  ( \CF\uvec{v}_F )_{F\in\Fh}
  \big).
\end{equation}

\subsubsection{Divergence}

For all $T\in\Th$, the \emph{element divergence} $\DT:\Xdiv{T}\to\Poly{k}(T)$ is defined by: For all $\uvec{w}_T = \big(\bvec{w}_{\cvec{G},T}, \bvec{w}_{\cvec{G},T}^\compl, (w_F)_{F\in\FT}\big)\in\Xdiv{T}$,
\begin{equation}\label{eq:DT}
  \int_T\DT\uvec{w}_T~ q_T
  = -\!\int_T\bvec{w}_{\cvec{G},T}\cdot\GRAD q_T
  + \hspace{-0.5ex}\sum_{F\in\FT}\!\omega_{TF}\!\int_F w_F q_T
  \quad\forall q_T\in\Poly{k}(T).
\end{equation}
The \emph{global discrete divergence} $\Dh:\Xdiv{h}\to\Poly{k}(\Th)$ is obtained setting, for all $\uvec{w}_h\in\Xdiv{h}$,
\begin{equation}\label{eq:Dh}
  (\Dh\uvec{w}_h)_{|T}\coloneq\DT\uvec{w}_T
  \qquad\forall T\in\Th,
\end{equation}

\begin{proposition}[Local exactness property]\label{prop:ImCT.KerDT}
  It holds, for all $T\in\Th$,
  \begin{equation}\label{eq:Ker.DT}
    \Image\uCT = \Ker\DT,
  \end{equation}
  where $\uCT$ denotes the restriction to $T$ of the global curl $\uCh$ defined by \eqref{eq:uCh}
\end{proposition}

\begin{proof}
  Let us start by proving that $\DT\big(\uCT\uvec{v}_T\big)=0$ for all $\uvec{v}_T\in\Xcurl{T}$, that is,
  $\Image\uCT\subset\Ker(\DT)$.
  By Proposition \ref{prop:cCT.CF}, for all $q_T\in\Poly{k}(T)$,
  \begin{equation}\label{eq:CT.CF}
    \int_T\Gproj{k-1}{T}(\cCT\uvec{v}_T)\cdot\GRAD q_T
    = \sum_{F\in\FT}\omega_{TF}\int_F\CF\uvec{v}_F~q_T,
  \end{equation}
  where we have used $\GRAD q_T\in\Goly{k-1}(T)$ to introduce the projector $\Gproj{k-1}{T}$.
  Hence, using the definition \eqref{eq:DT} of $\DT$, we have, for all $q_T\in\Poly{k}(T)$,
  \[
  \int_T\DT\big(\uCT\uvec{v}_T\big) q_T
  = -\int_T\Gproj{k-1}{T}\big(\cCT\uvec{v}_T\big)\cdot\GRAD q_T + \sum_{F\in\FT}\omega_{TF}\int_F\CF\uvec{v}_T~q_T
  =0.
  \]
  Since $q_T$ is arbitrary in $\Poly{k}(T)$, this shows that $\DT\big(\uCT\uvec{v}_T\big)=0$.
  \smallskip

  Let us now prove the inclusion $\Ker(\DT)\subset\Image\uCT$.
  We fix an element $\uvec{w}_T\in\Xdiv{T}$ such that $\DT\uvec{w}_T=0$ and prove the existence of $\uvec{v}_T\in\Xcurl{T}$ such that $\uvec{w}_T=\uCT\uvec{v}_T$.
  Enforcing $\DT\uvec{w}_T=0$ in \eqref{eq:DT} with $q_T=1$, we infer that $\sum_{F\in\FT}\omega_{TF}\int_Fw_F=0$.
  Thus, \cite[Lemma 5.3]{Di-Pietro.Droniou.ea:20}, which remains valid in the present context, provides $(\bvec{v}_{\cvec{R},F},\bvec{v}_{\cvec{R},F}^\compl)_{F\in\FT}$ and $(v_E)_{E\in\ET}$ such that, for all $F\in\FT$, letting $\uvec{v}_F\coloneq\big(\bvec{v}_{\cvec{R},F},\bvec{v}_{\cvec{R},F}^\compl,(v_E)_{E\in\EF}\big)$, it holds $w_F=\CF\uvec{v}_F$.
  Enforcing again $\DT\uvec{w}_T=0$ in \eqref{eq:DT}, this time for a generic test function $q_T\in\Poly{k}(T)$, and accounting for the previous result, we can write, for all $\uvec{v}_T\in\Xcurl{T}$ with boundary values as above,
  \[
  \int_T\bvec{w}_{\cvec{G},T}\cdot\GRAD q_T
  = \sum_{F\in\FT}\omega_{TF}\int_F\CF\uvec{v}_F~q_T
  = \int_T\Gproj{k-1}{T}(\cCT\uvec{v}_T)\cdot\GRAD q_T,
  \]
  where the conclusion follows from the relation \eqref{eq:CT.CF} linking volume and face curls.
  Since $\GRAD q_T$ spans $\Goly{k-1}(T)$ as $q_T$ spans $\Poly{k}(T)$, this proves that $\Gproj{k-1}{T}\big(\cCT\uvec{v}_T\big) = \bvec{w}_{\cvec{G},T}$.
  Finally, we select $\bvec{v}_{\cvec{R},T}\in\Roly{k-1}(T)$ in such a way as to have $\Gcproj{k}{T}\big(\cCT\uvec{v}_T\big) = \bvec{w}_{\cvec{G},T}^\compl$, that is, recalling \eqref{eq:cCT},
  \begin{multline}\label{eq:vRT}
    \int_T\bvec{v}_{\cvec{R},T}\cdot\CURL\bvec{z}_T
    = \int_T\bvec{w}_{\cvec{G},T}^\compl\cdot\bvec{z}_T
    - \sum_{F\in\FT}\omega_{TF}\int_F\trFt\uvec{v}_F\cdot(\bvec{z}_T\times\normal_F)
    \\
    \forall\bvec{z}_T\in\cGoly{k}(T).
  \end{multline}
  By the isomorphism \eqref{eq:iso:CURL}, this condition defines $\bvec{v}_{\cvec{R},T}$ uniquely.\qed
\end{proof}

\subsection{Discrete sequence}

Recalling the definitions \eqref{eq:Igradh}, \eqref{eq:uGh}, \eqref{eq:uCh}, and \eqref{eq:Dh} of the global discrete operators, the DDR sequence reads:
\begin{equation}\label{eq:global.sequence.3D}
  \begin{tikzcd}
    \Real\arrow{r}{\Igrad{h}} & \Xgrad{h}\arrow{r}{\uGh} & \Xcurl{h}\arrow{r}{\uCh} & \Xdiv{h}\arrow{r}{\Dh} & \Poly{k}(\Th)\arrow{r}{0} & \{0\}.
  \end{tikzcd}
\end{equation}

\begin{remark}[Variations]
  In the spirit of \cite[Section 9]{Beirao-da-Veiga.Brezzi.ea:16}, one could consider alternatives of the DDR sequence \eqref{eq:global.sequence.3D} obtained varying certain couples of polynomial degrees in such a way as to preserve the exactness properties.
  Thus one could, e.g., replace $\Roly{k-1}(T)$ with $\Roly{k}(T)$ in the definition \eqref{eq:Xcurl.h} of $\Xcurl{h}$ and, correspondingly, $\cGoly{k}(T)$ with $\cGoly{k+1}(T)$ in the definition \eqref{eq:Xdiv.h} of $\Xdiv{h}$.
  With these changes, the results of Proposition \ref{prop:ImCT.KerDT} (and, in particular, \eqref{eq:vRT}) remain valid.
  Assessing the impact such and similar changes on the consistency is, however, more delicate.
  These developments are left for a future work.
\end{remark}

\subsection{Commutation properties}

\begin{lemma}[Local commutation properties]\label{lem:commutation}
  It holds, for all $T\in\Th$,
  \begin{alignat}{4} \label{eq:uGT:commutation}
    \uGT\big(\Igrad{T} q\big) &= \Icurl{T}\big(\GRAD q\big) &\qquad&\forall q\in \rC{1}(\overline{T}),
    \\ \label{eq:uCT:commutation}
    \uCT\big(\Icurl{T}\bvec{v}\big) &= \Idiv{T}\big(\CURL\bvec{v}\big) &\qquad&\forall\bvec{v}\in\vSob{2}(T),
    \\ \label{eq:DT:commutation}
    \DT\big(\Idiv{T}\bvec{w}\big) &= \lproj{k}{T}\big(\DIV\bvec{w}\big) &\qquad&\forall\bvec{w}\in\vSob{1}(T).
  \end{alignat}
\end{lemma}

\begin{remark}[Global commutation properties]
  Global commutation properties can be readily inferred from the local ones stated in Lemma \ref{lem:commutation} when interpolating functions that have sufficient global regularity.
\end{remark}

\begin{remark}[Role of commutation properties in the design of robust methods]
  The commutation properties of Lemma \ref{lem:commutation} play a key role in the design of discretisation methods robust with respect to the variations of physical parameters.
  See, e.g., \cite{Di-Pietro.Droniou:21} concerning a DDR method for the Reissner--Mindlin plate bending problem robust with respect to plate thickness.
\end{remark}

\begin{proof}[Lemma \ref{lem:commutation}]
  We start by noticing that all the interpolates defined in \eqref{eq:uGT:commutation}--\eqref{eq:DT:commutation} are well-defined under the assumed regularities.
  \medskip\\  
  \underline{1. \emph{Proof of \eqref{eq:uGT:commutation}.}}
  By \eqref{eq:GE:consistency} it holds, for all $E\in\ET$, 
  \[
  \GE\big(\Igrad{E} q_{|E}\big) = \lproj{k}{E}(q_{|E}')={\lproj{k}{E}\big((\GRAD q)_{|E}\cdot\tangent_E\big)}.
  \]
  Let now $F\in\FT$. Writing the definition \eqref{eq:cGF} of $\cGF$ with $\underline{q}_F = \Igrad{F} q_{|F}$ and $\bvec{w}_F\in\RT{k}(F)$, and recalling \eqref{eq:RT.F.trace} to replace $q_{\EF}$ with $\lproj{k-1}{E}(q_{\EF})_{|E} = \lproj{k-1}{E} q_{|E}$ (see \eqref{eq:Igradh}) in each edge integral, we infer
  \[
  \begin{aligned}
    \int_F\cGF\big(\Igrad{F} q_{|F}\big)\cdot\bvec{w}_F
    &= -\int_F\lproj{k-1}{F} q_{|F}\DIV_F\bvec{w}_F
    \\
    &\quad
    + \sum_{E\in\EF}\omega_{FE}\int_E\lproj{k-1}{E} q_{|E}(\bvec{w}_F\cdot\normal_{FE})
    \\
    &= -\int_F q\DIV_F\bvec{w}_F
    + \sum_{E\in\EF}\omega_{FE}\int_E q(\bvec{w}_F\cdot\normal_{FE})
    \\
    &= \int_F\GRAD_F q_{|F}\cdot\bvec{w}_F,
  \end{aligned}
  \]
  where we have removed the projectors using their definition in the second equality and we have integrated by parts to conclude.
  Recalling the definition \eqref{eq:NE.RT} of $\RT{k}(F)$, we can first let $\bvec{w}_F$ span $\Roly{k-1}(F)$ to infer 
  \[
  \Rproj{k-1}{F}\big[\cGF\big(\Igrad{F} q_{|F}\big)\big] = \Rproj{k-1}{F}\big(\GRAD_F q_{|F}\big),
  \]
  and then $\cRoly{k}(F)$ to infer 
  \[
  \Rcproj{k}{F}\big[\cGF\big(\Igrad{F} q_{|F}\big)\big] = \Rcproj{k}{F}\big(\GRAD_F q_{|F}\big).
  \]
  The proof that $\Rproj{k-1}{T}\big[\cGT\big(\Igrad{T} q\big)\big] = \Rproj{k-1}{T}\big(\GRAD q\big)$ and $\Rcproj{k}{T}\big[\cGT\big(\Igrad{T} q\big)\big] = \Rcproj{k}{T}\big(\GRAD q\big)$ is similar: we write the definition \eqref{eq:cGT} of $\cGT$ for $\underline{q}_T = \Igrad{T} q$ and $\bvec{w}_T\in\RT{k}(T)$, use property \eqref{eq:RT.T.trace} along with \eqref{eq:lproj.k-1.trF} to replace the trace $\trF\big(\Igrad{F} q_{|F}\big)$ with $\lproj{k-1}{F}\big[\trF\big(\Igrad{F} q_{|F}\big)\big] = \lproj{k-1}{F} q_{|F}$ in each face integral, remove the projectors using their definitions, and integrate by parts.
  This concludes the proof of \eqref{eq:uGT:commutation}.
  \medskip\\
  \underline{2. \emph{Proof of \eqref{eq:uCT:commutation}.}}
  For all $F\in\FT$, by \eqref{eq:CF:consistency} it holds 
  \[
  \CF\big(\Icurl{F}\bvec{v}_{|F}\big) = \lproj{k}{F}\big((\CURL\bvec{v})_{|F}\cdot\normal_F\big),
  \]
  where we have used $\ROT_F \bvec{v}_{{\rm t},F}=(\CURL\bvec{v})_{|F}\cdot\normal_F$, see \cite[Eq. (3.7)]{Di-Pietro.Droniou.ea:20}.
  Writing the definition \eqref{eq:cCT} for $\bvec{w}_T\in\NE{k}(T)$, we have
  \begin{equation}\label{eq:uCT:commutation:1}
    \begin{aligned}
    \int_T\cCT\big(\Icurl{T}\bvec{v}{}&\big)\cdot\bvec{w}_T
    = \int_T\cancel{\Rproj{k-1}{T}}\bvec{v}\cdot\CURL\bvec{w}_T
    \\
    &
    + \sum_{F\in\FT}\omega_{TF}\int_F\boldsymbol{\pi}_{\cvec{RT},F}^k\big[\trFt\big(\Icurl{F}\uvec{v}_{{\rm t},F}\big)\big]\cdot(\bvec{w}_T\times\normal_F),
    \end{aligned}
  \end{equation}
  where we have removed $\Rproj{k-1}{T}$ using its definition and, recalling \eqref{eq:NE.T.trace}, we have introduced the $\Leb$-orthogonal projector $\boldsymbol{\pi}_{\cvec{RT},F}^k$ on $\RT{k}(F)$ in the boundary integral.
  By \eqref{eq:NE.RT} together with \eqref{eq:recovery.rec} written with the choices $(E,S,S^\compl) = (\RT{k}(F),\Roly{k-1}(F),\cRoly{k}(F))$ and \eqref{eq:Rproj.k-1.trFt},
  \[
  \boldsymbol{\pi}_{\cvec{RT},F}^k\big[\trFt\big(\Icurl{F}\bvec{v}_{{\rm t},F}\big)\big]
  = \rec{\Roly{k-1}(F),\cRoly{k}(F)}{\Rproj{k-1}{F}\bvec{v}_{{\rm t},F}}{\Rcproj{k}{F}\bvec{v}_{{\rm t},F}}
  = \boldsymbol{\pi}_{\cvec{RT},F}^k\bvec{v}_{{\rm t},F}.
  \]
  Plugging this relation into \eqref{eq:uCT:commutation:1}, we infer
  \[
  \begin{aligned}
    \int_T\cCT\big(\Icurl{T}\bvec{v}\big)\cdot\bvec{w}_T
    &= \int_T\bvec{v}\cdot\CURL\bvec{w}_T
    + \sum_{F\in\FT}\omega_{TF}\int_F\cancel{\boldsymbol{\pi}_{\cvec{RT},F}^k}\bvec{v}_{{\rm t},F}\cdot(\bvec{w}_T\times\normal_F)
    \\
    &= \int_T\CURL\bvec{v}\cdot\bvec{w}_T,
  \end{aligned}
  \]
  where we have used again \eqref{eq:NE.T.trace} to remove the projector in the boundary term and we have integrated by parts to conclude.
  Letting $\bvec{w}_T$ span $\Goly{k-1}(T)$ (respectively $\cGoly{k}(T)$), this yields $\Gproj{k-1}{T}\big[\cCT\big(\Icurl{T}\bvec{v}\big)\big] = \Gproj{k-1}{T}\big(\CURL\bvec{v}\big)$ (respectively $\Gcproj{k}{T}\big[\cCT\big(\Icurl{T}\bvec{v}\big)\big] = \Gcproj{k}{T}\big(\CURL\bvec{v}\big)$), thus concluding the proof of \eqref{eq:uCT:commutation}.
  \medskip\\
  \underline{3. \emph{Proof of \eqref{eq:DT:commutation}.}}
  The proof is done as in \cite[Lemma 5.4]{Di-Pietro.Droniou.ea:20}, noticing that the cancellation of the component in the complement of $\Goly{k}(T)$, obtained therein by orthogonality of this complement, is not required here since this component is absent from the definition \eqref{eq:DT} of $\DT$.\qed
\end{proof}

\subsection{Complex and exactness properties}

The properties collected in the following theorem show that the sequence \eqref{eq:global.sequence.3D} forms a (cochain) complex.

\begin{theorem}[Complex property]\label{thm:global.exactness}
  It holds
  \begin{align}\label{eq:Im.Igrad.equal.Ker.uGh}
    \Igrad{h}\Real &= \Ker\uGh,
    \\ \label{eq:Im.uGh.subset.Ker.uCh}
    \Image\uGh&\subset\Ker\uCh,
    \\ \label{eq:Im.uCh.subset.Ker.Dh}
    \Image\uCh&\subset\Ker\Dh,
    \\ \label{eq:Im.Dh.equal.Pk.Th}
    \Image\Dh&=\Poly{k}(\Th).
  \end{align}
\end{theorem}

\begin{proof}
  \underline{1. \emph{Proof of \eqref{eq:Im.Igrad.equal.Ker.uGh}.}}
  From the consistency properties \eqref{eq:GE:consistency}, \eqref{eq:cGF:consistency} and \eqref{eq:cGT:consistency} of the full gradients and the definition \eqref{eq:uGh} of $\uGh$, it is readily inferred that $\uGh\big(\Igrad{h}C\big) = \underline{\bvec{0}}$ for all $C\in\Real$, hence $\Igrad{h}\Real\subset\Ker\uGh$.

  To prove converse inclusion $\Ker\uGh\subset\Igrad{h}\Real$, let $\underline{q}_h\in\Xgrad{h}$ be such that $\uGh\underline{q}_h = \uvec{0}$.
  By the definitions \eqref{eq:uGh} of $\uGh$ and \eqref{eq:GE} of $\GE$, this means that $q_E' = 0$ for all $E\in\Eh$, that is, $(q_{\Eh})_{|E}$ is constant over $E$.
  Since $\Omega$ has only one connected component, accounting for the single-valuedness of $q_{\Eh}$ at vertices, we thus infer the existence of $C\in\Real$ such that $q_{\Eh} = C$.
  Let now $F\in\Fh$ and $\bvec{w}_F\in\cRoly{k}(F)$.
  We have $\Rcproj{k}{F}\big(\cGF\underline{q}_F\big)=\bvec{0}$, and thus
  \[
  \begin{aligned}
    0 &= \int_F\cGF\underline{q}_F\cdot\bvec{w}_F
    \\
    &= -\int_F q_F\DIV_F\bvec{w}_F
    + \sum_{E\in\EF}\omega_{FE}\int_E q_{\EF}(\bvec{w}_F\cdot\normal_{FE})
    \\
    &= \int_F (C - q_F)\DIV_F\bvec{w}_F,
  \end{aligned}
  \]
  where the second equality comes from the definition \eqref{eq:cGF} of $\cGF\underline{q}_F$, and the conclusion is obtained accounting for the fact that $q_{\EF}=C$ and integrating by parts.  
  Since $\bvec{w}_F$ is generic in $\cRoly{k}(F)$, recalling the isomorphism \eqref{eq:iso:DIV} this implies $\lproj{k-1}{F}(q_F - C)=0$, and thus $q_F=\lproj{k-1}{F}C$.
  As, for all $F\in\Fh$, the previous results give $\underline{q}_F=(q_F, q_{\EF})=\Igrad{F} C$, we also have $\trF\underline{q}_F=C$ by \eqref{eq:trF:consistency}.
  Similarly, let $T\in\Th$ and $\bvec{w}_T\in\cRoly{k}(T)$.
  Writing the definition \eqref{eq:cGT} of $\cGT\underline{q}_T$ for $\bvec{w}_T\in\cRoly{k}(T)$, and accounting for $\Rcproj{k}{T}\big(\cGT\underline{q}_T\big)=\bvec{0}$ and $\trF\underline{q}_F=C$, it is inferred
  \[
  \begin{aligned}
    0 &= \int_T\cGT\underline{q}_T\cdot\bvec{w}_T
    \\
    &= -\int_T q_T\DIV\bvec{w}_T
    + \sum_{F\in\FT}\omega_{TF}\int_F C (\bvec{w}_T\cdot\normal_F)
    \\
    &= \int_F (C - q_T)\DIV\bvec{w}_T,
  \end{aligned}
  \]
  which implies, invoking the isomorphism \eqref{eq:iso:DIV}, $\lproj{k-1}{T}(q_T - C) = 0$ since $\bvec{w}_T$ is generic in $\cRoly{k}(T)$. Hence $q_T=\lproj{k-1}{T}C$ for all $T\in\Th$, which concludes the proof that $\underline{q}_h=\Igrad{h}C$.
  \medskip\\
  \underline{2. \emph{Proof of \eqref{eq:Im.uGh.subset.Ker.uCh}.}}
  The inclusion \eqref{eq:Im.uGh.subset.Ker.uCh} follows from the local property:
  \begin{equation} \label{eq:Im.uGT.subset.Ker.uCT}
    \Image\uGT \subset \Ker\uCT \qquad \forall T\in\Th,
  \end{equation}
  i.e., $\uCT\big(\uGT\underline{q}_T\big)=\underline{\bvec{0}}$ for all $\underline{q}_T\in\Xgrad{T}$.
  Let $T\in\Th$.
  The relation \eqref{eq:Im.uGF.subset.Ker.CF} implies $\CF\big(\uGF\underline{q}_F\big)=0$ for all $F\in\FT$.
  The fact that $\Gproj{k-1}{T}\big[\cCT\big(\uGT\underline{q}_T\big)\big]=\bvec{0}$ then follows from \eqref{eq:CT.CF}.
  We next notice that it holds, for all $\bvec{w}_T\in\cGoly{k}(T)$,
  \[
  \begin{aligned}
    \int_T\cCT\big(\uGT\underline{q}_T\big)\cdot\bvec{w}_T
    &= \int_T\cancel{\Rproj{k-1}{T}}\big(\cGT\underline{q}_T\big)\cdot\CURL\bvec{w}_T
    \\
    &\quad+ \sum_{F\in\FT}\omega_{TF}\int_F\cGF\underline{q}_F\cdot(\bvec{w}_T\times\normal_F)
    = 0,
  \end{aligned}
  \]
  where we have used the definition \eqref{eq:cCT} of $\cCT$ and the property \eqref{eq:trFt:GF} of the tangential trace reconstruction in the first equality, the fact that $\CURL\bvec{w}_T\in\Roly{k-1}(T)$ to cancel the projector, and the link \eqref{eq:link.GT.GF} between volume and face gradients to conclude. This shows that $\Gcproj{k}{T}\big[\cCT\big(\uGT\underline{q}_T\big)\big]=\bvec{0}$ and concludes the proof of \eqref{eq:Im.uGT.subset.Ker.uCT}.
  \medskip\\
  \underline{3. \emph{Proof of \eqref{eq:Im.uCh.subset.Ker.Dh}.}}
  Immediate consequence of \eqref{eq:Ker.DT} after observing that $\uCT$ and $\DT$ are the restrictions of $\uCh$ and $\Dh$ to $T$, respectively.
  \medskip\\
  \underline{4. \emph{Proof of \eqref{eq:Im.Dh.equal.Pk.Th}.}}
  The inclusion $\Image\Dh\subset\Poly{k}(\Th)$ is an obvious consequence of the definition \eqref{eq:Dh} of the global divergence.
  To prove the converse inclusion, let $q_h\in\Poly{k}(\Th)$.
  Since the continuous divergence operator $\DIV:\vSob{1}(\Omega)\to \Leb(\Omega)$ is onto (see, e.g., \cite[Lemma 8.3]{Di-Pietro.Droniou:20}), there exists $\bvec{v}\in\vSob{1}(\Omega)$ such that $\DIV\bvec{v} = q_h$.
  Setting $\uvec{v}_h\coloneq\Idiv{h}\bvec{v}\in\Xdiv{h}$, the commutation property \eqref{eq:DT:commutation} and the definition \eqref{eq:Dh} of the global divergence yield $\Dh\uvec{v}_h = \lproj{k}{h}(\DIV\bvec{v}) = \lproj{k}{h} q_h = q_h$.
  This shows that $\Poly{k}(\Th)\subset\Image\Dh$, thereby concluding the proof of \eqref{eq:Im.Dh.equal.Pk.Th}.\qed
\end{proof}

\begin{remark}[Kernel of the full curl operator]\label{rem:ker.cCT}
  Combining the inclusion \eqref{eq:Im.uGF.subset.Ker.CF} with the relation \eqref{eq:cCT.CF} linking element and face curls, it is inferred that, for all $\underline{q}_T\in\Xgrad{T}$, $\Gproj{k}{T}\big[\cCT\big(\uGT\underline{q}_T\big)\big] = \bvec{0}$.
  On the other hand, \eqref{eq:Im.uGT.subset.Ker.uCT} implies $\Gcproj{k}{T}\big[\cCT\big(\uGT\underline{q}_T\big)\big] = \bvec{0}$. Hence $\cCT\big(\uGT\underline{q}_T\big) = \bvec{0}$ by \eqref{eq:recovery.rec} with $\bvec{a} = \cCT\big(\uGT\underline{q}_T\big)$ and $(S,S^\compl) = (\Goly{k}(T),\cGoly{k}(T))$.
  This shows that $\Image\uGT\subset\Ker\cCT$.
\end{remark}

The exactness properties of the DDR sequence, depending on the topology of the domain, are collected in the following theorem.

\begin{theorem}[Exactness]
  Denoting by $(b_0,b_1,b_2,b_3)$ the Betti numbers of $\Omega$ (with $b_0=1$ since $\Omega$ is connected and $b_3=0$ since $\Omega\subset\Real^3$), we have
  \begin{align}\label{eq:Im.uGh.equal.Ker.uCh}
    b_1=0 \implies \Image\uGh &= \Ker\uCh,
    \\ \label{eq:Im.uCh.equal.Ker.Dh}
    b_2=0 \implies \Image\uCh &= \Ker\Dh.
  \end{align}
\end{theorem}
\begin{remark}[Meaning of vanishing Betti numbers]
  In broad terms, the condition $b_1=0$ means that $\Omega$ does not have any tunnel, while $b_2=0$ means that $\Omega$ does not enclose any void. A typical example of $\Omega$ that has $b_1\not=0$ is (the interior of) a torus, and an example of $\Omega$ with $b_2\not=0$ is a domain enclosed between two concentric spheres.
\end{remark}
\begin{proof}
  \underline{1. \emph{Proof of \eqref{eq:Im.uGh.equal.Ker.uCh}.}}
  Recalling \eqref{eq:Im.uGh.subset.Ker.uCh}, we only have to show the inclusion
  \begin{equation}\label{eq:Ker.uCh.subset.Im.uGh}
    \Ker\uCh\subset\Image\uGh,
  \end{equation}
  that is, for all $\uvec{v}_h\in\Xcurl{h}$ such that $\uCh\uvec{v}_h = \uvec{0}$, there exists $\underline{q}_h\in\Xgrad{h}$ such that $\uvec{v}_h = \uGh\underline{q}_h$.
  In what follows, we show how to construct such a $\underline{q}_h$.
  \smallskip
  
  We start by constructing a function $q_{\Eh}\in\Poly[\rm c]{k+1}(\Eh)$ such that $v_E = (q_{\Eh})'$ for all $E\in\Eh$.
  Let $V_0,V\in\Vh$ be two distinct mesh vertices of coordinates $\bvec{x}_{V_0}$ and $\bvec{x}_V$, respectively, and denote by $\edges{P}\subset\Eh$ a set of edges that form a connected path $P$ from $V_0$ to $V$ (such a path always exists since $\Omega$ is connected).
  By the fundamental theorem of calculus, there is a unique function $q_{\edges{P}}\in\Poly[\rm c]{k+1}(\edges{P})$ such that $q_{\edges{P}}(\bvec{x}_{V_0})=0$ and $(q_{\edges{P}})_{|E}' = v_E$ for all $E\in\edges{P}$, the derivative being taken in the direction of $E$ ($q_{\edges{P}}$ is obtained integrating, in the direction defined on each edge $E$ by $\tangent_E$, the functions $(v_E)_{E\in\edges{P}}$).
  We want to show that the value $q_{\edges{P}}(\bvec{x}_{V})$ taken at $V$ is independent of the choice of the path $P$.
  To this end, denote by $\widetilde{P}$ another path from $V_0$ to $V$ formed by the edges in $\edges{\widetilde{P}}$, and denote by $-\widetilde{P}$ the same path but with reversed orientation.
  We assume, for the moment, that $\edges{P}$ and $\edges{\widetilde{P}}$ are disjoint.
  By similar considerations as before, there exists a unique $q_{\edges{\widetilde{P}}}\in\Poly[\rm c]{k+1}(\edges{\widetilde{P}})$ such that $q_{\edges{\widetilde{P}}}(\bvec{x}_{V_0})=0$ and $(q_{\edges{\widetilde{P}}})_{|E}' = v_E$ for all $E\in\edges{\widetilde{P}}$.
  Since $b_1=0$ (i.e., there is no ``tunnel'' crossing $\Omega$), the path $B\coloneq P -P'$ formed by the edges in $\edges{B}\coloneq\edges{P}\cup\edges{\widetilde{P}}$ is a $1$-boundary, i.e., there is a set of faces $\faces{B}\subset\Fh$ giving rise to a connected surface $S_B\coloneq\bigcup_{F\in\faces{B}}\overline{F}$ such that $B=\partial S_B$.
  We fix an orientation for $S_B$ and, for all $F\in\faces{B}$, we denote by $\omega_{BF}\in\{-1,1\}$ the orientation of $F$ relative to $S_B$.
  For all $E\in\edges{B}$, there is a unique face $F\in\faces{B}$ such that $E\in\EF$, and we let $\omega_{BE}\coloneq\omega_{BF}\omega_{FE}$ denote the orientation of $E$ relative to $S_B$.
  Since $\CF\uvec{v}_F = 0$ for all $F\in\faces{B}$, it holds
  \[
    0 = \sum_{F\in\faces{B}}\omega_{BF}\int_F \CF\uvec{v}_F
    = - \sum_{F\in\faces{B}}\omega_{BF}\sum_{E\in\EF}\omega_{FE}\int_E v_E
    = - \sum_{E\in\edges{B}}\omega_{BE}\int_E v_E,
  \]
  where the second equality is obtained from \eqref{eq:CF} with $r_F$ identically equal to 1, while the conclusion follows observing that all the edges that are interior to $S_B$ appear exactly twice in the sum, with opposite signs.
  Thus, reasoning as in \cite[Proposition 4.2]{Di-Pietro.Droniou.ea:20}, there exists $q_{\edges{B}}\in\Poly[\rm c]{k+1}(\edges{B})$ such that $(q_{\edges{B}})'=v_E$ for all $E\in\edges{B}$, which we can be uniquely identified by additionally prescribing that $q_{\edges{B}}(\bvec{x}_{V_0}) = 0$.
  Under this condition, by uniqueness we infer $(q_{\edges{B}})_{|E} = (q_{\edges{P}})_{|E}$ for all $E\in\edges{P}$ and $(q_{\edges{B}})_{|E} = (q_{\edges{\widetilde{P}}})_{|E}$ for all $E\in\edges{\widetilde{P}}$.
  Since $q_{\edges{B}}$ is continuous at the vertices of $B$, this shows that $q_{\edges{P}}(\bvec{x}_V) = q_{\edges{\widetilde{P}}}(\bvec{x}_V)$.
  This argument can be extended to paths $P$ and $\widetilde{P}$ such that $\edges{P}\cap\edges{\widetilde{P}}\neq\emptyset$, the only difference being that one should reason, in this case, on each connected component of the manifold $S_B$ (corresponding to a ``loop'' inside the path $B = P - \widetilde{P}$).
  
  Repeating this reasoning for each vertex $V\in\Vh$ and all possible paths connecting $V_0$ and $V$, we conclude that there exists a unique $q_{\Eh}\in\Poly[\rm c]{k+1}(\Eh)$ such that $q_{\Eh}(\bvec{x}_{V_0})=0$ and, recalling the notation \eqref{eq:qE},
  \begin{equation}\label{eq:Im.uGh.equal.Ker.uCh:qEh}
    q_E' = \GE q_E = v_E\qquad\forall E\in\Eh.
  \end{equation}
  \smallskip

  Let now $F\in\Fh$.
  We look for a $q_F\in\Poly{k-1}(F)$ such that $\underline{q}_F\coloneq(q_F,q_{\EF})\in\Xgrad{F}$ satisfies $\uvec{v}_F = \uGF\underline{q}_F$.
  Plugging $\CF\uvec{v}_F = 0$ into \eqref{eq:CF}, we infer, for all $r_F\in\Poly{k}(F)$,
  \[
  \begin{aligned}
    \int_F\bvec{v}_{\cvec{R},F}\cdot\VROT_F r_F
    &= \sum_{E\in\EF}\omega_{FE}\int_E v_E r_F
    \\
    &= \sum_{E\in\EF}\omega_{FE}\int_E \GE q_E~r_F
    \\
    &= \int_F\cGF\underline{q}_F\cdot\VROT_F r_F,
  \end{aligned}
  \]
  where
  the second equality is a consequence of \eqref{eq:Im.uGh.equal.Ker.uCh:qEh}, while
  the conclusion follows from \eqref{eq:exact.CF.middle}.
  This shows that $\bvec{v}_{\cvec{R},F} = \Rproj{k-1}{F}\big(\cGF\underline{q}_F)$ for all $q_F\in\Poly{k-1}(F)$.
  Let us now enforce $\bvec{v}_{\cvec{R},F}^\compl = \Rcproj{k}{F}\big(\cGF\underline{q}_F\big)$, that is, for all $\bvec{w}_F\in\cRoly{k}(F)$,
  \[
  \begin{aligned}
  \int_F\bvec{v}_{\cvec{R},F}\cdot\bvec{w}_F
  ={}& \int_F\cGF\underline{q}_F\cdot\bvec{w}_F\\
  ={}& -\int_F q_F \DIV_F\bvec{w}_F + \sum_{E\in\EF}\omega_{FE}\int_E q_{\EF} (\bvec{w}_F\cdot\normal_{FE}),
  \end{aligned}
  \]
  where we have used the definition \eqref{eq:cGF} of $\cGF$ in the second equality.
  Recalling the isomorphism \eqref{eq:iso:DIV}, the above condition defines the sought $q_F\in\Poly{k-1}(F)$ uniquely.
  \smallskip

  Writing the definition \eqref{eq:cCT} of $\cCT\uvec{v}_T$ with $\bvec{w}_T\in\cGoly{k}(T)$, using $\uvec{v}_h\in\Ker\uCh$ to see that $\Gcproj{k}{T}\big(\cCT\uvec{v}_T\big)=\bvec{0}$,
  invoking \eqref{eq:trFt:GF} to write $\trFt\uvec{v}_F=\cGF\underline{q}_F$ and using the link \eqref{eq:link.GT.GF} between element and face gradients with $\bvec{z}_T=\bvec{w}_T$, we see that, for any $T\in\Th$,
  $\bvec{v}_{\cvec{R},T} = \Rproj{k-1}{T}\big(\cGT\underline{q}_T\big)$ with $\underline{q}_T\coloneq \big(q_T, (q_F)_{F\in\FT}, q_{\ET}\big)$ for all $q_T\in\Poly{k-1}(T)$.
  Proceeding then as for $q_F$ above, we can select $q_T\in\Poly{k-1}(T)$ to additionally have $\bvec{v}_{\cvec{R},T}^\compl = \Rcproj{k}{T}\big(\cGT\underline{q}_T\big)$.
  This concludes the proof of \eqref{eq:Ker.uCh.subset.Im.uGh}.
  \medskip\\
  \underline{2. \emph{Proof of \eqref{eq:Im.uCh.equal.Ker.Dh}.}}
  The proof can be obtained reasoning as in \cite[Point 2b) of Theorem 3]{Di-Pietro.Droniou:20*1}.
  As a matter of fact, this argument is based on a local exactness property analogous to \eqref{eq:Ker.DT} together with a topological assembly of the mesh valid for domains that do not enclose voids ($b_2 = 0$), and it therefore does not depend on the specific choice of the complements in \eqref{eq:spaces.T} and \eqref{eq:spaces.F}.\qed
\end{proof}

\section{Potential reconstructions and $\Leb$-products on discrete spaces}
\label{sec:potentials.L2products}

The definitions of the element gradient $\cGT$ and curl $\cCT$ required us to introduce discrete scalar and tangential traces on the mesh faces. In this section, for each $T\in\Th$ and $\bullet\in\{\GRAD,\CURL,\DIV\}$, we define discrete potential reconstructions inside $T$ acting on the discrete space $\Xbullet{T}$.
These potentials have consistency properties, and enable the design of discrete $\Leb$-inner products on DDR spaces that are also consistent.

\subsection{Scalar potential on $\Xgrad{T}$}

The scalar potential reconstruction $\Pgrad:\Xgrad{T}\to\Poly{k+1}(T)$ is such that, for all $\underline{q}_T\in\Xgrad{T}$,
\begin{multline}\label{eq:PgradT}
  \int_T\Pgrad\underline{q}_T\DIV\bvec{v}_T
  = -\int_T\cGT\underline{q}_T\cdot\bvec{v}_T
  + \sum_{F\in\FT}\omega_{TF}\int_F\trF\underline{q}_F(\bvec{v}_T\cdot\normal_F)
  \\
  \forall\bvec{v}_T\in\cRoly{k+2}(T),  
\end{multline}
with $\trF$ defined by \eqref{eq:trF}.
This relation defines $\Pgrad\underline{q}_T$ uniquely since $\DIV:\cRoly{k+2}(T)\to\Poly{k+1}(T)$ is an isomorphism by \eqref{eq:iso:DIV} with $\ell=k+2$.

\begin{remark}[Validity of \eqref{eq:PgradT}]\label{rem:validity.Pgrad}
  The definition \eqref{eq:cGT} of $\cGT$ and the identity $\DIV\CURL=0$ show that both sides of \eqref{eq:PgradT} vanish when $\bvec{v}_T\in\Roly{k}(T)$. Hence, \eqref{eq:PgradT} actually holds for any $\bvec{v}_T\in\Roly{k}(T)\oplus\cRoly{k+2}(T)=\vPoly{k}(T)+\cRoly{k+2}(T)$, the second equality following from $\cRoly{k}(T)\subset \cRoly{k+2}(T)$ and the decomposition \eqref{eq:decomposition:Poly.ell.T}.
\end{remark}

Using the polynomial consistency properties $\cGT\big(\Igrad{T} q\big) = \GRAD q$ and $\trF\big(\Igrad{F}q_{|F}\big) = q_{|F}$, valid for all $q\in\Poly{k+1}(T)$ (see \eqref{eq:cGT:consistency} and \eqref{eq:trF:consistency}, respectively), the following polynomial consistency of $\Pgrad$ is obtained:
\begin{equation}
  \label{eq:Pgrad:polynomial.consistency}
  \Pgrad\big(\Igrad{T} q\big) = q \qquad\forall q\in\Poly{k+1}(T).
\end{equation}
Moreover, applying \eqref{eq:PgradT} to $\bvec{v}_T\in\cRoly{k}(T)$ (see Remark \ref{rem:hierarchical.complements}), using the definition \eqref{eq:cGT} of $\cGT$ with $\bvec{w}_T=\bvec{v}_T$, and recalling that $\DIV:\cRoly{k}(T)\to\Poly{k-1}(T)$ is onto, we obtain
\begin{equation}
  \label{eq:Pgrad:projection}
  \lproj{k-1}{T}\big(\Pgrad\underline{q}_T\big) = q_T\qquad\forall \underline{q}_T\in\Xgrad{T}.
\end{equation}

\subsection{Vector potential on $\Xcurl{T}$}\label{sec:Pcurl}

Recalling the definition \eqref{eq:def.Xrec} of $\recRT{k}{\cdot}{\cdot}$, the vector potential reconstruction $\Pcurl:\Xcurl{T}\to\vPoly{k}(T)$ is such that, for all $\uvec{v}_T\in\Xcurl{T}$,
\begin{equation}\label{eq:Pcurl}
  \Pcurl\uvec{v}_T \coloneq \recRT{k}{\PcurlR\uvec{v}_T}{\bvec{v}_{\cvec{R},T}^\compl},
\end{equation}
where $\PcurlR\uvec{v}_T\in\Roly{k}(T)$ is defined, using the fact that $\CURL:\cGoly{k+1}(T)\to\Roly{k}(T)$ is an isomorphism (see \eqref{eq:iso:CURL}), by
\begin{multline}\label{eq:PcurlR}
  \int_T\PcurlR\uvec{v}_T\cdot\CURL\bvec{w}_T
  = \int_T\cCT\uvec{v}_T\cdot\bvec{w}_T
  - \sum_{F\in\FT}\omega_{TF}\int_F\trFt\uvec{v}_F\cdot(\bvec{w}_T\times\normal_F)
  \\
  \forall\bvec{w}_T\in\cGoly{k+1}(T).
\end{multline}
\begin{remark}[Discrete integration by parts formula for $\Pcurl$]\label{rem:ibp.PcurlT}
  Formula \eqref{eq:PcurlR} can be extended to test functions in the N\'ed\'elec space $\NE{k+1}(T)$ defined by \eqref{eq:NE.RT}.
  To check it, simply notice that both sides vanish whenever $\bvec{w}_T\in\Goly{k}(T)$ (use $\CURL\GRAD=0$ and the definition \eqref{eq:cCT} of $\cCT$). Since $\Rproj{k}{T}\big(\Pcurl\uvec{v}_T\big) = \PcurlR\uvec{v}_T$ (see \eqref{eq:Pcurl} and \eqref{eq:recovery.proj}), we infer that
  \begin{multline}\label{eq:ibp.Pcurl}
    \int_T\Pcurl\uvec{v}_T\cdot\CURL\bvec{z}_T
    = \int_T\cCT\uvec{v}_T\cdot\bvec{z}_T
    - \sum_{F\in\FT}\omega_{TF}\int_F\trFt\uvec{v}_F\cdot(\bvec{z}_T\times\normal_F)
    \\
    \forall\bvec{z}_T\in\NE{k+1}(T).
  \end{multline}
\end{remark}
Apply \eqref{eq:ibp.Pcurl} to $\uvec{v}_T=\Icurl{T}\bvec{v}$ with $\bvec{v}\in\vPoly{k}(T)$, use the consistency properties $\trFt\big(\Icurl{F}\bvec{v}\big) = \vlproj{k}{F}\bvec{v}_{{\rm t},F}=\bvec{v}_{{\rm t},F}$ and $\cCT\big(\Icurl{T}\bvec{v}\big) = \CURL\bvec{v}$ (see \eqref{eq:trFt.cons} and \eqref{eq:CT:consistency}, respectively), and integrate by parts. Since $\CURL:\NE{k+1}(T)\to\Roly{k}(T)$ is onto (due to the isomorphism property \eqref{eq:iso:CURL}), we obtain the relation $\Rproj{k}{T}\big[\Pcurl\big(\Icurl{T}\bvec{v}\big)\big] = \Rproj{k}{T}\bvec{v}$. The definition \eqref{eq:Pcurl} and the property \eqref{eq:recovery.proj} of the recovery operator also yield $\Rcproj{k}{T}\big[\Pcurl\big(\Icurl{T}\bvec{v}\big)\big] = \Rcproj{k}{T}\bvec{v}$. As a consequence,
\begin{equation}\label{eq:Pcurl:polynomial.consistency}
  \Pcurl\big(\Icurl{T}\bvec{v}\big) = \bvec{v}\qquad\forall\bvec{v}\in\vPoly{k}(T).
\end{equation}
Using similar arguments as in the proof of Proposition \ref{prop:trFt}, we also have
\begin{equation}\label{eq:proj.Pcurl}
  \text{
    $\Rproj{k-1}{T}\big(\Pcurl\uvec{v}_T\big) = \bvec{v}_{\cvec{R},T}$
    and
    $\Rcproj{k}{T}\big(\Pcurl\uvec{v}_T\big) = \bvec{v}_{\cvec{R},T}^\compl$}
  \quad\forall \uvec{v}_T\in\Xcurl{T}.
\end{equation}

\subsection{Vector potential on $\Xdiv{T}$}\label{sec:Pdiv}

Recalling the definition \eqref{eq:def.Xrec} of $\recGT{k}{\cdot}{\cdot}$, the vector potential reconstruction $\Pdiv:\Xdiv{T}\to\vPoly{k}(T)$ is such that, for all $\uvec{w}_T\in\Xdiv{T}$,
\begin{equation}\label{eq:Pdiv}
\Pdiv\uvec{w}_T = \recGT{k}{\PdivG\uvec{w}_T}{\bvec{w}_{\cvec{G},T}^\compl},
\end{equation}
where $\PdivG\uvec{w}_T\in\Goly{k}(T)$ is defined by
\begin{multline}\label{eq:PdivG}
  \int_T\PdivG\uvec{w}_T\cdot \GRAD r_T
  = -\int_T\DT\uvec{w}_T~r_T
  + \sum_{F\in\FT}\omega_{TF}\int_Fw_F~r_T
  \\
  \forall r_T\in\Poly{0,k+1}(T).
\end{multline}

\begin{remark}[Discrete integration by parts formula for $\Pdiv$]\label{rem:ibp.PdivT}
  Observing that $\PdivG=\Gproj{k}{T}\Pdiv$ (use \eqref{eq:recovery.proj}) and that \eqref{eq:PdivG} holds for any $r_T\in\Poly{k+1}(T)$ (as can be proved taking $r_T$ constant in $T$ and observing that both sides of this equation vanish due to the definition \eqref{eq:DT} of $\DT$), we infer
  \begin{multline}\label{eq:Pdiv:ibp}
    \int_T\Pdiv\uvec{w}_T\cdot \GRAD r_T
    = -\int_T\DT\uvec{w}_T~r_T
    + \sum_{F\in\FT}\omega_{TF}\int_Fw_F~r_T
    \\
    \forall r_T\in\Poly{k+1}(T).
  \end{multline}
\end{remark}
Writing this relation for $\uvec{w}_T = \Idiv{T}\bvec{w}$ with $\bvec{w}\in\RT{k+1}(T)$, observing that $\DT\big(\Idiv{T}\bvec{w}\big) = \lproj{k}{T}(\DIV\bvec{w}) = \DIV\bvec{w}$ by \eqref{eq:DT:commutation} and $\lproj{k}{F}(\bvec{w}_{|F}\cdot\normal_F) = \bvec{w}_{|F}\cdot\normal_F$ for all $F\in\FT$ by \eqref{eq:RT.T.trace}, and integrating by parts the right-hand side of the resulting expression, we infer $\Gproj{k}{T}\big[\Pdiv\big(\Idiv{T}\bvec{w}\big)\big] = \Gproj{k}{T}\bvec{w}$; since $\Gcproj{k}{T}\big[\Pdiv\big(\Idiv{T}\bvec{w}\big)\big] = \Gcproj{k}{T}\bvec{w}$ by definition of $\Pdiv$, $\Idiv{T}$ and \eqref{eq:recovery.proj}, we deduce that
\begin{equation}\label{eq:Pdiv:polynomial.consistency}
  \Pdiv\big(\Idiv{T}\bvec{w}\big) = \vlproj{k}{T}\bvec{w}\qquad\forall\bvec{w}\in\RT{k+1}(T).
\end{equation}
Moreover, using similar arguments as in Proposition \ref{prop:trFt} we get
\begin{equation}\label{eq:proj.Pdiv}
    \Gproj{k-1}{T}\big(\Pdiv\uvec{w}_T\big) = \bvec{w}_{\cvec{G},T}\mbox{ and }
    \Gcproj{k}{T}\big(\Pdiv\uvec{w}_T\big) = \bvec{w}_{\cvec{G},T}^\compl
  \quad\forall\uvec{w}_T\in\Xdiv{T}.
\end{equation}

\subsection{Discrete $\Leb$-products}\label{sec:discrete.L2}

We now define discrete $\Leb$-inner products on the DDR spaces. These products are all constructed in a similar way: by assembling local contributions composed of a consistent term based on the potential reconstruction and a stabilisation term that provides a control over the polynomial components on the lower dimensional geometrical objects. 
Specifically, each $\Leb$-product $(\cdot,\cdot)_{\GRAD,h}:\Xgrad{h}\times\Xgrad{h}\to\Real$, $(\cdot,\cdot)_{\CURL,h}:\Xcurl{h}\times\Xcurl{h}\to\Real$, and $(\cdot,\cdot)_{\DIV,h}:\Xdiv{h}\times\Xdiv{h}\to\Real$ is the sum over $T\in\Th$ of its local counterpart defined by:
\begin{multline} \label{eq:Xgrad:l2.prod}
  (\underline{r}_T,\underline{q}_T)_{\GRAD,T}
  \coloneq \int_T\Pgrad\underline{r}_T~\Pgrad\underline{q}_T + \mathrm{s}_{\GRAD,T}(\underline{r}_T, \underline{q}_T)
  \\
  \forall (\underline{r}_T,\underline{q}_T)\in\Xgrad{T}\times\Xgrad{T},
\end{multline}
\begin{multline} \label{eq:Xcurl:l2.prod}
  (\uvec{w}_T,\uvec{v}_T)_{\CURL,T}
  \coloneq \int_T\Pcurl\uvec{w}_T{\cdot}\Pcurl\uvec{v}_T + \mathrm{s}_{\CURL,T}(\uvec{w}_T, \uvec{v}_T)
  \\
  \forall (\uvec{w}_T,\uvec{v}_T)\in\Xcurl{T}\times\Xcurl{T},
\end{multline}
\begin{multline} \label{eq:Xdiv:l2.prod}
  (\uvec{w}_T,\uvec{v}_T)_{\DIV,T}
  \coloneq \int_T\Pdiv\uvec{w}_T{\cdot}\Pdiv\uvec{v}_T + \mathrm{s}_{\DIV,T}(\uvec{w}_T, \uvec{v}_T)
  \\
  \forall (\uvec{w}_T,\uvec{v}_T)\in\Xdiv{T}\times\Xdiv{T},
\end{multline}
with symmetric, positive semidefinite stabilisation bilinear forms $\mathrm{s}_{\bullet,T}$, $\bullet\in\{\GRAD,\CURL,\DIV\}$ defined as follows:
\begin{align}
  \mathrm{s}_{\GRAD,T}(\underline{r}_T,\underline{q}_T)
  \coloneq{}
    &\sum_{F\in\FT}h_F\int_F\big(\Pgrad\underline{r}_T-\trF\underline{r}_F\big) \big(\Pgrad\underline{q}_T-\trF\underline{q}_F\big)
    \nonumber\\
    & + \sum_{E\in\ET}h_E^2\int_E \big(\Pgrad\underline{r}_T-r_E\big) \big(\Pgrad\underline{q}_T-q_E\big),
  \label{eq:sgradT}
\end{align}
 \begin{align} 
   \mathrm{s}_{\CURL,T}{}&(\uvec{w}_T,\uvec{v}_T)\coloneq
   \nonumber\\
   {}&
     \sum_{F\in\FT}h_F\int_F \big( (\Pcurl\uvec{w}_T)_{{\rm t},F}-\trFt\uvec{w}_F\big)\cdot\big( (\Pcurl\uvec{v}_T)_{{\rm t},F}-\trFt\uvec{v}_F\big)
     \nonumber\\
     & + \sum_{E\in\ET}h_E^2\int_E\big(\Pcurl\uvec{w}_T\cdot\tangent_E-w_E\big)\big(\Pcurl\uvec{v}_T\cdot\tangent_E-v_E\big),
    \label{eq:scurlT}
  \end{align}
where we recall that the index ${\rm t},F$ denotes the tangential trace on $F$, and
\begin{equation}\label{eq:divT}
  \mathrm{s}_{\DIV,T}(\uvec{w}_T,\uvec{v}_T)
  \coloneq\sum_{F\in\FT}h_F\int_F\big(\Pdiv\uvec{w}_T\cdot\normal_F-w_F\big)\big(\Pdiv\uvec{v}_T\cdot\normal_F-v_F\big).
\end{equation}
These local stabilisation bilinear forms $\mathrm{s}_{\bullet,T}$ are polynomially consistent, i.e., they vanish whenever one of their arguments is the interpolate of a polynomial of total degree $\le k+1$ for $\bullet=\GRAD$, or $\le k$ for $\bullet\in\{\CURL,\DIV\}$. 
The consistency properties on interpolates of smooth functions of the potential reconstructions and stabilisation forms proved in Section \ref{sec:consistency.results} below make these discrete $\Leb$-products natural candidates for use in the discretisation of PDEs in weak formulation; see the application of DDR to the magnetostatic problem in Section \ref{sec:application}.

For $\bullet\in\{\GRAD,\CURL,\DIV\}$, we denote by $\norm[\bullet,T]{{\cdot}}$ the norm on $\Xbullet{T}$ induced by the corresponding local discrete $\Leb$-product $(\cdot,\cdot)_{\bullet,T}$, and by $\norm[\bullet,h]{{\cdot}}$ the norm on $\Xbullet{h}$ corresponding to the global discrete $\Leb$-product $(\cdot,\cdot)_{\bullet,h}$.

\subsection{Component $\Leb$-norms, bounds, and equivalence properties}\label{sec:components.norms}

The analysis of the stability and consistency properties of the DDR sequence is facilitated by the introduction of $\Leb$-like norms naturally associated with the choices of polynomial components in the DDR spaces.
Specifically we set, for all $\underline{q}_h\in\Xgrad{h}$,
\[
\tnorm[\GRAD,h]{\underline{q}_h}\coloneq\bigg(
\sum_{T\in\Th}\tnorm[\GRAD,T]{\underline{q}_T}^2
\bigg)^{\frac12}\mbox{ with}
\]%
\begin{equation}\label{eq:tnorm.T.F}
  \begin{alignedat}{4}
    &\forall T\in\Th,\; \forall\underline{q}_T\in\Xgrad{T},\\
    &\qquad
    \tnorm[\GRAD,T]{\underline{q}_T}\coloneq\bigg(
    \norm[\Leb(T)]{q_T}^2 + \sum_{F\in\FT} h_F\tnorm[\GRAD,F]{\underline{q}_F}^2
    \bigg)^{\frac12}
    \\
    &\forall F\in\Fh,\; \forall\underline{q}_F\in\Xgrad{F},\\
    &\qquad
    \tnorm[\GRAD,F]{\underline{q}_F}\coloneq\bigg(
    \norm[\Leb(F)]{q_F}^2 + \sum_{E\in\EF} h_E\norm[\Leb(E)]{q_E}^2
    \bigg)^{\frac12}.
  \end{alignedat}
\end{equation}
Similarly, for all $\uvec{v}_h\in\Xcurl{h}$,
\[
\tnorm[\CURL,h]{\uvec{v}_h}\coloneq
\bigg(
\sum_{T\in\Th}\tnorm[\CURL,T]{\uvec{v}_T}^2
\bigg)^{\frac12}\mbox{ with}
\]
\begin{equation}\label{eq:tnorm.curl}
  \begin{alignedat}{4}
  &\forall T\in\Th,\; \forall\uvec{v}_T\in\Xcurl{T},\\
  &\qquad
    \tnorm[\CURL,T]{\uvec{v}_T}\coloneq\bigg(
    \norm[\vLeb(T)]{\bvec{v}_{\cvec{R},T}}^2
    + \norm[\vLeb(T)]{\bvec{v}_{\cvec{R},T}^\compl}^2
    + \sum_{F\in\FT} h_F\tnorm[\CURL,F]{\uvec{v}_F}^2
    \bigg)^{\frac12}
    \\
    &\forall F\in\Fh,\; \forall\uvec{v}_F\in\Xcurl{F},\\
    &\qquad
    \tnorm[\CURL,F]{\uvec{v}_F}\coloneq\bigg(
    \norm[\vLeb(F)]{\bvec{v}_{\cvec{R},F}}^2
    + \norm[\vLeb(F)]{\bvec{v}_{\cvec{R},F}^\compl}^2
    + \sum_{E\in\EF} h_E\norm[\Leb(E)]{v_E}^2
    \bigg)^{\frac12}.
  \end{alignedat}
\end{equation}
Finally, for all $\uvec{w}_h\in\Xdiv{h}$,
\[
\tnorm[\DIV,h]{\uvec{w}_h}\coloneq\bigg(
\sum_{T\in\Th}\tnorm[\DIV,T]{\uvec{w}_T}^2
\bigg)^{\frac12}\mbox{ with}
\]
\[
\begin{alignedat}{2}
&\forall T\in\Th,\;\forall\uvec{w}_T\in\Xdiv{T},\\
&\qquad\tnorm[\DIV,T]{\uvec{w}_T}\coloneq\bigg(
\norm[\vLeb(T)]{\bvec{w}_{\cvec{G},T}}^2
+ \norm[\vLeb(T)]{\bvec{w}_{\cvec{G},T}^\compl}^2
+ \sum_{F\in\FT} h_F \norm[\Leb(F)]{w_F}^2
\bigg)^{\frac12}.
\end{alignedat}
\]

\begin{remark}[Alternative stabilisations]
  Each of the component norms $\tnorm[\bullet,T]{{\cdot}}$, for $\bullet\in\{\GRAD,\CURL,\DIV\}$, is a Euclidean norm on the corresponding local space $\Xbullet{T}$. It is therefore associated to an inner product $[\cdot,\cdot]_{\bullet,T}$, which can be used to design an alternative  stabilisation to $\mathrm{s}_{\bullet,T}$ by setting $\widetilde{\mathrm{s}}_{\bullet,T}(\uvec{v}_T,\uvec{w}_T)\coloneq[\uvec{v}_T-\uvec{I}_{\bullet,T}\bvec{P}_{\bullet,T}\uvec{v}_T,\uvec{w}_T-\uvec{I}_{\bullet,T}\bvec{P}_{\bullet,T}\uvec{w}_T]_{\bullet,T}$, where $\uvec{I}_{\bullet,T}$ and $\bvec{P}_{\bullet,T}$ are respectively the interpolator and potential reconstruction on $\Xbullet{T}$. This alternative stabilisation is the one chosen for $\Xcurl{T}$ in \cite{Di-Pietro.Droniou:20*1}. We also note that $\mathrm{s}_{\DIV,T}=\widetilde{\mathrm{s}}_{\DIV,T}$.
\end{remark}

The next proposition follows from \eqref{eq:estimate.norm.rec} and Lemma \ref{lem:norm.isomorphisms} in Appendix \ref{sec:results}, in a similar way as in the proof of \cite[Proposition 13]{Di-Pietro.Droniou:20*1}.

\begin{proposition}[Boundedness of local potentials]\label{prop:bound.P}
  It holds, for all $T\in\Th$ and all $F\in\FT$,
  \begin{alignat}{3}
    \label{eq:bound.trF.Pgrad}
    &\left\{\begin{aligned}
      \norm[\Leb(F)]{\trF\underline{q}_F} &\lesssim \tnorm[\GRAD,F]{\underline{q}_F} \\
      \norm[\Leb(T)]{\Pgrad\underline{q}_T} &\lesssim \tnorm[\GRAD,T]{\underline{q}_T}
    \end{aligned}\right.
    &\qquad&\forall \underline{q}_T\in\Xgrad{T},
    \\
    \label{eq:bound.trFt.Pcurl}
    &
    \left\{\begin{aligned}
      \norm[\vLeb(F)]{\trFt\uvec{v}_F} &\lesssim \tnorm[\CURL,F]{\uvec{v}_F} \\
      \norm[\vLeb(T)]{\Pcurl\uvec{v}_T} &\lesssim \tnorm[\CURL,T]{\uvec{v}_T}
    \end{aligned}\right.
    &\qquad&\forall \uvec{v}_T\in\Xcurl{T},\\
    \label{eq:bound.Pdiv}
    &\norm[\vLeb(T)]{\Pdiv\uvec{w}_T}\lesssim \tnorm[\DIV,T]{\uvec{w}_T}
    &\qquad&\forall \uvec{w}_T\in\Xdiv{T}.
  \end{alignat}
\end{proposition}

We next establish the equivalence of the norms corresponding to the discrete $\Leb$-products and the component norms.

\begin{lemma}[Equivalence of norms]\label{lem:equiv.norms}
Let $\bullet\in\{\GRAD,\CURL,\DIV\}$. We have, for any $T\in\Th$,
\begin{equation}\label{eq:equiv.norms}
  \tnorm[\bullet,T]{\underline{z}_T}\simeq \norm[\bullet,T]{\underline{z}_T}\qquad\forall \underline{z}_T\in\Xbullet{T}.
\end{equation}
\end{lemma}

\begin{proof}
We only prove the result for $\bullet=\DIV$, the other cases being similar (see also \cite[Proposition 14]{Di-Pietro.Droniou:20*1} for $\bullet=\CURL$ in the case of orthogonal complements, instead of the Koszul complements \eqref{eq:spaces.F}, \eqref{eq:spaces.T}).
Let $T\in\Th$ and $\uvec{w}_T\in\Xdiv{T}$.
By definition \eqref{eq:Xdiv:l2.prod} of the $\Leb$-product on $\Xdiv{T}$, we have
\begin{align*}
\norm[\DIV,T]{\uvec{w}_T}^2={}&\norm[\vLeb(T)]{\Pdiv\uvec{w}_T}^2+\sum_{F\in\FT}h_F\norm[\Leb(F)]{\Pdiv\uvec{w}_T\cdot\normal_F-w_F}^2\\
\lesssim{}& \norm[\vLeb(T)]{\Pdiv\uvec{w}_T}^2+\sum_{F\in\FT}h_F\norm[\Leb(F)]{w_F}^2\lesssim \tnorm[\DIV,T]{\uvec{w}_T}^2,
\end{align*}
where the first inequality follows from a triangle inequality together with the discrete trace inequality $h_F\norm[\vLeb(F)]{\Pdiv\uvec{w}_T}^2\lesssim\norm[\vLeb(T)]{\Pdiv\uvec{w}_T}^2$ (see \cite[Lemma 1.32]{Di-Pietro.Droniou:20}), while the conclusion is a consequence of \eqref{eq:bound.Pdiv} together with the definition of $\tnorm[\DIV,T]{\uvec{w}_T}$. This proves $\gtrsim$ in \eqref{eq:equiv.norms}.

To prove the converse inequality, we start from \eqref{eq:proj.Pdiv} to write
\begin{align*}
  \tnorm[\DIV,T]{\uvec{w}_T}^2
  ={}&\norm[\vLeb(T)]{\Gproj{k-1}{T}(\Pdiv\uvec{w}_T)}^2+ \norm[\vLeb(T)]{\Gcproj{k}{T}(\Pdiv\uvec{w}_T)}^2\\
  &+ \sum_{F\in\FT} h_F \norm[\Leb(F)]{w_F}^2\\
  \lesssim{}& \norm[\vLeb(T)]{\Pdiv\uvec{w}_T}^2 + \sum_{F\in\FT} h_F \norm[\Leb(F)]{w_F-\Pdiv\uvec{w}_T\cdot\normal_F}^2
  \\
  &+\sum_{F\in\FT} h_F \norm[\Leb(F)]{\Pdiv\uvec{w}_T\cdot\normal_F}^2\\
\lesssim{}& \norm[\DIV,T]{\uvec{w}_T}^2,
\end{align*}
where the first inequality follows from the $\Leb$-boundedness of the orthogonal projectors $\Gproj{k-1}{T}$ and $\Gcproj{k}{T}$ together with a triangle inequality,
and the conclusion is obtained invoking the same discrete trace inequality as before together with the definition of $\norm[\DIV,T]{\uvec{w}_T}$. This proves $\lesssim$ in \eqref{eq:equiv.norms}.\qed
\end{proof}

\begin{lemma}[Boundedness of local interpolators]\label{lem:bound.I}
  It holds, for all $T\in\Th$,
  \begin{alignat}{2}
    \label{eq:bound.Igrad}
    \tnorm[\GRAD,T]{\Igrad{T} q}&\lesssim \norm[\Leb(T)]{q}+h_T\seminorm[\Sob{1}(T)]{q}+h_T^2\seminorm[\Sob{2}(T)]{q}&\quad&\forall q\in \Sob{2}(T),\\
    \label{eq:bound.Icurl}
    \tnorm[\CURL,T]{\Icurl{T} \bvec{v}}&\lesssim \norm[\vLeb(T)]{\bvec{v}}+h_T\seminorm[\vSob{1}(T)]{\bvec{v}}+h_T^2\seminorm[\vSob{2}(T)]{\bvec{v}}&\quad&\forall \bvec{v}\in \vSob{2}(T),\\
    \label{eq:bound.Idiv}
    \tnorm[\DIV,T]{\Idiv{T}\bvec{w}}&\lesssim \norm[\vLeb(T)]{\bvec{w}}+h_T\seminorm[\vSob{1}(T)]{\bvec{w}}&\quad&\forall \bvec{w}\in \vSob{1}(T).
  \end{alignat}
\end{lemma}  

\begin{remark}[Boundedness in other norms]
  The boundedness of $\Igrad{T}$ and $\Icurl{T}$ could easily be stated using norms in larger spaces (typically, $C^0(\overline{T})$ for $\Igrad{T}$, and the same spaces on which usual N\'ed\'elec interpolators are defined for $\Icurl{T}$ -- see \cite[Section 2.5.3]{Boffi.Brezzi.ea:13}). However, the role of Lemma \ref{lem:bound.I} is to enable primal consistency estimates (Theorem \ref{thm:approx.PTIT}); since these estimates require higher regularity on the solutions, the bounds \eqref{eq:bound.Igrad} and \eqref{eq:bound.Icurl} stated in non-minimal norms are sufficient to our purpose.
\end{remark}

\begin{proof}[Lemma \ref{lem:bound.I}]
  The definition \eqref{eq:Igradh} of $\Igrad{T}$ shows that $\tnorm[\GRAD,T]{\Igrad{T} q}\lesssim |T|^{\nicefrac12}\max_T|q|$. By \cite[Eq.\ (5.110)]{Di-Pietro.Droniou:20}, it holds 
  \[
  \max_T|q|\lesssim |T|^{-\frac12}\sum_{r=0}^2 h_T^r\seminorm[\Sob{r}(T)]{q},
  \]
  which concludes the proof of \eqref{eq:bound.Igrad}. The estimate \eqref{eq:bound.Icurl} is obtained the same way. As for \eqref{eq:bound.Idiv}, by the continuous trace inequality of \cite[Lemma 1.31]{Di-Pietro.Droniou:20}, we have 
  \[
  \norm[\Leb(F)]{\lproj{k}{F}(\bvec{w}\cdot\normal_F)}\le 
  \norm[\vLeb(F)]{\bvec{w}}\lesssim h_F^{-\frac12}\norm[\vLeb(T)]{\bvec{w}}+h_F^{\frac12}\seminorm[\vSob{1}(T)]{\bvec{w}}.
  \]
  Using this bound in the definition \eqref{eq:Idivh} of $\Idiv{T}$ yields \eqref{eq:bound.Idiv}.\qed
\end{proof}


\subsection{Links between discrete vector potentials and vector calculus operators}

In the next proposition, we show that the element gradient and curl can be recovered applying the suitable potential reconstruction to the corresponding discrete vector calculus operator, in a similar way
as in \eqref{eq:trFt:GF} for the tangential face reconstruction and face gradients.

\begin{proposition}[Link between discrete vector potentials and vector calculus operators]
  For all $T\in\Th$, it holds
  \begin{alignat}{2} \label{eq:Pcurl.uGT=cGT}
      \Pcurl\big(\uGT\underline{q}_T\big) &= \cGT\underline{q}_T &\qquad& \forall\underline{q}_T\in\Xgrad{T},
      \\ \label{eq:Pdiv.uCT=cCT}
      \Pdiv\big(\uCT\uvec{v}_T\big) &= \cCT\uvec{v}_T &\qquad& \forall\uvec{v}_T\in\Xcurl{T}.
    \end{alignat}
\end{proposition}
\begin{proof}
  \underline{1. \emph{Proof of \eqref{eq:Pcurl.uGT=cGT}}.}
  By the second projection property in \eqref{eq:proj.Pcurl}, we have
  $\Rcproj{k}{T}\big[\Pcurl\big(\uGT\underline{q}_T\big)\big] = \Rcproj{k}{T}\big(\cGT\underline{q}_T\big)$.
  To infer the conclusion, it then suffices to prove that
  \begin{equation}\label{eq:rproj.pcurl.uGT}
    \Rproj{k}{T}\big[\Pcurl\big(\uGT\underline{q}_T\big)\big]
    = \Rproj{k}{T}\big(\cGT\underline{q}_T\big)
  \end{equation}
  and invoke \eqref{eq:recovery.rec}. To prove \eqref{eq:rproj.pcurl.uGT}, we take $\bvec{z}_T\in\NE{k+1}(T)$ and apply \eqref{eq:ibp.Pcurl} with $\uvec{v}_T=\uGT\underline{q}_T$. Using the inclusion $\Image\uGT\subset\Ker\cCT$ (see Remark \ref{rem:ker.cCT}) and the relation $\trFt\big(\uGF\underline{q}_F\big)=\cGF\underline{q}_F$ valid for all $F\in\FT$ (see Proposition \ref{prop:trFt}), we obtain
  \[
  \begin{aligned}
  \int_T\Pcurl\big(\uGT\underline{q}_T\big)\cdot\CURL\bvec{z}_T
  ={}& -\sum_{F\in\FT}\omega_{TF}\int_F\cGF\underline{q}_F\cdot(\bvec{z}_T\times\normal_F)\\
  ={}&\int_T\cGT\underline{q}_T\cdot\CURL\bvec{z}_T,
  \end{aligned}
  \]
  where the conclusion follows from the link between element and face gradients established in Proposition \ref{prop:link.GT.GF}.  
  By the isomorphism \eqref{eq:iso:CURL} with $\ell=k+1$ and since $\cGoly{k+1}(T)\subset\NE{k+1}(T)$, this establishes \eqref{eq:rproj.pcurl.uGT} and concludes the proof of \eqref{eq:Pcurl.uGT=cGT}.
  \medskip\\
  \underline{2. \emph{Proof of \eqref{eq:Pdiv.uCT=cCT}}.}
  The second projection property in \eqref{eq:proj.Pdiv} ensures that 
  \[
  \Gcproj{k}{T}\big[\Pdiv\big(\uCT\uvec{v}_T\big)\big] = \Gcproj{k}{T}\big(\cCT\uvec{v}_T\big).
  \]
  As before, it therefore remains to analyse the projections on $\Goly{k}(T)$.
  Apply \eqref{eq:Pdiv:ibp} to $\uvec{w}_T=\uCT\uvec{v}_T$ and a generic $r_T\in\Poly{k+1}(T)$, and use the inclusion $\Image\uCT\subset\Ker\DT$ (see Proposition \ref{prop:ImCT.KerDT}) to get
  \[
  \int_T\Pdiv\big(\uCT\uvec{v}_T\big)\cdot\GRAD r_T
  = \sum_{F\in\FT}\omega_{TF}\int_F \CF\uvec{v}_F~r_T
  = \int_T \cCT\uvec{v}_T\cdot\GRAD r_T,
  \]
  where the conclusion is obtained applying the link between element and face curls of Proposition \ref{prop:cCT.CF}.
  This yields $\Gproj{k}{T}\big[\Pdiv\big(\uCT\uvec{v}_T\big)\big] = \Gproj{k}{T}\big(\cCT\uvec{v}_T\big)$, proving \eqref{eq:Pdiv.uCT=cCT}.\qed
\end{proof}

\begin{corollary}[Bounds on discrete gradients and curl]\label{cor:bound.full.op}
  For all $F\in\Fh$, it holds
  \begin{equation}
    \label{eq:bound.cGF}
    \norm[\vLeb(F)]{\cGF\underline{q}_F}^2+\sum_{E\in\EF}h_E\norm[\Leb(E)]{G^k_Eq_E}^2\lesssim\tnorm[\CURL,F]{\uGF\underline{q}_F}
    \quad\forall\underline{q}_F\in\Xgrad{F}.
  \end{equation}
  For all $T\in\Th$, it holds
  \begin{multline}
    \label{eq:bound.cGT}
    \norm[\vLeb(T)]{\cGT\underline{q}_T}^2+\sum_{F\in\FT}h_F\norm[\vLeb(F)]{\cGF\underline{q}_F}^2+\sum_{E\in\ET}\hspace{-0.5ex}h_E^2\norm[\Leb(E)]{G^k_Eq_E}^2    \lesssim\tnorm[\CURL,T]{\uGT\underline{q}_T}\\
    \forall\underline{q}_T\in\Xgrad{T},
  \end{multline}
  \begin{equation}
    \label{eq:bound:cCT}
    \norm[\vLeb(T)]{\cCT\uvec{v}_T}^2+\sum_{F\in\FT}h_F\norm[\vLeb(F)]{\CF\uvec{v}_F}^2
    \lesssim\tnorm[\DIV,T]{\uCT\uvec{v}_T}\qquad\forall\uvec{v}_T\in\Xcurl{T}.
  \end{equation}
\end{corollary}

\begin{proof}
  The definitions of $\tnorm[\CURL,F]{{\cdot}}$, $\tnorm[\CURL,T]{{\cdot}}$, $\uGF$ and $\uGT$ show that the edge gradient contributions in the left-hand sides of \eqref{eq:bound.cGF} and \eqref{eq:bound.cGT} are bounded by the corresponding right-hand sides. 
  To bound the face and element gradient contributions in the left-hand sides of \eqref{eq:bound.cGF} and \eqref{eq:bound.cGT}, simply apply \eqref{eq:bound.trFt.Pcurl} to $\uvec{v}_T=\uGT\underline{q}_T$ and use \eqref{eq:trFt:GF} along with \eqref{eq:Pcurl.uGT=cGT}.
  The estimate \eqref{eq:bound:cCT} is established in a similar way, using \eqref{eq:Pdiv.uCT=cCT}.\qed
\end{proof}


\section{Poincar\'e inequalities}\label{sec:poincare}

In this section we state and prove Poincar\'e-type inequalities for the operators in the DDR sequence.
Notice that we consider here the complex without boundary conditions, but one could alternatively consider the complex with (homogeneous) boundary conditions, for which similar inequalities are expected to hold.
The details are left for a future work.

\subsection{Discrete Poincar\'e inequalities}

\begin{theorem}[Poincar\'e inequality for the gradient]\label{thm:poincare.grad}
  Let $\underline{q}_h\in\Xgrad{h}$ be such that
  \begin{equation}\label{eq:orthogonality:unit}
  \sum_{T\in\Th}\int_T\Pgrad \underline{q}_T=0.
  \end{equation}
  Then, there exists a real number $C>0$ independent of $h$ and $\underline{q}_h$, and depending only on $\Omega$, $k$, and the mesh regularity parameter, such that
  \begin{equation}\label{eq:poincare.grad}
    \tnorm[\GRAD,h]{\underline{q}_h}\le C\tnorm[\CURL,h]{\uGh\underline{q}_h}.
  \end{equation}
\end{theorem}

\begin{proof}
  See Section \ref{sec:poincare:proof.poincare.grad}.
\end{proof}

\begin{remark}[Condition \eqref{eq:orthogonality:unit}]
  For $k\ge 1$, owing to \eqref{eq:Pgrad:projection} the condition \eqref{eq:orthogonality:unit} is equivalent to
  \[
  \sum_{T\in\Th}\int_T q_T=0.
  \]
  For $k=0$, the absence of element components means that translating \eqref{eq:orthogonality:unit} in terms of the components of $\underline{q}_h$ is less straightforward.
  Assuming that, for all $Y\in\Th\cup\Fh$, $Y$ is star-shaped with respect to $\bvec{x}_Y\coloneq\frac{1}{|Y|}\int_Y\bvec{x}$, and that this point is selected in the definition of the complements in Section \ref{sec:polynomial.spaces}, condition \eqref{eq:orthogonality:unit} becomes
  \[
  \sum_{T\in\Th}\sum_{F\in\FT}\sum_{E\in\EF}|P_{TFE}|\lproj{0}{E}q_E
  = \frac12\sum_{T\in\Th}\sum_{F\in\FT}\sum_{E\in\EF}\sum_{V\in\VE}|P_{TFE}| q_{\EF}(\bvec{x}_V)
  = 0,
  \]
  where, for any mesh element $T\in\Th$, face $F\in\FT$, and edge $E\in\EF$ of vertices $V_1$ and $V_2$, $P_{TFE}$ is the tetrahedron of vertices $\bvec{x}_T$, $\bvec{x}_F$, $\bvec{x}_{V_1}$, and $\bvec{x}_{V_2}$.
  This corresponds to the construction on the dual barycentric mesh of \cite[Section 4.1]{Bonelle.Ern:14}.
  \smallskip
  
  We also notice, in passing, that condition \eqref{eq:orthogonality:unit} is not needed when considering the subspace of $\Xgrad{h}$ with homogeneous boundary conditions.
\end{remark}

For the sake of completeness, we state in what follows Poincar\'e inequalities for the curl and the divergence that are easy consequences of the results of \cite{Di-Pietro.Droniou:20*1}.

\begin{theorem}[Poincar\'e inequality for the curl]\label{thm:poincare.curl}
  Denote by $(b_0,b_1,b_2,b_3)$ the Betti numbers of $\Omega$ (with $b_0=1$ and $b_3=0$) and assume $b_2=0$.
  Let $(\Ker\uCh)^\perp$ be the orthogonal of $\Ker\uCh$ in $\Xcurl{h}$ for an inner product whose norm is, uniformly in $h$, equivalent to $\tnorm[\CURL,h]{{\cdot}}$.
  Then, $\uCh:(\Ker\uCh)^\perp\to\Ker\Dh$ is an isomorphism.
  Further assuming that $b_1=0$, there exists $C>0$ independent of $h$, and depending only on $\Omega$, $k$ and the mesh regularity parameter, such that
  \begin{equation}\label{eq:poincare.curl}
    \tnorm[\CURL,h]{\uvec{v}_h}\le C\tnorm[\DIV,h]{\uCh\uvec{v}_h}\qquad\forall\uvec{v}_h\in (\Ker\uCh)^\perp.
  \end{equation}
\end{theorem}

\begin{proof}
  The isomorphism property is a consequence of \eqref{eq:Im.uCh.equal.Ker.Dh}.
  In order to prove the Poincar\'e inequality \eqref{eq:poincare.curl}, combine \cite[Theorem 20]{Di-Pietro.Droniou:20*1}  with \cite[Proposition 16]{Di-Pietro.Droniou:20*1} (which requires the additional assumption $b_1=0$) and the norm equivalence \eqref{eq:equiv.norms}.\qed
\end{proof}

\begin{theorem}[Poincar\'e inequality for the divergence]\label{thm:poincare.div}
  Let $(\Ker\Dh)^\perp$ be the orthogonal of $\Ker\Dh$ in $\Xdiv{h}$ for an inner product whose norm is, uniformly in $h$, equivalent to $\tnorm[\DIV,h]{{\cdot}}$.
  Then, $\Dh:(\Ker\Dh)^\perp\to\Poly{k}(\Th)$ is an isomorphism and there exists $C>0$ independent of $h$, and depending only on $\Omega$, $k$ and the mesh regularity parameter, such that
  \begin{equation}\label{eq:poincare.div}
    \tnorm[\DIV,h]{\uvec{w}_h}\le C\norm[\Leb(\Omega)]{\Dh\uvec{w}_h}\qquad\forall\uvec{w}_h\in(\Ker\Dh)^\perp.
  \end{equation}
\end{theorem}

\begin{proof}
  The isomorphism property is a consequence of \eqref{eq:Im.Dh.equal.Pk.Th}.
  The Poincar\'e inequality \eqref{eq:poincare.div} follows from \cite[Theorem 18]{Di-Pietro.Droniou:20*1} accounting for the norm equivalence \eqref{eq:equiv.norms}.\qed
\end{proof}

\subsection{Proof of the discrete Poincar\'e inequality for the gradient}\label{sec:poincare:proof.poincare.grad}

We first prove a preliminary result, which will also be useful to establish adjoint consistency properties for the discrete gradient operator in Section \ref{sec:adjoint.consistency.grad}.

\begin{lemma}[Estimates on local $\Sob{1}$-seminorms of potentials]\label{lem:jump.traces}
  For all $F\in\Fh$ and all $\underline{q}_F\in\Xgrad{F}$, it holds
  \begin{equation}
    \label{eq:grad.trF.qE}
    \norm[\vLeb(F)]{\GRAD\trF\underline{q}_F}^2\,{+}\sum_{E\in\EF}h_E^{-1}\norm[\Leb(E)]{\trF\underline{q}_F-q_E}^2\lesssim \tnorm[\CURL,F]{\uGF\underline{q}_F}^2.
  \end{equation}
  For all $T\in\Th$ and all $\underline{q}_T\in\Xgrad{T}$, it holds
  \begin{equation}
    \label{eq:grad.Pgrad.trF}
    \norm[\vLeb(T)]{\GRAD\Pgrad\underline{q}_T}^2\,{+}\sum_{F\in\FT}h_F^{-1}\norm[\Leb(F)]{\Pgrad\underline{q}_T-\trF\underline{q}_F}^2\lesssim \tnorm[\CURL,T]{\uGT\underline{q}_T}^2.
  \end{equation}
\end{lemma}

\begin{proof}~\\
  \underline{1. \emph{Proof of \eqref{eq:grad.trF.qE}}.}
  Let $\underline{q}_F\in\Xgrad{F}$ and define $A_{q,\partial F}\in\Real$ as the average of $q_{\EF}$ over $\partial F$. Introducing $A_{q,\partial F}=\trF\big(\Igrad{F}A_{q,\partial F}\big)$ (see \eqref{eq:trF:consistency}), using $h_E\simeq h_F$ and $\card(\EF)\lesssim 1$, and invoking a discrete trace inequality on $\trF\big(\underline{q}_F-A_{q,\partial F}\big)$, we have
  \begin{equation}\label{eq:diff.trF.qE}
    \begin{aligned}
      \sum_{E\in\EF}h_E^{-1}\norm[\Leb(E)]{\trF\underline{q}_F-q_E}^2
      &\lesssim\sum_{E\in\EF}h_E^{-1}\norm[\Leb(E)]{q_E-A_{q,\partial F}}^2
      \\
      &\quad
      + h_F^{-2}\norm[\Leb(F)]{\trF\big(\underline{q}_F-A_{q,\partial F}\big)}^2.
    \end{aligned}
  \end{equation}
  Since $q_{\EF}$ is continuous, recalling that $q_E=(q_{\Eh})_{|E}$ for all $E\in\EF$ and using a Poincar\'e--Wirtinger inequality along $\partial F$ followed by the definition \eqref{eq:tnorm.curl} of $\tnorm[\CURL,F]{{\cdot}}$ yields 
  \begin{equation}\label{eq:PW.partialF}
    \sum_{E\in\EF}h_E^{-1}\norm[\Leb(E)]{q_E-A_{q,\partial F}}^2
    \lesssim h_F\sum_{E\in\EF}\norm[\Leb(E)]{\GE q_E}^2\lesssim \tnorm[\CURL,F]{\uGF\underline{q}_F}^2.
  \end{equation}
  We now turn to the second term in \eqref{eq:diff.trF.qE}.
  Using the isomorphism property \eqref{eq:iso:DIV}, we select $\bvec{v}_F\in\cRoly{k+2}(F)$ such that $\DIV_F\bvec{v}_F = {\trF\big(\underline{q}_F-\Igrad{F}A_{q,\partial F}\big)}$.
  By Lemma \ref{lem:norm.isomorphisms} in Appendix \ref{sec:results}, we have 
  \[
  \norm[\vLeb(F)]{\bvec{v}_F}\lesssim h_F\norm[\Leb(F)]{\trF\big(\underline{q}_F-\Igrad{F}A_{q,\partial F}\big)}.
  \]
  The discrete trace inequality of \cite[Lemma 1.32]{Di-Pietro.Droniou:20} and the consistency property \eqref{eq:cGF:consistency} of $\cGF$ then yield
  \begin{gather*}
    \norm[\vLeb(F)]{\bvec{v}_F}+\left(\sum_{E\in\EF}h_E\norm[\Leb(E)]{\bvec{v}_F}^2\right)^{\frac12}
    \lesssim h_F\norm[\Leb(F)]{\trF\big(\underline{q}_F-\Igrad{F}A_{q,\partial F}\big)},
    \\
    \cGF\big(\underline{q}_F-\Igrad{F}A_{q,\partial F}\big) = \cGF\underline{q}_F.
  \end{gather*}
  Hence, applying the definition \eqref{eq:trF} of $\trF$ to $\underline{q}_F-\Igrad{F}A_{q,\partial F}\in \Xgrad{F}$, taking $\bvec{v}_F$ above as a test function, and using Cauchy--Schwarz inequalities, we obtain
  \[
  \begin{aligned}
    &\norm[\Leb(F)]{\trF\big(\underline{q}_F-\Igrad{F}A_{q,\partial F}\big)}^2\\
    &\qquad\lesssim  h_F\norm[\vLeb(F)]{\cGF\underline{q}_F}\norm[\Leb(F)]{\trF\big(\underline{q}_F-\Igrad{F}A_{q,\partial F}\big)}\\
    &\qquad\quad+ \left(\sum_{E\in\EF}h_E^{-1}\norm[\Leb(E)]{q_E-A_{q,\partial F}}^2\right)^{\frac12}
    h_F\norm[\Leb(F)]{\trF\big(\underline{q}_F-\Igrad{F}A_{q,\partial F}\big)}.
  \end{aligned}
  \]
  Simplifying and recalling \eqref{eq:bound.cGF} and \eqref{eq:PW.partialF}, we infer $\norm[\Leb(F)]{\trF\big(\underline{q}_F-A_{q,\partial F}\big)}\lesssim  h_F\tnorm[\CURL,F]{\uGF\underline{q}_F}$ which, plugged together with \eqref{eq:PW.partialF} into \eqref{eq:diff.trF.qE}, gives the following estimate on the second term in the left-hand side of \eqref{eq:grad.trF.qE}:
  \begin{equation}\label{eq:estimate.trF-qE}
    \sum_{E\in\EF}h_E^{-1}\norm[\Leb(E)]{\trF\underline{q}_F-q_E}^2
    \lesssim \tnorm[\CURL,F]{\uGF\underline{q}_F}^2.
  \end{equation}

  Integrating by parts the definition \eqref{eq:trF} of $\trF$ applied to a generic $\bvec{v}_F\in\vPoly{k}(F)$ (see Remark \ref{rem:validity.trF}), we have
  \[
  \begin{aligned}
  \int_F\GRAD_F\trF\underline{q}_F\cdot\bvec{v}_F
  ={}& \int_F\cGF\underline{q}_F\cdot\bvec{v}_F\\
  &+ \sum_{E\in\EF}\omega_{FE}\int_E (\trF\underline{q}_F-q_E)(\bvec{v}_F\cdot\normal_{FE}).
  \end{aligned}
  \]
  Making $\bvec{v}_F=\GRAD_F\trF\underline{q}_F$, using Cauchy--Schwarz inequalities, \eqref{eq:bound.cGF}, a discrete trace inequality, and \eqref{eq:estimate.trF-qE} then yields the bound on the first term in the left-hand side of \eqref{eq:grad.trF.qE}.
  \medskip\\
  \underline{2. \emph{Proof of \eqref{eq:grad.Pgrad.trF}}.}
  The ideas are similar to those used to prove \eqref{eq:grad.trF.qE}, but first we need to establish a Poincar\'e--Wirtinger inequality for face potentials (which is not straightforward given their discontinuity). Let
  \[
  A_{q,\partial T}\coloneq\frac{1}{|\partial T|}\sum_{F\in\FT}|F|A_{q,F}\quad\text{with}\quad
  A_{q,F}\coloneq\frac{1}{|F|}\int_F\trF\underline{q}_F
  \]
  denote the average over $\partial T$ of the piecewise polynomial function defined by $(\trF\underline{q}_F)_{F\in\FT}$. We write, using triangle inequalities,
  \begin{equation}\label{eq:Pgrad.trF.diff}
    \begin{aligned}
      &\sum_{F\in\FT}h_F^{-1}\norm[\Leb(F)]{\Pgrad\underline{q}_T-\trF\underline{q}_F}^2
      \\
      &\quad \lesssim \sum_{F\in\FT}h_F^{-1}\norm[\Leb(F)]{\trF\underline{q}_F-A_{q,F}}^2
      \\
      &\qquad +\sum_{F\in\FT}h_F^{-1}\norm[\Leb(F)]{A_{q,F}-A_{q,\partial T}}^2
      \\
      &\qquad +\sum_{F\in\FT}h_F^{-1}\norm[\Leb(F)]{\Pgrad\underline{q}_T-A_{q,\partial T}}^2
      \eqcolon\term_1+\term_2+\term_3.
    \end{aligned}
  \end{equation}
  The first term is estimated using a Poincar\'e--Wirtinger inequality on $\trF\underline{q}_F$ and invoking \eqref{eq:grad.trF.qE} together with the definition \eqref{eq:tnorm.curl} of $\tnorm[\CURL,T]{{\cdot}}$ to get
  \begin{equation}\label{eq:norm.1.T:T1}
    \begin{aligned}
      \term_1&\lesssim \sum_{F\in\FT}h_F^{-1}\left(h_F\norm[\vLeb(F)]{\GRAD_F\trF\underline{q}_F}\right)^2
      \\
      &\lesssim \sum_{F\in\FT}h_F\tnorm[\CURL,F]{\uGF\underline{q}_F}^2
      \\
      &\lesssim\tnorm[\CURL,T]{\uGT\underline{q}_T}^2.
    \end{aligned}
  \end{equation}
  
  Let us turn to the second term in \eqref{eq:Pgrad.trF.diff}. Since $A_{q,\partial T}$ is a weighted average of all $(A_{q,F})_{F\in\FT}$, the bound
  \begin{equation}\label{eq:norm.1.T:T2}
    \term_2\lesssim \tnorm[\CURL,T]{\uGT\underline{q}_T}^2
  \end{equation}
  follows if we prove that, for all $F,F'\in\FT$,
  \begin{equation}\label{eq:AqF1-AqF2}
    h_F^{-1}\norm[\Leb(F)]{A_{q,F} - A_{q,F'}}^2
    = h_F^{-1}|F|\,|A_{q,F} - A_{q,F'}|^2
    \lesssim\tnorm[\CURL,T]{\uGT\underline{q}_T}^2.
  \end{equation}
  Creating a sequence $(F=F_0,F_1,\ldots,F_m=F')$ of faces in $\FT$ such that, for all $i=0,\ldots,m-1$, the faces $F_i,F_{i+1}$ share an edge $E_i$, inserting
  \begin{multline*}
    -\lproj{0}{E_0}q_{E_0}+\sum_{i=0}^{m-2}\left[\left(\lproj{0}{E_i}q_{E_i} - A_{q,F_{i+1}}\right)-\left(\lproj{0}{E_{i+1}}q_{E_{i+1}}-A_{q,F_{i+1}}\right)\right]
    \\
    +\lproj{0}{E_{m-1}}q_{E_{m-1}}
    = 0
  \end{multline*}
  into $|A_{q,F}-A_{q,F'}|$, using triangle inequalities and the fact that $h_{F_i}\simeq h_{F_{i+1}}$ and $|F_{i}|\simeq |F_{i+1}|$ for all $i=0,\ldots,m-1$ by mesh regularity, and recalling the definition \eqref{eq:tnorm.curl} of $\tnorm[\CURL,F]{{\cdot}}$, \eqref{eq:AqF1-AqF2} is a consequence of
  \begin{equation}\label{eq:AF.AE}
    \forall F\in\FT,\quad
    h_F^{-1}|F|\,|A_{q,F}-\lproj{0}{E}q_E|^2\lesssim h_F\tnorm[\CURL,F]{\uGF\underline{q}_F}^2
    \quad\forall E\in\EF.
  \end{equation}
  To prove this relation, we write
  \begin{align*}
    \norm[\Leb(E)]{A_{q,F}-\lproj{0}{E}q_E}^2
    \le{}&\norm[\Leb(E)]{A_{q,F}-q_E}^2
    \\
    \lesssim{}&h_F^{-1}\norm[\Leb(F)]{A_{q,F} - \trF\underline{q}_F}^2+\norm[\Leb(E)]{\trF\underline{q}_F - q_E}^2
    \\
    \lesssim{}& h_F\tnorm[\CURL,F]{\uGF\underline{q}_F}^2,
  \end{align*}
  where the first inequality comes from the $\Leb$-boundedness of $\lproj{0}{E}$,
  the second inequality is obtained introducing $\trF\underline{q}_F$ and using a triangle inequality together with a discrete trace inequality,
  while \eqref{eq:grad.trF.qE} together with the same arguments that lead to \eqref{eq:norm.1.T:T1} yield the conclusion.
  The relation \eqref{eq:AF.AE} follows noticing that $|F|\simeq h_F|E|$, so that $|F|\,|A_{q,F}-\lproj{0}{E}q_E|^2\simeq h_F\norm[\Leb(E)]{A_{q,F}-\lproj{0}{E}q_E}^2$.
  This concludes the proof of \eqref{eq:AqF1-AqF2}, hence of \eqref{eq:norm.1.T:T2}.
  
  Finally, for $\term_3$, we apply the definition \eqref{eq:PgradT} of $\Pgrad\big(\underline{q}_T-\Igrad{T} A_{q,\partial T}\big)$ with $\bvec{v}_T\in \cRoly{k+2}(T)$ such that 
  \[
  \DIV\bvec{v}_T=\Pgrad\big(\underline{q}_T-\Igrad{T} A_{q,\partial T}\big)
  \]
  and $\norm[\vLeb(T)]{\bvec{v}_T}\lesssim h_T\norm[\Leb(T)]{\Pgrad(\underline{q}_T-\Igrad{T} A_{q,\partial T})}$, see Lemma \ref{lem:norm.isomorphisms}. Using the consistency properties \eqref{eq:Pgrad:polynomial.consistency} of $\Pgrad$, \eqref{eq:cGT:consistency} of $\cGT$ and \eqref{eq:trF:consistency} of $\trF$, and a discrete trace inequality, this gives
  \begin{align}
    \norm[\Leb(T)]{\Pgrad\underline{q}_T- A_{q,\partial T}}
    \lesssim{}& h_T\norm[\vLeb(T)]{\cGT\underline{q}_T}\nonumber\\
    & +h_T\sum_{F\in\FT}h_F^{-\frac12}\norm[\Leb(F)]{\trF\underline{q}_F-A_{q,\partial T}}\nonumber\\
    \lesssim{}&h_T\tnorm[\CURL,T]{\uGT\underline{q}_T} + h_T\left(\term_1^{\frac12}+\term_2^{\frac12}\right),
    \label{eq:Pgrad.trF.diff.2}
  \end{align}
  where the second inequality follows from \eqref{eq:bound.cGT} and a triangle inequality to write 
  \[
  \begin{aligned}
    \sum_{F\in\FT}h_F^{-\frac12}\norm[\Leb(F)]{\trF\underline{q}_F-A_{q,\partial T}}
    &\le \sum_{F\in\FT}h_F^{-\frac12}\norm[\Leb(F)]{\trF\underline{q}_F-A_{q,F}}
    \\
    &\quad +\sum_{F\in\FT}h_F^{-\frac12}\norm[\Leb(F)]{A_{q,F}-A_{q,\partial T}}.
  \end{aligned}
  \]
  Using discrete trace inequalities and the estimates \eqref{eq:norm.1.T:T1} and \eqref{eq:norm.1.T:T2} on $\term_1$ and $\term_2$, \eqref{eq:Pgrad.trF.diff.2} leads to
  \[
  \term_3\lesssim h_T^{-2}\norm[\Leb(T)]{\Pgrad\underline{q}_T- A_{q,\partial T}}^2\lesssim \tnorm[\CURL,T]{\uGT\underline{q}_T}^2.
  \]
  Plugging this bound together with the estimates on $\term_1$ and $\term_2$ into \eqref{eq:Pgrad.trF.diff} concludes the proof of the bound on the second term in the right-hand side of \eqref{eq:grad.Pgrad.trF}.
  To bound the first term in the left-hand side of \eqref{eq:grad.Pgrad.trF}, we proceed as for $\GRAD_F\trF\underline{q}_F$ in Step 1 of this proof, using an integration by parts in the definition \eqref{eq:PgradT} of $\Pgrad\underline{q}_T$ and selecting the test function $\bvec{v}_T=\GRAD\Pgrad\underline{q}_T$ (see Remark \ref{rem:validity.Pgrad}).\qed
\end{proof}

 We are now ready to prove the discrete Poincar\'e inequality for the gradient.

 \begin{proof}[Theorem \ref{thm:poincare.grad}]
   By the orthogonality condition \eqref{eq:orthogonality:unit}, we can apply the discrete Poincar\'e--Wirtinger inequality in Hybrid High-Order spaces \cite[Theorem 6.5]{Di-Pietro.Droniou:20} (with $p=q=2$) to the vector of element- and face-polynomials $((\Pgrad\underline{q}_T)_{T\in\Th},(\trF\underline{q}_T)_{F\in\Fh})$ to get
   \begin{equation}\label{eq:estimate.poincare.Pgrad}
     \begin{aligned}
       &\sum_{T\in\Th}\norm[\Leb(T)]{\Pgrad\underline{q}_T}^2
       \\
       &\quad\lesssim
       \sum_{T\in\Th}\left(
       \norm[\vLeb(T)]{\GRAD \Pgrad\underline{q}_T}^2
       {+} \hspace{-1ex}\sum_{F\in\FT} h_F^{-1}\norm[\Leb(F)]{\Pgrad\underline{q}_T{-}\,\trF\underline{q}_T}^2
       \right)
       \\
       &\quad\lesssim\tnorm[\CURL,h]{\uGh\underline{q}_h}^2,  
     \end{aligned}
   \end{equation}
  where the conclusion is a consequence of \eqref{eq:grad.Pgrad.trF} followed by the definition \eqref{eq:tnorm.curl} of the $\tnorm[\CURL,h]{{\cdot}}$-norm.
  
  Let $T\in\Th$. By definition \eqref{eq:sgradT} of $\mathrm{s}_{\GRAD,T}$ we have
  \begin{equation}\label{eq:estimate.poincare.sT}
  \begin{aligned}
    \mathrm{s}_{\GRAD,T}(\underline{q}_T,\underline{q}_T)
    &=
    \sum_{F\in\FT}h_F\norm[\Leb(F)]{\Pgrad\underline{q}_T-\trF\underline{q}_T}^2
    \\
    &\quad
    +\sum_{E\in\ET}h_E^2\norm[\Leb(E)]{\Pgrad\underline{q}_T-q_E}^2
    \\
    &\lesssim \sum_{F\in\FT}h_F\norm[\Leb(F)]{\Pgrad\underline{q}_T-\trF\underline{q}_T}^2
    \\
    &\quad
    +\sum_{F\in\FT}\sum_{E\in\EF}h_E^2\norm[\Leb(E)]{\trF\underline{q}_T-q_E}^2
    \\
    &\lesssim h_T^2\tnorm[\CURL,T]{\uGT\underline{q}_T}^2,    
  \end{aligned}
  \end{equation}
  where the first inequality follows writing $\sum_{E\in\ET}\bullet=\frac12 \sum_{F\in\FT}\sum_{E\in\EF}\bullet$, introducing $\pm\trF\underline{q}_F$ in the norms and using triangle and discrete trace inequalities, while the conclusion is obtained invoking \eqref{eq:grad.trF.qE}, \eqref{eq:grad.Pgrad.trF}, $h_E\simeq h_F\simeq h_T$ and the definition of $\tnorm[\CURL,T]{\uGT\underline{q}_T}$.
  
  Using $h_T\lesssim 1$, summing \eqref{eq:estimate.poincare.sT} over $T\in\Th$, and adding the resulting estimate to \eqref{eq:estimate.poincare.Pgrad} we infer that $\norm[\GRAD,h]{\underline{q}_h}\lesssim \tnorm[\CURL,h]{\uGh\underline{q}_h}$. The Poincar\'e inequality \eqref{eq:poincare.grad} then follows from the norm equivalence \eqref{eq:equiv.norms}.\qed
\end{proof}


\section{Consistency results}\label{sec:consistency.results}


\subsection{Primal consistency}

In this section we state consistency results for the discrete potentials, vector calculus operators, stabilisation bilinear forms, and discrete $\Leb$-products.
Because of the nature of the interpolator on $\Xcurl{T}$ (which requires higher regularity of functions), we introduce the following notation: For $T\in\Th$ and $\bvec{v}\in \vSob{\max(k+1,2)}(T)$,
\begin{equation}\label{eq:def.Hkp12}
\seminorm[\vSob{(k+1,2)}(T)]{\bvec{v}}\coloneq\left\{
\begin{array}{l@{\quad}l}\seminorm[\vSob{1}(T)]{\bvec{v}}+h_T\seminorm[\vSob{2}(T)]{\bvec{v}}&\mbox{ if $k=0$},\\
\seminorm[\vSob{k+1}(T)]{\bvec{v}}&\mbox{ if $k\ge 1$}.
\end{array}\right.
\end{equation}
The corresponding global broken seminorm $\seminorm[\vSob{(k+1,2)}(\Th)]{{\cdot}}$ is such that, for all $\bvec{v}\in\vSob{(k+1,2)}(\Th)$,
\[
\seminorm[\vSob{(k+1,2)}(\Th)]{\bvec{v}}\coloneq\left(\sum_{T\in\Th}\seminorm[\vSob{(k+1,2)}(T)]{\bvec{v}}^2\right)^{\nicefrac12}.
\]
The proofs of the following theorems are postponed to Section \ref{sec:proof.primal.consistency}.

\begin{theorem}[Consistency of the potential reconstructions]\label{thm:approx.PTIT}
  It holds, for all $T\in\Th$,
  \begin{alignat}{2} \label{eq:approx.PgradIgrad}
    \norm[\Leb(T)]{\Pgrad\big(\Igrad{T} q\big) - q}&\lesssim h_T^{k+2}\seminorm[\Sob{k+2}(T)]{q}&\quad&\forall q\in \Sob{k+2}(T),
    \\
    \label{eq:approx.PcurlIcurl}
    \norm[\vLeb(T)]{\Pcurl\big(\Icurl{T}\bvec{v}\big) - \bvec{v}}&\lesssim h_T^{k+1}\seminorm[\vSob{(k+1,2)}(T)]{\bvec{v}}&\quad&\forall \bvec{v}\in \vSob{\max(k+1,2)}(T),
    \\
    \label{eq:approx.PdivIdiv}
    \norm[\vLeb(T)]{\Pdiv\big(\Idiv{T}\bvec{w}\big) - \bvec{w}}&\lesssim h_T^{k+1}\seminorm[\vSob{k+1}(T)]{\bvec{w}}&\quad&\forall \bvec{w}\in \vSob{k+1}(T).
  \end{alignat}
\end{theorem}

\begin{theorem}[Primal consistency of the discrete vector calculus operators]\label{thm:primal.consistency}
  It holds, for all $T\in\Th$, 
  \begin{alignat}{2}
    \norm[\vLeb(T)]{\cGT\big(\Igrad{T} q\big) {-} \GRAD q}
    \lesssim{}& h_T^{k+1}\seminorm[\Sob{k+2}(T)]{q}
    \qquad \forall q\in\Sob{k+2}(T),\label{eq:approx:cGT.Igrad}
    \\ 
    \norm[\vLeb(T)]{\cCT\big(\Icurl{T}\bvec{v}\big) {-} \CURL\bvec{v}}
    \lesssim{}& h_T^{k+1}\seminorm[\vSob{k+1}(T)]{\CURL\bvec{v}}\nonumber\\
    &\qquad 
      \forall\bvec{v}\in\vSob{2}(T) \mbox{ s.t. } \CURL\bvec{v}\in\vSob{k+1}(T),\label{eq:approx:cCT.Icurl}
    \\
    \norm[\Leb(T)]{\DT\big(\Idiv{T}\bvec{w}\big) {-} \DIV\bvec{w}}
    \lesssim{}& h_T^{k+1}\seminorm[\Sob{k+1}(T)]{\DIV\bvec{w}}\nonumber\\
    &\qquad \forall \bvec{w}\in\vSob{1}(T) \mbox{ s.t. } \DIV\bvec{w}\in\Sob{k+1}(T). \label{eq:approx:DT.Idiv}
  \end{alignat}
\end{theorem}

\begin{theorem}[Consistency of stabilisation forms]\label{thm:consistency.sT}
  For all $T\in\Th$, the stabilisation forms defined by \eqref{eq:sgradT}--\eqref{eq:divT} satisfy the following consistency properties:
  \begin{alignat}{2} \label{eq:s.grad.T:consistency}
    \mathrm{s}_{\GRAD,T}(\Igrad{T} q, \Igrad{T} q)^{\frac12}
    &\lesssim h_T^{k+2}\seminorm[\Sob{k+2}(T)]{q}
    &\qquad& \forall q\in \Sob{k+2}(T),
    \\ \label{eq:s.curl.T:consistency}
    \mathrm{s}_{\CURL,T}(\Icurl{T}\bvec{v},\Icurl{T}\bvec{v})^{\frac12}
    &\lesssim h_T^{k+1}\seminorm[\vSob{(k+1,2)}(T)]{\bvec{v}}
    &\qquad& \forall\bvec{v}\in \vSob{\max(k+1,2)}(T),
    \\ \label{eq:s.div.T:consistency}
    \mathrm{s}_{\DIV,T}(\Idiv{T}\bvec{w},\Idiv{T}\bvec{w})^{\frac12}
    &\lesssim h_T^{k+1}\seminorm[\vSob{k+1}(T)]{\bvec{w}}
    &\qquad& \forall\bvec{w}\in \vSob{k+1}(T).
  \end{alignat}
\end{theorem}

The following corollary is a straightforward consequence of Theorems \ref{thm:approx.PTIT} and \ref{thm:consistency.sT}, and of the definitions \eqref{eq:Xgrad:l2.prod}--\eqref{eq:Xdiv:l2.prod} of the discrete $\Leb$-products. Its proof is therefore omitted.

\begin{corollary}[Consistency of discrete $\Leb$-products]
It holds, for all $T\in\Th$,
  \begin{alignat}{2} 
    \left|\int_T q\,\Pgrad \underline{r}_T - (\Igrad{T}q,\underline{r}_T)_{\GRAD,T}\right|&\lesssim h_T^{k+2}\seminorm[\Sob{k+2}(T)]{q}\norm[\GRAD,T]{\underline{r}_T}\nonumber\\
    &\hspace{-3ex}\forall q\in \Sob{k+2}(T)\,,\;\forall\underline{r}_T\in\Xgrad{T},
    \label{eq:consistency.L2.Xgrad}
    \\
    \left|\int_T \bvec{v}\cdot\Pcurl \uvec{\zeta}_T - (\Icurl{T}\bvec{v},\uvec{\zeta}_T)_{\CURL,T}\right|&\lesssim h_T^{k+1}\seminorm[\vSob{(k+1,2)}(T)]{\bvec{v}}\norm[\CURL,T]{\uvec{\zeta}_T}\nonumber\\
    &\hspace{-3ex}\forall \bvec{v}\in \vSob{\max(k+1,2)}(T),\,\forall \uvec{\zeta}_T\in\Xcurl{T},
    \label{eq:consistency.L2.Xcurl}
    \\
    \left|\int_T \bvec{w}\cdot\Pdiv \uvec{\xi}_T - (\Idiv{T}\bvec{w},\uvec{\xi}_T)_{\DIV,T}\right|&\lesssim h_T^{k+1}\seminorm[\vSob{k+1}(T)]{\bvec{w}}\norm[\DIV,T]{\uvec{\xi}_T}\nonumber\\
    &\hspace{-3ex}\forall \bvec{w}\in \vSob{k+1}(T),\,\forall\uvec{\xi}_T\in\Xdiv{T}.
    \label{eq:consistency.L2.Xdiv}
    \end{alignat}
\end{corollary}


\subsection{Adjoint consistency}

Whenever a (formal) integration by parts is used to write the weak formulation of a PDE problem underpinning its discretisation, a form of adjoint consistency is required in the convergence analysis.
We state here the adjoint consistency of the operators in the DDR sequence \eqref{eq:global.sequence.3D}.
Since this sequence does not incorporate boundary conditions, the corresponding adjoint consistency will be based on essential (homogeneous) boundary conditions.
The regularity requirements will be expressed in terms of the broken Sobolev spaces and norms such that, for any $\ell\ge 1$,
\[
\begin{gathered}
  \Sob{\ell}(\Th)\coloneq\left\{
  g\in \Leb(\Omega)\st\text{$g_{|T}\in \Sob{\ell}(T)$ for all $T\in\Th$}
  \right\}
  \\
  \text{
    and
  $\seminorm[\Sob{\ell}(\Th)]{g}\coloneq\left(
  \sum_{T\in\Th}\seminorm[\Sob{\ell}(T)]{g_{|T}}^2
  \right)^{\frac12}$.%
  }
\end{gathered}
\]
The corresponding seminorms for vector-valued functions are denoted using boldface letters, as usual.
We denote in what follows by $\Sob{1}_0(\Omega)$, $\HDdiv{\Omega}$, and $\HDcurl{\Omega}$ the subspaces of $\Sob{1}(\Omega)$, $\Hdiv{\Omega}$, and $\Hcurl{\Omega}$ spanned by functions whose trace, normal trace, and tangential trace vanish on the boundary $\partial\Omega$ of $\Omega$, respectively.

\begin{theorem}[Adjoint consistency for the gradient]\label{thm:adjoint.consistency.grad}
  Define the gradient adjoint consistency error $\dEgrad:\big(\vC{0}(\overline{\Omega})\cap\HDdiv{\Omega}\big)\times\Xgrad{h}\to\Real$ by:
  For all $(\bvec{v},\underline{q}_h)\in\big(\vC{0}(\overline{\Omega})\cap\HDdiv{\Omega}\big)\times\Xgrad{h}$,
  \[
  \dEgrad(\bvec{v},\underline{q}_h)
  \coloneq \sum_{T\in\Th}\left[
    (\Icurl{T}\bvec{v}_{|T},\uGT\underline{q}_T)_{\CURL,T} + \int_T\DIV\bvec{v}~\Pgrad\underline{q}_T
    \right].
  \]
  Then, it holds, for all $\bvec{v}\in \vC{0}(\overline{\Omega})\cap\HDdiv{\Omega}$ such that $\bvec{v}\in \vSob{\max(k+1,2)}(\Th)$ and all $\underline{q}_h\in\Xgrad{h}$,
  \begin{equation} \label{eq:dEgrad:consistency}
    |\dEgrad(\bvec{v},\underline{q}_h)|
    \lesssim h^{k+1}\seminorm[\vSob{(k+1,2)}(\Th)]{\bvec{v}}
    \norm[\CURL,h]{\uGh\underline{q}_h}.
  \end{equation}  
\end{theorem}

\begin{proof}
  See Section \ref{sec:adjoint.consistency.grad}.
\end{proof}

\begin{theorem}[Adjoint consistency for the curl]\label{thm:adjoint.consistency.curl}
  Define the curl adjoint consistency error $\dEcurl:\big(\vC{0}(\overline{\Omega})\cap\HDcurl{\Omega}\big)\times \Xcurl{h}\to\Real$ by:
  For all $(\bvec{w},\uvec{v}_h)\in\big(\vC{0}(\overline{\Omega})\cap\HDcurl{\Omega}\big)\times\Xcurl{h}$,
  \begin{equation}\label{eq:dEcurl}
    \dEcurl(\bvec{w},\uvec{v}_h)\coloneq
    \sum_{T\in\Th}\left[
      (\Idiv{T}\bvec{w}_{|T},\uCT\uvec{v}_T)_{\DIV,T}
      - \int_T\CURL\bvec{w}\cdot\Pcurl\uvec{v}_T
      \right].
  \end{equation}
  Then, for all $\bvec{w}\in \vC{0}(\overline{\Omega})\cap\HDcurl{\Omega}$ such that $\bvec{w}\in \vSob{k+2}(\Th)$ and all $\uvec{v}_h\in\Xcurl{h}$,
  \begin{equation}\label{eq:dEcurl:consistency}
    |\dEcurl(\bvec{w},\uvec{v}_h)|
    \lesssim h^{k+1}\left(
    \seminorm[\vSob{k+1}(\Th)]{\bvec{w}}
    +\seminorm[\vSob{k+2}(\Th)]{\bvec{w}}
    \right)\left(
    \norm[\CURL,h]{\uvec{v}_h}
    +\norm[\DIV,h]{\uCh\uvec{v}_h}
    \right).
  \end{equation}
\end{theorem}

\begin{proof}
  See Section \ref{sec:consistency.adjoint.curl}.
\end{proof}

\begin{theorem}[Adjoint consistency for the divergence]\label{thm:adjoint.consistency.div}
  Define the divergence adjoint consistency error $\dEdiv:\big(\rC{0}(\overline{\Omega})\cap H_0^1(\Omega)\big)\times\Xdiv{h}\to\Real$ by:
  For all $(q,\uvec{v}_h)\in\big(\rC{0}(\overline{\Omega})\cap H_0^1(\Omega)\big)\times\Xdiv{h}$,
  \begin{equation} \label{eq:dEdiv}
    \dEdiv(q,\uvec{v}_h)\coloneq
    \int_\Omega \lproj{k}{h} q~\Dh\uvec{v}_h
    + \sum_{T\in\Th}\int_\Omega\GRAD q\cdot\Pdiv\uvec{v}_T.
  \end{equation}
  Then, for all $q\in \rC{0}(\overline{\Omega})\cap H_0^1(\Omega)$ such that $q\in \Sob{k+2}(\Th)$ and all $\uvec{v}_h\in\Xdiv{h}$,
  \begin{equation} \label{eq:dEdiv:consistency}
    |\dEdiv(q,\uvec{v}_h)|
    \lesssim h^{k+1}\seminorm[\Sob{k+2}(\Th)]{q}\norm[\DIV,h]{\uvec{v}_h}.
  \end{equation}
\end{theorem}

\begin{proof}
  See Section \ref{sec:adjoint.consistency.div}.
\end{proof}


\subsection{Proof of the primal consistency}\label{sec:proof.primal.consistency}

\begin{proof}[Theorem \ref{thm:approx.PTIT}]
  Let us start with \eqref{eq:approx.PgradIgrad}. Since $\Sob{2}(T)\subset \rC{0}(\overline{T})$, the mapping $\Pgrad\circ\Igrad{T}:\Sob{2}(T)\to \Poly{k+1}(T)$ is well-defined and, owing to \eqref{eq:Pgrad:polynomial.consistency}, it is a projector. Moreover, combining \eqref{eq:bound.Igrad} and \eqref{eq:bound.trF.Pgrad}, it satisfies the $\Leb(T)$-boundedness
  \[
  \norm[\Leb(T)]{\Pgrad\big(\Igrad{T} q\big)}\lesssim \norm[\Leb(T)]{q}+h_T\seminorm[\Sob{1}(T)]{q}+h_T^2\seminorm[\Sob{2}(T)]{q}
  \quad\forall q\in \Sob{2}(T).
  \]
  The approximation property \eqref{eq:approx.PgradIgrad} is thus a direct consequence of \cite[Lemma 1.43]{Di-Pietro.Droniou:20}.
  The proofs of \eqref{eq:approx.PcurlIcurl} (for $k\ge 1$) and \eqref{eq:approx.PdivIdiv} are similar, using the fact that the considered operators are projectors onto $\vPoly{k}(T)$ (see \eqref{eq:Pcurl:polynomial.consistency} and \eqref{eq:Pdiv:polynomial.consistency}) and invoking Proposition \ref{prop:bound.P} and Lemma \ref{lem:bound.I} to establish their $\Leb$-boundedness.
  In the case $k=0$, since $\Pcurlz\circ\Icurlz{T}$ requires the $\vSob{2}$-regularity of its argument, with $2>k+1$, \eqref{eq:approx.PcurlIcurl} cannot be deduced directly from \cite[Lemma 1.43]{Di-Pietro.Droniou:20}. However, using the bounds \eqref{eq:bound.trFt.Pcurl} and \eqref{eq:bound.Icurl} a direct proof can be done by introducing $\vlproj{0}{T}\bvec{v}=\Pcurlz\big(\Icurlz{T} \vlproj{0}{T}\bvec{v}\big)$:
  \[
  \begin{aligned}
    &\norm[\vLeb(T)]{\Pcurlz\big(\Icurlz{T}\bvec{v}\big) - \bvec{v}}
    \\
    &\quad
    \le\norm[\vLeb(T)]{\Pcurlz\big[\Icurlz{T}(\bvec{v}-\vlproj{0}{T}\bvec{v})\big]}
    +\norm[\vLeb(T)]{\vlproj{0}{T}\bvec{v}-\bvec{v}}\\
    &\quad
    \lesssim\norm[\vLeb(T)]{\bvec{v}-\vlproj{0}{T}\bvec{v}}+h_T\seminorm[\vSob{1}(T)]{\bvec{v}-\vlproj{0}{T}\bvec{v}}
    +h_T^2\seminorm[\vSob{2}(T)]{\bvec{v}-\vlproj{0}{T}\bvec{v}},    
  \end{aligned}
  \]
  and \eqref{eq:approx.PcurlIcurl} follows using the approximation properties of $\vlproj{0}{T}$, the fact that the $\vSob{1}(T)$- and $\vSob{2}(T)$-seminorms of $\vlproj{0}{T}\bvec{v}$ vanish, and the definition \eqref{eq:def.Hkp12} of $\seminorm[\vSob{(k+1,2)}(T)]{{\cdot}}$.\qed
\end{proof}


\begin{proof}[Theorem \ref{thm:primal.consistency}]
  Let us prove \eqref{eq:approx:cGT.Igrad}.
  For any $\underline{q}_T\in\Xgrad{T}$, taking $\bvec{w}_T = \cGT\underline{q}_T$ in \eqref{eq:cGT} and using Cauchy--Schwarz inequalities along with discrete inverse and trace inequalities, it is inferred, after simplification,
  \[
  \norm[\vLeb(T)]{\cGT\underline{q}_T}
  \lesssim h_T^{-1}\norm[\Leb(T)]{q_T} + \sum_{F\in\FT} h_F^{-\nicefrac12}\norm[\Leb(F)]{\trF\underline{q}_F}
  \lesssim h_T^{-1}\tnorm[\GRAD,T]{\underline{q}},
  \]
  where the conclusion follows from the estimate on $\trF\underline{q}_F$ in \eqref{eq:bound.trF.Pgrad} and from the definition of $\tnorm[\GRAD,T]{{\cdot}}$.
  As a result, for any $r\in\Sob{2}(T)$, making $\underline{q}_T = \Igrad{T} r$ and invoking \eqref{eq:bound.Igrad}, we infer
  \begin{equation}\label{eq:bound.GT.Igrad}
    \norm[\vLeb(T)]{\cGT\big(\Igrad{T} r\big)}
    \lesssim
    h_T^{-1}\norm[\Leb(T)]{r} + \seminorm[\Sob{1}(T)]{r} + h_T\seminorm[\Sob{2}(T)]{r}.
  \end{equation}
  Letting now $q\in\Sob{k+2}(T)$, we use the polynomial consistency \eqref{eq:cGT:consistency} of $\cGT$ followed by a triangle inequality to write
  \begin{multline*}
  \norm[\vLeb(T)]{\cGT\big(\Igrad{T} q\big) - \GRAD q}
  \\
  \le\norm[\vLeb(T)]{\cGT\big[\Igrad{T} \big(q - \lproj{k+1}{T} q\big)\big]}
  + \norm[\vLeb(T)]{\GRAD\big(\lproj{k+1}{T} q - q\big)}
  \end{multline*}
  and conclude using \eqref{eq:bound.GT.Igrad} with $r = q - \lproj{k+1}{T} q$ for the first term in the right-hand side followed by the approximation properties of $\lproj{k+1}{T}$ (see \cite[Theorem 1.45]{Di-Pietro.Droniou:20}).

  To prove \eqref{eq:approx:cCT.Icurl}, we notice that
  \[
  \cCT\big(\Icurl{T}\bvec{v}\big) = \Pdiv\big[\uCT\big(\Icurl{T}\bvec{v}\big)\big] = \Pdiv\big[\Idiv{T}\big(\CURL\bvec{v}\big)\big]
  \]
  owing to \eqref{eq:Pdiv.uCT=cCT} along with the commutation property \eqref{eq:uCT:commutation}, and conclude using the approximation properties \eqref{eq:approx.PdivIdiv} with $\bvec{w} = \CURL\bvec{v}$.
  
  Finally, \eqref{eq:approx:DT.Idiv} is a straightforward consequence of the commutation property $\DT\big(\Idiv{T}\bvec{w}\big)=\lproj{k}{T}(\DIV\bvec{w})$ stated in \eqref{eq:DT:commutation} together with \cite[Theorem 1.45]{Di-Pietro.Droniou:20}.\qed
\end{proof}

  \begin{remark}[Alternative proof of \eqref{eq:approx:cGT.Igrad}]
    When $q\in\rC{1}(\overline{T})$ is such that $\GRAD q\in\vSob{\max(k+1,2)}(T)$, the proof of \eqref{eq:approx:cGT.Igrad} can be done following similar arguments as for \eqref{eq:approx:cCT.Icurl}, i.e., we write
    \[
    \cGT\big(\Igrad{T} q\big) = \Pcurl\big[\uGT\big(\Igrad{T} q\big)\big] = \Pcurl\big[\Icurl{T}\big(\GRAD q\big)\big]
    \]
    using \eqref{eq:Pcurl.uGT=cGT} followed by \eqref{eq:uGT:commutation}, and conclude using the approximation properties \eqref{eq:approx.PcurlIcurl} with $\bvec{v} = \GRAD q$.
    This argument, however, requires additional regularity on $q$ with respect to the one used above.
  \end{remark}


\begin{proof}[Theorem \ref{thm:consistency.sT}]
We only prove \eqref{eq:s.curl.T:consistency}, the other consistency properties being established in a similar way. Let $\bvec{v}\in \vSob{\max(k+1,2)}(T)$. By the polynomial consistency \eqref{eq:trFt.cons} of $\trFt$ and \eqref{eq:Pcurl:polynomial.consistency} of $\Pcurl$, it is easily checked that, for all $\bvec{z}_T\in\vPoly{k}(T)$ and all $\uvec{w}_T\in\Xcurl{T}$, it holds $\mathrm{s}_{\CURL,T}(\Icurl{T}\bvec{z}_T,\uvec{w}_T)=0$. Applying this with $\bvec{z}_T=\vlproj{k}{T}\bvec{v}$ we infer
\[
\begin{aligned}
  \mathrm{s}_{\CURL,T}(\Icurl{T}\bvec{v},\Icurl{T}\bvec{v})
  &=\mathrm{s}_{\CURL,T}(\Icurl{T}(\bvec{v}-\vlproj{k}{T}\bvec{v}),\Icurl{T}(\bvec{v}-\vlproj{k}{T}\bvec{v}))
  \\
  &\lesssim \tnorm[\CURL,T]{\Icurl{T}(\bvec{v}-\vlproj{k}{T}\bvec{v})}^2,
\end{aligned}
\]
the conclusion following from the definition of $\norm[\CURL,T]{{\cdot}}$ and the norm equivalence \eqref{eq:equiv.norms}. Invoking then \eqref{eq:bound.Icurl} we infer
\[
\begin{aligned}
  \mathrm{s}_{\CURL,T}(\Icurl{T}\bvec{v},\Icurl{T}\bvec{v})^{\frac12}
  &\lesssim \norm[\vLeb(T)]{\bvec{v}-\vlproj{k}{T}\bvec{v}}
  \\
  &\quad
  + h_T\seminorm[\vSob{1}(T)]{\bvec{v}-\vlproj{k}{T}\bvec{v}}+h_T^2\seminorm[\vSob{2}(T)]{\bvec{v}-\vlproj{k}{T}\bvec{v}},
\end{aligned}
\]
and the estimate \eqref{eq:s.curl.T:consistency} follows from the approximation properties of $\vlproj{k}{T}$, see \cite[Theorem 1.45]{Di-Pietro.Droniou:20}, and the definition \eqref{eq:def.Hkp12} of $\seminorm[\vSob{(k+1,2)}(T)]{{\cdot}}$,
using in the case $k=0$ the same arguments as in the proof of Theorem \ref{thm:approx.PTIT}.\qed
\end{proof}


\subsection{Proof of the adjoint consistency for the gradient}\label{sec:adjoint.consistency.grad}

\begin{proof}[Theorem \ref{thm:adjoint.consistency.grad}.]
  It holds, by definition \eqref{eq:Xcurl:l2.prod} of the local discrete $\Leb$-product in $\Xcurl{h}$ and \eqref{eq:Pcurl.uGT=cGT},
  \begin{equation}\label{eq:dEgrad.explicit}
    \begin{aligned}
    \dEgrad(\bvec{v},\underline{q}_h)    
    = \sum_{T\in\Th}\bigg[
      &\int_T \Pcurl\big(\Icurl{T}\bvec{v}\big)\cdot\cGT\underline{q}_T
      \\
      &+ \mathrm{s}_{\CURL,T}(\Icurl{T} \bvec{v}_{|T},\uGT\underline{q}_T)
      + \int_T\DIV\bvec{v}~\Pgrad\underline{q}_T
      \bigg].
    \end{aligned}
  \end{equation}
  Using Remark \ref{rem:validity.Pgrad}, we have, for all $\bvec{w}_T\in\vPoly{k}(T)$,
  \[
  \begin{aligned}
    \int_T\Pgrad\underline{q}_T~\DIV\bvec{w}_T
    + \int_T\cGT\underline{q}_T\cdot\bvec{w}_T
    - \sum_{F\in\FT}\omega_{TF}\int_F\trF\underline{q}_F(\bvec{w}_T\cdot\normal_F)
    = 0.
  \end{aligned}
  \]
  Subtracting this quantity from \eqref{eq:dEgrad.explicit}, we obtain
  \begin{equation*}
    \begin{aligned}
      &\dEgrad(\bvec{v},\underline{q}_h)
      \\
      &\quad=
      \sum_{T\in\Th}\bigg[
        \int_T \left(\Pcurl\big(\Icurl{T}\bvec{v}\big)-\bvec{w}_T\right)\cdot\cGT\underline{q}_T
        +\mathrm{s}_{\CURL,T}(\Icurl{T} \bvec{v}_{|T},\uGT\underline{q}_T)
        \bigg]
      \\
      &\qquad
      + \sum_{T\in\Th}\bigg[
        \int_T\DIV(\bvec{v}-\bvec{w}_T)\Pgrad\underline{q}_T      
        +\sum_{F\in\FT}\omega_{TF}\int_F (\bvec{w}_T-\bvec{v})\cdot\normal_F \trF\underline{q}_F
        \bigg],
    \end{aligned}
  \end{equation*}
  where $\bvec{v}$ is introduced into the boundary term by single-valuedness of the discrete trace, and using $\bvec{v}_{|F}\cdot\normal_F=0$ whenever $F\subset\partial\Omega$.
  Integrating by parts the third term in the right-hand side of the above expression, we obtain
  \begin{equation}\label{eq:Ediv.recast}
    \begin{aligned}
    &\dEgrad(\bvec{v},\underline{q}_h)
    \\
    &{}= \sum_{T\in\Th}\bigg[
      \int_T \left(\Pcurl\big(\Icurl{T}\bvec{v}\big)-\bvec{w}_T\right)\cdot\cGT\underline{q}_T
      + \mathrm{s}_{\CURL,T}(\Icurl{T} \bvec{v}_{|T},\uGT\underline{q}_T)
      \bigg]
    \\
    &\quad
    + \sum_{T\in\Th}
      \begin{aligned}[t]
        \bigg[
          &- \int_T(\bvec{v}-\bvec{w}_T)\cdot\GRAD\Pgrad\underline{q}_T
          \\
          &+ \sum_{F\in\FT}\omega_{TF}\int_F (\bvec{w}_T-\bvec{v})\cdot\normal_F (\trF\underline{q}_F-\Pgrad\underline{q}_T)
        \bigg].
      \end{aligned}
    \end{aligned}
  \end{equation}
  We set $\bvec{w}_T=\vlproj{k}{T}\bvec{v}$ and use \eqref{eq:approx.PcurlIcurl} and the approximation properties of $\vlproj{k}{T}$ stated in \cite[Theorem 1.45]{Di-Pietro.Droniou:20} to see that 
  \begin{multline*}
  \norm[\vLeb(T)]{\Pcurl\big(\Icurl{T}\bvec{v}\big)-\vlproj{k}{T}\bvec{v}}
  + \norm[\vLeb(T)]{\bvec{v}-\vlproj{k}{T}\bvec{v}}
  \\
  +\sum_{F\in\FT}h_F^{\frac12}\norm[\vLeb(F)]{\bvec{v} - \vlproj{k}{T}\bvec{v}}
  \lesssim h_T^{k+1}\seminorm[\vSob{(k+1,2)}(T)]{\bvec{v}}.
  \end{multline*}
  Using Cauchy--Schwarz inequalities on the integrals and on the stabilisation bilinear form in \eqref{eq:Ediv.recast}, the bound \eqref{eq:bound.cGT} together with the norm equivalence \eqref{eq:equiv.norms}, and the consistency property \eqref{eq:s.curl.T:consistency} of the stabilisation term, we arrive at
  \[
    \begin{aligned}
      \left|\dEgrad(\bvec{v},\underline{q}_h)\right|
      \le{}& \sum_{T\in\Th}h_T^{k+1}\seminorm[\vSob{(k+1,2)}(T)]{\bvec{v}}\norm[\CURL,T]{\uGT\underline{q}_T}
      \\
      &  + \sum_{T\in\Th}h_T^{k+1}\seminorm[\vSob{(k+1,2)}(T)]{\bvec{v}}\norm[\vLeb(T)]{\GRAD\Pgrad\underline{q}_T}
      \\
      &  + \sum_{T\in\Th}\sum_{F\in\FT}
       h_T^{k+1}\seminorm[\vSob{(k+1,2)}(T)]{\bvec{v}}~h_F^{-\frac12}
      \norm[\Leb(F)]{\trF\underline{q}_F-\Pgrad\underline{q}_T}.
    \end{aligned}
  \]
  The conclusion follows from the estimate \eqref{eq:grad.Pgrad.trF}, and Cauchy--Schwarz inequalities on the sums.\qed
\end{proof}

\subsection{Proof of the adjoint consistency for the curl}\label{sec:consistency.adjoint.curl}

The proof of the adjoint consistency for the curl hinges on liftings defined as solutions of local problems.
For any $F\in\Fh$, the \emph{face lifting} $\RcurlF:\Xcurl{F}\to\Hrot{F}\cap\Hdiv{F}$ is such that, for all $\uvec{v}_F\in\Xcurl{F}$, $\RcurlF\uvec{v}_F = \bvec{\phi}_{\uvec{v}_F} + \GRAD_F\psi_{\uvec{v}_F}$ with $\bvec{\phi}_{\uvec{v}_F}\in\Hrot{F}\cap\Hdiv{F}$ such that
\begin{subequations}\label{eq:RcurlF:phi}
  \begin{alignat}{2}\label{eq:RcurlF:phi:rot}
    \ROT_F\bvec{\phi}_{\uvec{v}_F} &= \CF\uvec{v}_F &\qquad& \text{in $F$},
    \\
    \label{eq:RcurlF:phi:div}
    \DIV_F\bvec{\phi}_{\uvec{v}_F} &= 0 &\qquad& \text{in $F$},
    \\
    \label{eq:RcurlF:phi:bc}
    \bvec{\phi}_{\uvec{v}_F}\cdot\tangent_E &= v_E &\qquad& \text{on all $E\in\EF$},
  \end{alignat}
\end{subequations}
while $\psi_{\uvec{v}_F}\in C_{\rm c}^\infty(F)$ is such that
\begin{equation}\label{eq:RcurlF:psi}
  -\int_F\psi_{\uvec{v}_F}~\DIV_F\bvec{z}_F
  = \int_F(\trFt\uvec{v}_F - \bvec{\phi}_{\uvec{v}_F})\cdot\bvec{z}_F
  \qquad\forall\bvec{z}_F\in\cRoly{k+1}(F).
\end{equation}

Let now $T\in\Th$. The \emph{curl correction} $\CurlCorr:\Xcurl{T}\to\Hcurl{T}\cap\Hdiv{T}$ is such that, for all $\uvec{v}_T\in\Xcurl{T}$,
\begin{subequations}\label{eq:deltaT}
  \begin{alignat}{2}
    \DIV\CurlCorr\uvec{v}_T &= -\DIV\cCT\uvec{v}_T &\qquad& \text{in $T$},\label{eq:deltaT:div}
    \\
    \label{eq:deltaT:curl}
    \CURL\CurlCorr\uvec{v}_T &= \bvec{0}  &\qquad& \text{in $T$},
    \\
    \CurlCorr\uvec{v}_T\cdot\normal_F &= \CF\uvec{v}_F - \cCT\uvec{v}_T\cdot\normal_F &\qquad& \text{on all $F\in\FT$}.\label{eq:deltaT:bc}
  \end{alignat}
\end{subequations}
The curl correction lifts the difference between the face curl $\CF\uvec{v}_F$ and the normal component of the element curl $\cCT\uvec{v}_T$ as a function defined over $T$. Its role is to ensure the well-posedness of the problem defining the \emph{element lifting} $\RcurlT:\Xcurl{T}\to\Hcurl{T}\cap\Hdiv{T}$ such that, for all $\uvec{v}_T\in\Xcurl{T}$,
\begin{subequations}\label{eq:RcurlT}
  \begin{alignat}{2}\label{eq:RcurlT:curl}
    \CURL\RcurlT\uvec{v}_T &= \cCT\uvec{v}_T + \CurlCorr\uvec{v}_T &\qquad& \text{in $T$},
    \\
    \label{eq:RcurlT:div}
    \DIV\RcurlT\uvec{v}_T &= 0 &\qquad& \text{in $T$},
    \\
    \label{eq:RcurlT:bc}
    (\RcurlT\uvec{v}_T)_{{\rm t},F} &= \RcurlF\uvec{v}_F &\qquad& \text{on all $F\in\FT$}.
  \end{alignat}
\end{subequations}
In Appendix \ref{appen:RcurlT} we prove that these lifting operators are well-defined, and that they satisfy the following two key properties:
\begin{compactitem}
\item \emph{Orthogonality of the face lifting:} For all $F\in\Fh$,
  \begin{equation}\label{eq:RcurlF:orth}
    \int_F (\trFt\uvec{v}_F-\RcurlF\uvec{v}_F)\cdot\bvec{z}_F=0\qquad\forall(\uvec{v}_F,\bvec{z}_F)\in\Xcurl{F}\times\RT{k+1}(F);
\end{equation}
\item\emph{Boundedness of the element lifting:} For all $T\in\Th$,
\begin{multline}\label{eq:RcurlT:bound}
  \norm[\vLeb(T)]{\RcurlT\uvec{v}_T}+\norm[\vLeb(T)]{\CURL\RcurlT\uvec{v}_T}\lesssim \norm[\CURL,T]{\uvec{v}_T}+\norm[\DIV,T]{\uCT\uvec{v}_T}
  \\
  \forall \uvec{v}_T\in\Xcurl{T}.
\end{multline}
\end{compactitem}

\begin{lemma}[Approximation properties of {$\NE{k+1}(T)$} on polyhedral elements]\label{lem:approx.NEk}
  For all $T\in\Th$ and all $\bvec{w}\in \vSob{k+2}(T)$, there exists $\bvec{z}_T\in\NE{k+1}(T)$ such that
  \begin{align}\label{eq:approx.NEk.w}
    \norm[\vLeb(T)]{\bvec{w}-\bvec{z}_T}\lesssim{}& h_T^{k+1}\big(
    \seminorm[\vSob{k+1}(T)]{\bvec{w}}+\seminorm[\vSob{k+2}(T)]{\bvec{w}}
    \big),
    \\ \label{eq:approx.NEk.curl}
    \norm[\vLeb(T)]{\CURL\bvec{w}-\CURL\bvec{z}_T}\lesssim{}& h_T^{k+1}\seminorm[\vSob{k+2}(T)]{\bvec{w}}.
  \end{align}
\end{lemma}

\begin{proof}
  By the mesh regularity assumption, there is a simplex $S\subset T$ whose inradius is $\gtrsim h_T$. Following the arguments
  in the proof of \cite[Lemma 1.25]{Di-Pietro.Droniou:20}, we infer the norm equivalence
  \begin{equation}\label{eq:norm.equiv}
    \norm[\Leb(S)]{q}\simeq \norm[\Leb(T)]{q}\qquad\forall q\in\Poly{k+1}(T).
  \end{equation}
  Let us take $\bvec{z}_T$ as the N\'ed\'elec interpolant in $\NE{k+1}(S)$ of $\bvec{w}$; $\bvec{z}_T$ can be uniquely extended as an element of $\NE{k+1}(T)$.
  By the arguments in the proof of \cite[Theorem 3.14 and Corollary 3.17]{Hiptmair:02}, and since $S\subset T$, it holds
  \begin{equation}\label{eq:approx.NEk}
    \begin{aligned}
      \norm[\vLeb(S)]{\bvec{w}-\bvec{z}_T}\lesssim{}& h_T^{k+1}\big(\seminorm[\vSob{k+1}(T)]{\bvec{w}}+\seminorm[\vSob{k+2}(T)]{\bvec{w}}\big),\\
      \norm[\vLeb(S)]{\CURL\bvec{w}-\CURL\bvec{z}_T}\lesssim{}& h_T^{k+1}\seminorm[\vSob{k+2}(T)]{\bvec{w}}.
    \end{aligned}
  \end{equation}
  We then write, introducing $\vlproj{k+1}{T}\bvec{w}$ and using triangle inequalities,
  \begin{align*}
    \norm[\vLeb(T)]{\bvec{w}-\bvec{z}_T}
    &\lesssim \norm[\vLeb(T)]{\bvec{w}-\vlproj{k+1}{T}\bvec{w}} + \norm[\vLeb(T)]{\vlproj{k+1}{T}\bvec{w}-\bvec{z}_T}
    \\
    &\lesssim h_T^{k+1}\seminorm[\vSob{k+1}(T)]{\bvec{w}} + \norm[\vLeb(S)]{\vlproj{k+1}{T}\bvec{w}-\bvec{z}_T}
    \\
    &\lesssim h_T^{k+1}\big(
    \seminorm[\vSob{k+1}(T)]{\bvec{w}}
    + \seminorm[\vSob{k+2}(T)]{\bvec{w}}
    \big),
  \end{align*}
  where we have used the approximation property of $\vlproj{k+1}{T}$ together with the norm equivalence \eqref{eq:norm.equiv} in the second line,
  and concluded by introducing $\bvec{w}$ and invoking \eqref{eq:approx.NEk} to write
  \[
  \begin{aligned}
    \norm[\vLeb(S)]{\vlproj{k+1}{T}\bvec{w}-\bvec{z}_T}
    &\le \norm[\vLeb(S)]{\vlproj{k+1}{T}\bvec{w}-\bvec{w}}
    + \norm[\vLeb(S)]{\bvec{w}-\bvec{z}_T}
    \\
    &\lesssim h_T^{k+1}\seminorm[\vSob{k+1}(T)]{\bvec{w}}
    + h_T^{k+1}\big(
    \seminorm[\vSob{k+1}(T)]{\bvec{w}}
    + \seminorm[\vSob{k+2}(T)]{\bvec{w}}
    \big).
  \end{aligned}
  \] 
  This concludes the proof of \eqref{eq:approx.NEk.w}.
  The proof of \eqref{eq:approx.NEk.curl} is done in a similar way, introducing $\CURL(\vlproj{k+1}{T}\bvec{w})$
  and using the approximation property $\norm[\vLeb(T)]{\CURL\bvec{w}-\CURL(\vlproj{k+1}{T}\bvec{w})}\lesssim h_T^{k+1}\seminorm[\vSob{k+2}(T)]{\bvec{w}}$.\qed
\end{proof}

\begin{proof}[Theorem \ref{thm:adjoint.consistency.curl}]
  For all $T\in\Th$, select $\bvec{z}_T\in\NE{k+1}(T)$ given by Lemma \ref{lem:approx.NEk}.
  Using \eqref{eq:Xdiv:l2.prod} to expand $(\cdot,\cdot)_{\DIV,h}$ together with \eqref{eq:Pdiv.uCT=cCT}, and recalling \eqref{eq:ibp.Pcurl}, we see that it holds, for all $\uvec{v}_h\in\Xcurl{h}$,
  \begin{equation}\label{eq:dEcurl:basic}
    \begin{aligned}
      \dEcurl(\bvec{w},\uvec{v}_h)
      &=
      \sum_{T\in\Th}\int_T\big(\Pdiv(\Idiv{T}\bvec{w}_{|T})-\bvec{z}_T\big)\cdot\cCT\uvec{v}_T
      \\
      &\quad
      + \sum_{T\in\Th}\mathrm{s}_{\DIV,T}(\Idiv{T}\bvec{w}_{|T},\uCT\uvec{v}_T)
      \\
      &\quad      
      + \sum_{T\in\Th}\int_T\CURL(\bvec{z}_T-\bvec{w})\cdot\Pcurl\uvec{v}_T
      \\
      &\quad
      + \sum_{T\in\Th}\sum_{F\in\FT}\omega_{TF}\int_F(\bvec{z}_T\times\normal_F)\cdot\trFt\uvec{v}_F.
      \\
      &\eqcolon \term_1 + \term_2 + \term_3 + \term_4.
    \end{aligned}
  \end{equation}

  Using Cauchy--Schwarz and triangle inequalities, it is readily inferred for the first term
  \begin{equation} \label{eq:dEcurl:consistency:T1}
    \begin{aligned}
      |\term_1|
      &\lesssim\left[
        \sum_{T\in\Th}\left(
        \norm[\vLeb(T)]{\Pdiv(\Idiv{T}\bvec{w}) - \bvec{w}}^2
        + \norm[\vLeb(T)]{\bvec{w} - \bvec{z}_T}^2
        \right)
        \right]^{\frac12}
      \\
      &\quad\times\left(
      \sum_{T\in\Th}\norm[\vLeb(T)]{\cCT\uvec{v}_T}^2
      \right)^{\frac12}\\
      &\lesssim
      h^{k+1}\left(\seminorm[\vSob{k+1}(\Th)]{\bvec{w}}+\seminorm[\vSob{k+2}(\Th)]{\bvec{w}}\right)\norm[\DIV,h]{\uCh\uvec{v}_h},      
    \end{aligned}
  \end{equation}
  where the conclusion follows using the approximation properties \eqref{eq:approx.PdivIdiv} and \eqref{eq:approx.NEk.w} to bound the first factor, and \eqref{eq:bound:cCT} along with the norm equivalence \eqref{eq:equiv.norms} to bound the second.
  
  For $\term_2$, combining the consistency property \eqref{eq:s.div.T:consistency} of $\mathrm{s}_{\DIV,T}$ with discrete Cauchy--Schwarz inequalities and the definition of the $\norm[\DIV,h]{{\cdot}}$-norm readily gives
  \begin{equation}\label{eq:dEcurl:consistency:T2}
    |\term_2|\lesssim h^{k+1}\seminorm[\vSob{k+1}(\Th)]{\bvec{w}}\norm[\DIV,h]{\uCh\uvec{v}_h}.
  \end{equation}

  For $\term_3$, Cauchy--Schwarz inequalities, the approximation property \eqref{eq:approx.NEk.curl}, and the definition of the norm $\norm[\CURL,h]{{\cdot}}$ yield
  \begin{equation}\label{eq:dEcurl:consistency:T3}
    \begin{aligned}
      |\term_3|
      &\le\left(
      \sum_{T\in\Th}\!\norm[\vLeb(T)]{\CURL(\bvec{z}_T - \bvec{w})}^2
      \right)^{\frac12}\left(
      \sum_{T\in\Th}\!\norm[\vLeb(T)]{\Pcurl\uvec{v}_T}^2
      \right)^{\frac12}
      \\
      &\lesssim h^{k+1}\seminorm[\vSob{k+2}(\Th)]{\bvec{w}}\norm[\CURL,h]{\uvec{v}_h}.
    \end{aligned}
  \end{equation}

  Let us now consider the last term in the right-hand side of \eqref{eq:dEcurl:basic}.
  Since $(\bvec{z}_T)_{|F}\times\normal_F\in\RT{k+1}(F)$ as a consequence of \eqref{eq:NE.T.trace} with $\ell=k+1$, by \eqref{eq:RcurlF:orth} we can replace $\trFt\uvec{v}_T$ by $\RcurlF\uvec{v}_F$ in the boundary integral.
  Using the fact that both $\RcurlF\uvec{v}_F$ and the (rotated) tangential trace of $\bvec{w}$ are continuous across interfaces, along with the fact that $\omega_{T_1F}+\omega_{T_2F}=0$ for all $F\in \Fh$ between two elements $T_1,T_2$, and $\bvec{w}_{|F}\times\normal_F=\bvec{0}$ for all $F\subset\partial\Omega$, we then have
  \[
  \begin{aligned}
    \term_4    
    &= \sum_{T\in\Th}\sum_{F\in\FT}\omega_{TF}\int_F(\bvec{z}_T-\bvec{w})\times\normal_F\cdot\RcurlF\uvec{v}_F
    \\
    &= \sum_{T\in\Th}\left(
    \int_T(\bvec{z}_T - \bvec{w})\cdot\CURL\RcurlT\uvec{v}_T
    - \int_T\CURL(\bvec{z}_T - \bvec{w})\cdot\RcurlT\uvec{v}_T
    \right),
  \end{aligned}
  \]
  where the conclusion follows recalling that $\RcurlF\uvec{v}_F = (\RcurlT\uvec{v}_T)_{{\rm t},F}$ for all $T\in\Th$ and all $F\in\FT$ (see \eqref{eq:RcurlT:bc}), and integrating by parts.
  Using Cauchy--Schwarz inequalities, it is inferred
  \[
  \begin{aligned}
    |\term_4|
    &\le\left[
      \sum_{T\in\Th}\left(
      \norm[\vLeb(T)]{\bvec{z}_T - \bvec{w}}^2
      + \norm[\vLeb(T)]{\CURL(\bvec{z}_T - \bvec{w})}^2
      \right)
      \right]^{\frac12}
    \\
    &\quad\times\left[
      \sum_{T\in\Th}\left(
      \norm[\vLeb(T)]{\CURL\RcurlT\uvec{v}_T}^2
      +     \norm[\vLeb(T)]{\RcurlT\uvec{v}_T}^2
      \right)
      \right]^{\frac12}.
  \end{aligned}
  \]
  The approximation properties \eqref{eq:approx.NEk.w}--\eqref{eq:approx.NEk.curl} of $\bvec{z}_T$ along with the boundedness \eqref{eq:RcurlT:bound} of $\RcurlT\uvec{v}_T$ yield
  \begin{equation}\label{eq:dEcurl:consistency:T4}
    |\term_4|\lesssim h^{k+1}\left(\seminorm[\vSob{k+1}(\Th)]{\bvec{w}}+\seminorm[\vSob{k+2}(\Th)]{\bvec{w}}\right)\left(
    \norm[\CURL,h]{\uvec{v}_h} + \norm[\DIV,h]{\uCh\uvec{v}_h}
    \right).
  \end{equation}
  Plugging \eqref{eq:dEcurl:consistency:T1}--\eqref{eq:dEcurl:consistency:T4} into \eqref{eq:dEcurl:basic}, \eqref{eq:dEcurl:consistency} follows.\qed
\end{proof}

\subsection{Proof of the adjoint consistency for the divergence}\label{sec:adjoint.consistency.div}

\begin{proof}[Theorem \ref{thm:adjoint.consistency.div}]
  Combining the definition \eqref{eq:dEdiv} of the adjoint consistency error for the divergence with \eqref{eq:Pdiv:ibp} summed over $T\in\Th$, we infer that it holds, for all $(q,\uvec{v}_h)$ as in the theorem and all $q_h\in\Poly{k+1}(\Th)$ with $q_T\coloneq (q_h)_{|T}$ for all $T\in\Th$,
  \begin{multline*}
    \dEdiv(q,\uvec{v}_h)
    =
    \sum_{T\in\Th}\bigg[
      \int_T(\cancel{\lproj{k}{T}} q - q_T)\DT\uvec{v}_T
      \\
      + \int_T\GRAD(q - q_T)\cdot\Pdiv\uvec{v}_T
      + \sum_{F\in\FT}\omega_{TF}\int_F (q_T - q) v_F
      \bigg],
  \end{multline*}  
  where
  the cancellation of $\lproj{k}{T}$ is justified by its definition along with $\DT\uvec{v}_T\in\Poly{k}(T)$,
  while the insertion of $q$ into the boundary integral is possible thanks to its single-valuedness at interfaces along with the fact that it vanishes on $\partial\Omega$.
  Taking absolute values and using Cauchy--Schwarz inequalities in the right-hand side along with $h_F\simeq h_T$ for all $T\in\Th$ and all $F\in\FT$, we infer
  \begin{equation}\label{eq:dEdiv:bound}
    \begin{aligned}
      &\left|\dEdiv(q,\uvec{v}_h)\right|
      \\
      &\quad\lesssim
      \left[
        \sum_{T\in\Th}\left(
        h_T^{-2}\norm[\Leb(T)]{q - q_T}^2
        {+}\, \norm[\vLeb(T)]{\GRAD(q - q_T)}^2
        {+}\, h_T^{-1}\norm[\partial T]{q_T - q}^2
        \right)
        \right]^{\frac12}
      \\
      &\qquad\times\left[
        \sum_{T\in\Th}\left(
        h_T^2\norm[\Leb(T)]{\DT\uvec{v}_T}^2
        {+}\, \norm[\vLeb(T)]{\Pdiv\uvec{v}_T}^2
        {+}\hspace{-1ex} \sum_{F\in\FT}h_F\norm[\Leb(F)]{v_F}^2
        \right)
        \right]^{\frac12}.
    \end{aligned}
  \end{equation}
  Taking $q_h$ such that $q_T=\lproj{k+1}{T}q_{|T}$ for all $T\in\Th$ and using the approximation properties of the $\Leb$-orthogonal projector \cite[Theorem 1.45]{Di-Pietro.Droniou:20}, it is inferred that the first factor in the right-hand side of \eqref{eq:dEdiv:bound} is $\lesssim h^{k+1}\seminorm[\Sob{k+2}(\Th)]{q}$.
  Moving to the second factor, we use, for all $T\in\Th$, \cite[Lemma 8]{Di-Pietro.Ern:17} followed by the local seminorm equivalence \eqref{eq:equiv.norms} to write $h_T\norm[\Leb(T)]{\DT\uvec{v}_T}\lesssim\tnorm[\DIV,T]{\uvec{v}_T}\lesssim\norm[\DIV,T]{\uvec{v}_T}$. The same norm equivalence and the definition of the $\norm[\DIV,T]{{\cdot}}$-norm also yields $\norm[\vLeb(T)]{\Pdiv\uvec{v}_T}^2 + \sum_{F\in\FT}h_F\norm[\Leb(F)]{v_F}^2\lesssim \norm[\DIV,T]{\uvec{v}_T}$. The second factor in the right-hand side of \eqref{eq:dEdiv:bound} is therefore $\lesssim\norm[\DIV,h]{\uvec{v}_h}$, and the proof is complete.\qed
\end{proof}

\section{Convergence analysis for a DDR discretisation of magnetostatics}\label{sec:application}

We analyse in this section the DDR approximation of the following magnetostatics model, in which the unknowns are the magnetic field $\bvec{H}\in \Hcurl{\Omega}$ and the vector potential $\bvec{A}\in\Hdiv{\Omega}$:
\begin{equation}\label{eq:strong}
  \begin{aligned}
    &\mu\bvec{H} - \CURL\bvec{A} = \bvec{0}\,,\quad\CURL\bvec{H} = \bvec{J}\,,\quad\DIV\bvec{A} = 0 &\qquad& \text{in $\Omega$},
    \\
    &\bvec{A}\times\normal = \bvec{0} &\qquad& \text{on $\partial\Omega$}.
  \end{aligned}
\end{equation}
The free current $\bvec{J}$ belongs to $\CURL\Hcurl{\Omega}$ and we assume, for the sake of simplicity, that the magnetic permeability $\mu$ is piecewise-constant on the considered meshes, with $\mu:\Omega\to[\mu_-,\mu_+]$ for some constant numbers $0<\mu_-\le \mu_+$.

\subsection{Scheme}

As shown in \cite{Di-Pietro.Droniou:20*1}, a scheme based on the discrete de Rham tools can be written by replacing, in the weak formulation of \eqref{eq:strong}, the continuous $\Leb$-products by discrete ones built on the local products.
Denote by $\mu_T$ the constant value of $\mu$ over $T\in\Th$ and define
the bilinear forms $a_h:\Xcurl{h}\times\Xcurl{h}\to\Real$, $b_h:\Xcurl{h}\times\Xdiv{h}\to\Real$, and $c_h:\Xdiv{h}\times\Xdiv{h}\to\Real$ as follows:
For all $\uvec{\upsilon}_h,\uvec{\zeta}_h\in\Xcurl{h}$ and all $\uvec{w}_h,\uvec{v}_h\in\Xdiv{h}$,
\begin{gather*}
  \mathrm{a}_h(\uvec{\upsilon}_h,\uvec{\zeta}_h)\coloneq \sum_{T\in\Th}\mu_T(\uvec{\upsilon}_T,\uvec{\zeta}_T)_{\CURL,T},\quad
  \mathrm{b}_h(\uvec{\zeta}_h,\uvec{v}_h)\coloneq (\uCh\uvec{\zeta}_h,\uvec{v}_h)_{\DIV,h},
  \\
  \mathrm{c}_h(\uvec{w}_h,\uvec{v}_h)\coloneq\int_\Omega\Dh\uvec{w}_h~\Dh\uvec{v}_h.
\end{gather*}
The discrete problem then reads:
Find $\uvec{H}_h\in\Xcurl{h}$ and $\uvec{A}_h\in\Xdiv{h}$ such that
\begin{equation}\label{eq:discrete}
  \begin{alignedat}{2}
    \mathrm{a}_h(\uvec{H}_h,\uvec{\zeta}_h) - \mathrm{b}_h(\uvec{\zeta}_h,\uvec{A}_h) &= 0
    &\qquad&\forall\uvec{\zeta}_h\in\Xcurl{h},
    \\
    \mathrm{b}_h(\uvec{H}_h,\uvec{v}_h) + \mathrm{c}_h(\uvec{A}_h,\uvec{v}_h) &= \sum_{T\in\Th}\int_T\bvec{J}\cdot\Pdiv\uvec{v}_T
    &\qquad&\forall\uvec{v}_h\in\Xdiv{h}.
  \end{alignedat}
\end{equation}
The equations of this problem can be recast in the standard variational form $\mathcal A_h((\uvec{H}_h,\uvec{A}_h),(\uvec{\zeta}_h,\uvec{v}_h))=\mathcal L_h(\uvec{\zeta}_h,\uvec{v}_h)$, where $\mathcal A_h: (\Xcurl{h}\times\Xdiv{h})^2\to \Real$ and $\mathcal L_h:\Xcurl{h}\times\Xdiv{h}\to\Real$ are the bilinear and linear forms, respectively, such that
\[
\begin{aligned}
 \mathcal A_h((\uvec{\upsilon}_h,\uvec{w}_h),(\uvec{\zeta}_h,\uvec{v}_h))\coloneq{}&
\mathrm{a}_h(\uvec{\upsilon}_h,\uvec{\zeta}_h) - \mathrm{b}_h(\uvec{\zeta}_h,\uvec{w}_h)
+\mathrm{b}_h(\uvec{\upsilon}_h,\uvec{v}_h) + \mathrm{c}_h(\uvec{w}_h,\uvec{v}_h),\\
\mathcal L_h(\uvec{\zeta}_h,\uvec{v}_h)\coloneq{}&\sum_{T\in\Th}\int_T\bvec{J}\cdot\Pdiv\uvec{v}_T.
\end{aligned}
\]

\subsection{Error estimate}

To measure the error, we introduce the following $\Hcurl{\Omega}$- and $\Hdiv{\Omega}$-like (graph) norms on $\Xcurl{h}$ and $\Xdiv{h}$, respectively:
\[
\begin{alignedat}{2}
  \norm[\mu,\CURL,1,h]{\uvec{\zeta}_h}
  &\coloneq\left( \mathrm{a}_h(\uvec{\zeta}_h,\uvec{\zeta}_h) + \norm[\DIV,h]{\uCh\uvec{\zeta}_h}^2\right)^{\frac12}
  &\qquad&\forall\uvec{\zeta}_h\in\Xcurl{h},
  \\
  \norm[\DIV,1,h]{\uvec{v}_h}
  &\coloneq\left( \norm[\DIV,h]{\uvec{v}_h}^2 + \norm[\Leb(\Omega)]{\Dh\uvec{v}_h}^2\right)^{\frac12}
  &\qquad&\forall\uvec{v}_h\in\Xdiv{h}.
\end{alignedat}
\]
\begin{theorem}[Error estimate for the magnetostatics problem]\label{th:error.estimates} 
Assume that both the first and second Betti numbers of $\Omega$ are zero (i.e., $\Omega$ is not crossed by any tunnel and does not enclose any void).
  Then, there exists a unique solution $(\uvec{H}_h,\uvec{A}_h)\in\Xcurl{h}\times\Xdiv{h}$ to \eqref{eq:discrete}. Moreover, letting $(\bvec{H},\bvec{A})\in\Hcurl{\Omega}\times \Hdiv{\Omega}$ be the weak solution to \eqref{eq:strong} and assuming that $\bvec{H}\in \vC{0}(\overline{\Omega})\cap \vSob{k+2}(\Th)$ and $\bvec{A}\in \vC{0}(\overline{\Omega})\times \vSob{k+2}(\Th)$, we have
  \begin{multline}\label{eq:error.estimate}
    \norm[\mu,\CURL,1,h]{\uvec{H}_h-\Icurl{h}\bvec{H}}+\norm[\DIV,1,h]{\uvec{A}_h-\Idiv{h}\bvec{A}}\\
    \lesssim
    h^{k+1}\left(\seminorm[\vSob{k+1}(\Th)]{\CURL \bvec{H}}+\seminorm[\vSob{(k+1,2)}(\Th)]{\bvec{H}}+
    \seminorm[\vSob{k+1}(\Th)]{\bvec{A}}+\seminorm[\vSob{k+2}(\Th)]{\bvec{A}}\right),
  \end{multline}
  where the hidden constant in $\lesssim$ only depends on $\Omega$, $k$, the mesh regularity parameter, and $\mu_-$, $\mu_+$.
\end{theorem}

\begin{proof}
  As shown in the proof of \cite[Theorem 10]{Di-Pietro.Droniou:20*1}, the exactness of the rightmost part of the sequence \eqref{eq:global.sequence.3D}, which holds owing to \eqref{eq:Im.uCh.equal.Ker.Dh} and \eqref{eq:Im.Dh.equal.Pk.Th}, and the Poincar\'e inequalities for $\uCh$ and $\Dh$ (see Theorems \ref{thm:poincare.curl} and \ref{thm:poincare.div}) enable a reproduction of the arguments of the continuous inf-sup condition (see, e.g., \cite[Section 2]{Di-Pietro.Droniou.ea:20} or \cite[Theorem 4.9]{Arnold:18}) to see that $\mathcal A_h$ satisfies a uniform inf-sup condition with respect to the norm on $\Xcurl{h}\times\Xdiv{h}$ induced by $\norm[\mu,\CURL,1,h]{{\cdot}}$ and $\norm[\DIV,1,h]{{\cdot}}$.

  Using the Third Strang Lemma \cite{Di-Pietro.Droniou:18}, we therefore obtain \eqref{eq:error.estimate} provided we can prove that the consistency error
  \[
  \mathcal E_h((\bvec{H},\bvec{A});(\uvec{\zeta}_h,\uvec{v}_h))\coloneq
  \mathcal L_h(\uvec{\zeta}_h,\uvec{v}_h)
  - \mathcal A_h((\Icurl{h}\bvec{H},\Idiv{h}\bvec{A}),(\uvec{\zeta}_h,\uvec{v}_h))
  \]
  satisfies, for all $(\uvec{\zeta}_h,\uvec{v}_h)\in\Xcurl{h}\times\Xdiv{h}$,
  \begin{equation}\label{eq:global.consistency.error}
    \begin{aligned}
      &\mathcal E_h((\bvec{H},\bvec{A});(\uvec{\zeta}_h,\uvec{v}_h))
      \\
      &\quad\lesssim h^{k+1}\left(\seminorm[\vSob{k+1}(\Th)]{\CURL \bvec{H}}+\seminorm[\vSob{(k+1,2)}(\Th)]{\bvec{H}}+
      \seminorm[\vSob{k+1}(\Th)]{\bvec{A}}+\seminorm[\vSob{k+2}(\Th)]{\bvec{A}}\right)
      \\
      &\qquad
      \times
      \left(\norm[\mu,\CURL,1,h]{\uvec{\zeta}_h}+\norm[\DIV,1,h]{\uvec{v}_h}\right).
    \end{aligned}
  \end{equation}
  Expanding according to the respective definitions $\mathcal A_h$, $\mathcal L_h$, $\mathrm{a}_h$, $\mathrm{b}_h$, and $\mathrm{c}_h$, we have
  \begin{equation}\label{eq:Eh.sum}
    \begin{aligned}
      \mathcal E_h((\bvec{H},\bvec{A});(\uvec{\zeta}_h,\uvec{v}_h))
      &= \mathcal E_{h,1}((\bvec{H},\bvec{A});(\uvec{\zeta}_h,\uvec{v}_h))
      \\
      &\quad
      + \mathcal E_{h,2}((\bvec{H},\bvec{A});(\uvec{\zeta}_h,\uvec{v}_h))
      + \mathcal E_{h,3}((\bvec{H},\bvec{A});(\uvec{\zeta}_h,\uvec{v}_h)),
    \end{aligned}
  \end{equation}
  with
  \begin{align*}
    \mathcal E_{h,1}((\bvec{H},\bvec{A});(\uvec{\zeta}_h,\uvec{v}_h))\,{\coloneq}&
    \sum_{T\in\Th}\left(\int_T\bvec{J}\cdot\Pdiv\uvec{v}_T-(\uCT\big(\Icurl{T}\bvec{H}\big),\uvec{v}_T)_{\DIV,T}\right),\\
    \mathcal E_{h,2}((\bvec{H},\bvec{A});(\uvec{\zeta}_h,\uvec{v}_h))\,{\coloneq}&
    -\hspace{-1ex}\sum_{T\in\Th}\int_T \DT\big(\Idiv{T}\bvec{A}\big)\DT\uvec{v}_T\\
    \mathcal E_{h,3}((\bvec{H},\bvec{A});(\uvec{\zeta}_h,\uvec{v}_h))\,{\coloneq}&
    -\hspace{-1ex}\sum_{T\in\Th}\left(\mu_T(\Icurl{T}\bvec{H},\uvec{\zeta}_T)_{\CURL,T} - (\uCT\uvec{\zeta}_T,\Idiv{T}\bvec{A})_{\DIV,T}\right).
  \end{align*}

  Let us first estimate $\mathcal E_{h,1}$. Recalling that $\bvec{J}=\CURL\bvec{H}$, using the commutation formula \eqref{eq:uCT:commutation}, invoking the consistency \eqref{eq:consistency.L2.Xdiv} of the discrete $\Leb$-product on $\Xdiv{T}$ and applying a Cauchy--Schwarz inequality, we have
  \begin{align}
    \mathcal E_{h,1}((\bvec{H},\bvec{A});(\uvec{\zeta}_h,\uvec{v}_h))
    &=\sum_{T\in\Th}\left(\int_T\CURL\bvec{H}\cdot\Pdiv\uvec{v}_T-(\Idiv{T}(\CURL\bvec{H}),\uvec{v}_T)_{\DIV,T}\right)
    \nonumber\\
    &\lesssim \sum_{T\in\Th}h_T^{k+1}\seminorm[\vSob{k+1}(T)]{\CURL\bvec{H}}\norm[\DIV,T]{\uvec{v}_T}
    \nonumber\\
    &\le h^{k+1}\seminorm[\vSob{k+1}(\Th)]{\CURL \bvec{H}}\norm[\DIV,h]{\uvec{v}_h}.
   \label{eq:est.Eh1}
  \end{align}
  To handle $\mathcal E_{h,2}$, we use the commutation formula \eqref{eq:DT:commutation} to get $\DT\big(\Idiv{T}\bvec{A}\big) = \lproj{k}{T}(\DIV\bvec{A})=0$, and thus
  \begin{equation}\label{eq:est.Eh2}
    \mathcal E_{h,2}((\bvec{H},\bvec{A});(\uvec{\zeta}_h,\uvec{v}_h))=0.
  \end{equation}
  Finally, we turn to $\mathcal E_{h,3}$. Since $\bvec{A}\in\HDcurl{\Omega}$, the adjoint consistency Theorem \ref{thm:adjoint.consistency.curl} enables us to replace, in $\mathcal E_{h,3}$, the term $(\uCT\uvec{\zeta}_T,\Idiv{T}\bvec{A})_{\DIV,T}$ with ${\int_T \CURL\bvec{A}\cdot \Pcurl\uvec{\zeta}_T} = {\int_T \mu_T\bvec{H}\cdot \Pcurl\uvec{\zeta}_T}$ up to a term that is controlled, i.e.,
  \[
  \begin{aligned}
    &\mathcal E_{h,3}((\bvec{H},\bvec{A});(\uvec{\zeta}_h,\uvec{v}_h))
    \\
    &\quad\le
    -\sum_{T\in\Th}\left(\mu_T(\Icurl{T}\bvec{H},\uvec{\zeta}_T)_{\CURL,T} - \mu_T\int_T \bvec{H}\cdot\Pcurl\uvec{\zeta}_T\right)\\
    &\qquad
    +h^{k+1}\left(\seminorm[\vSob{k+1}(\Th)]{\bvec{A}}+\seminorm[\vSob{k+2}(\Th)]{\bvec{A}}\right)\norm[\mu,\CURL,1,h]{\uvec{\zeta}_h}
    \\
    &\le
    -\sum_{T\in\Th}\left(\mu_T\int_T\big[\Pcurl\big(\Icurl{T}\bvec{H}\big)-\bvec{H}\big]\cdot\Pcurl\uvec{\zeta}_T+\mathrm{s}_{\CURL,T}(\Icurl{T}\bvec{H},\uvec{\zeta}_T)\right)
    \\
    &\qquad
    +h^{k+1}\left(\seminorm[\vSob{k+1}(\Th)]{\bvec{A}}+\seminorm[\vSob{k+2}(\Th)]{\bvec{A}}\right)\norm[\mu,\CURL,1,h]{\uvec{\zeta}_h},
  \end{aligned}
  \]
  where we have used $\norm[\CURL,h]{\uvec{\zeta}_h}+\norm[\DIV,h]{\uCh\uvec{\zeta}_h}\lesssim \norm[\mu,\CURL,1,h]{\uvec{\zeta}_h}$ and the second inequality comes from expanding $(\cdot,\cdot)_{\CURL,T}$ according to its definition. Cauchy--Schwarz inequalities and the consistency properties \eqref{eq:approx.PcurlIcurl} and \eqref{eq:s.curl.T:consistency} then lead to
  \[
  \begin{aligned}
    \mathcal E_{h,3}((\bvec{H},\bvec{A});(\uvec{\zeta}_h,\uvec{v}_h))
    &\lesssim h^{k+1}\seminorm[\vSob{(k+1,2)}(\Th)]{\bvec{H}}\norm[\CURL,h]{\uvec{\zeta}_h}
    \\
    &\quad
    +h^{k+1}\left(\seminorm[\vSob{k+1}(\Th)]{\bvec{A}}+\seminorm[\vSob{k+2}(\Th)]{\bvec{A}}\right)\norm[\mu,\CURL,1,h]{\uvec{\zeta}_h}.
  \end{aligned}
  \]
  Plugging this estimate together with \eqref{eq:est.Eh1} and \eqref{eq:est.Eh2} into  \eqref{eq:Eh.sum}, we infer that \eqref{eq:global.consistency.error} holds, which concludes the proof.\qed
\end{proof}

\subsection{Numerical tests}\label{sec:ddr:numerical.tests}

We present here the results of some numerical tests obtained with the DDR scheme \eqref{eq:discrete} for the magnetostatics model \eqref{eq:strong}, focusing on comparing outputs obtained using either the complements \eqref{eq:spaces.T}, hereafter denoted by (K), or the orthogonal complements of \cite{Di-Pietro.Droniou.ea:20,Di-Pietro.Droniou:20*1}, denoted by ($\perp$). Both versions of the DDR complex, and related schemes, have been implemented in the \texttt{HArDCore3D} C++ framework (see \url{https://github.com/jdroniou/HArDCore}), using linear algebra facilities from the \texttt{Eigen3} library (see \url{http://eigen.tuxfamily.org}) and the \texttt{Intel MKL PARDISO} library (see \url{https://software.intel.com/en-us/mkl}) for the resolution of the global sparse linear system.
This solver proved to be the most efficient among those at our disposal.
All tests were run on a 16-inch 2019 MacBook Pro equipped with an 8-core Intel Core i9 processor (I9-9980HK) and 32Gb of RAM, and running macOS Big Sur version 11.5.1.
 We consider a constant permeability $\mu=1$, and the same exact smooth solution and mesh families as in \cite[Section 4.4]{Di-Pietro.Droniou:20*1} for comparability.


Figure \ref{fig:conv} presents the errors, for various values of $k$, computed in the relative discrete $\Hcurl{\Omega}\times\Hdiv{\Omega}$ norm:
\[
\frac{\big(\norm[\mu,\CURL,1,h]{\uvec{H}_h-\Icurl{h}\bvec{H}}^2+\norm[\DIV,1,h]{\uvec{A}_h-\Idiv{h}\bvec{A}}^2\big)^{\nicefrac12}}
{\big(\norm[\mu,\CURL,1,h]{\Icurl{h}\bvec{H}}^2+\norm[\DIV,1,h]{\Idiv{h}\bvec{A}}^2\big)^{\nicefrac12}}.
\]
In the case of the Koszul complements, Theorem \ref{th:error.estimates} states that this error should decrease as $\mathcal O(h^{k+1})$ with the mesh size. No such estimate is known for the DDR scheme using orthogonal complements and, due to the lack of key properties of these complements (hierarchical inclusions, structure of traces), it is not clear whether the analysis carried out in the rest of this paper could be adapted to such complements. Nonetheless, the graphs in Figure \ref{fig:conv} show that both schemes converge with an order $k+1$. The errors between (K) and ($\perp$) are essentially indistinguishable, except for $k\ge 1$ on tetrahedral meshes, where ($\perp$) leads to slightly larger errors than (K) -- about twice as large on the finest mesh with $k=3$.

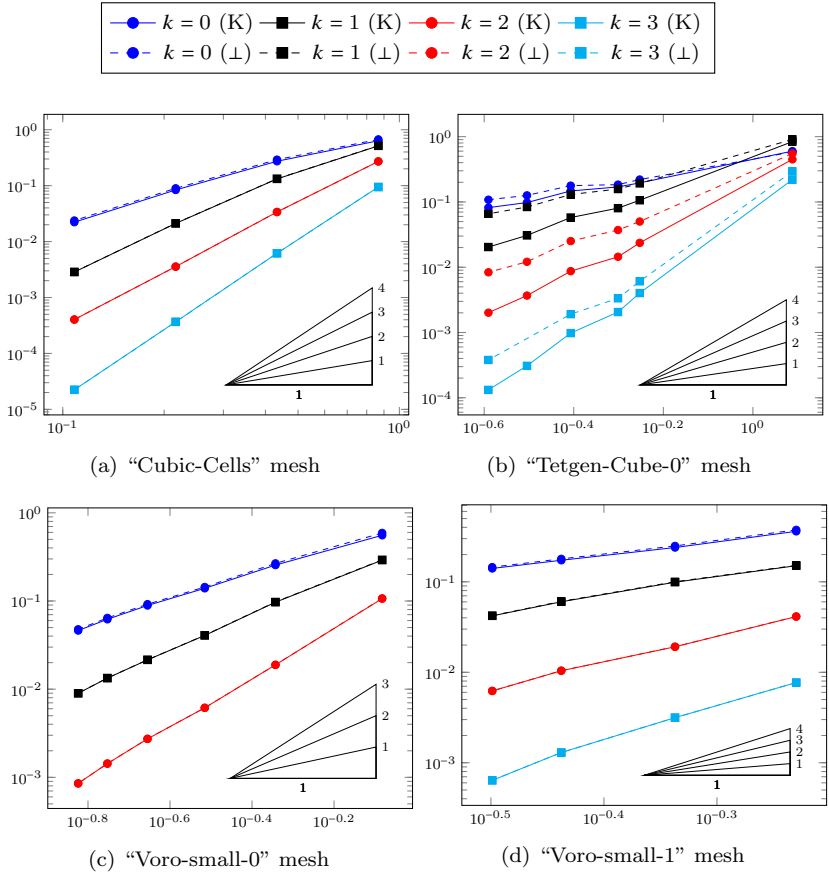
\begin{figure}\centering
  {
    \hypersetup{hidelinks}
    \ref{conv.cubic-cells}
  }
  \vspace{0.50cm}\\
  \begin{minipage}{0.45\textwidth}
    \begin{tikzpicture}[scale=0.70]
      \begin{loglogaxis}[legend columns=4] 
        \addplot [mark=*, blue] table[x=MeshSize,y=HcurlHdivError] {dat_revision/compl/Cubic-Cells_0/data_rates.dat};
        \logLogSlopeTriangle{0.90}{0.4}{0.1}{1}{black};
        \addplot [mark=square*, black] table[x=MeshSize,y=HcurlHdivError] {dat_revision/compl/Cubic-Cells_1/data_rates.dat};
        \logLogSlopeTriangle{0.90}{0.4}{0.1}{2}{black};
        \addplot [mark=*, red] table[x=MeshSize,y=HcurlHdivError] {dat_revision/compl/Cubic-Cells_2/data_rates.dat};
        \logLogSlopeTriangle{0.90}{0.4}{0.1}{3}{black};
        \addplot [mark=square*, cyan] table[x=MeshSize,y=HcurlHdivError] {dat_revision/compl/Cubic-Cells_3/data_rates.dat};
        \logLogSlopeTriangle{0.90}{0.4}{0.1}{4}{black};        
        \addplot [mark=*, mark options=solid, blue, dashed] table[x=MeshSize,y=HcurlHdivError] {dat_revision/orth/Cubic-Cells_0/data_rates.dat};
        \addplot [mark=square*, mark options=solid, black, dashed] table[x=MeshSize,y=HcurlHdivError] {dat_revision/orth/Cubic-Cells_1/data_rates.dat};
        \addplot [mark=*, mark options=solid, red, dashed] table[x=MeshSize,y=HcurlHdivError] {dat_revision/orth/Cubic-Cells_2/data_rates.dat};
        \addplot [mark=square*, mark options=solid, cyan, dashed] table[x=MeshSize,y=HcurlHdivError] {dat_revision/orth/Cubic-Cells_3/data_rates.dat};
      \end{loglogaxis}            
    \end{tikzpicture}
    \subcaption{``Cubic-Cells'' mesh}
  \end{minipage}
  \begin{minipage}{0.45\textwidth}
    \begin{tikzpicture}[scale=0.70]
      \begin{loglogaxis}[legend columns=4] 
        \addplot [mark=*, blue] table[x=MeshSize,y=HcurlHdivError] {dat_revision/compl/Tetgen-Cube-0_0/data_rates.dat};
        \logLogSlopeTriangle{0.90}{0.4}{0.1}{1}{black};
        \addplot [mark=square*, black] table[x=MeshSize,y=HcurlHdivError] {dat_revision/compl/Tetgen-Cube-0_1/data_rates.dat};
        \logLogSlopeTriangle{0.90}{0.4}{0.1}{2}{black};
        \addplot [mark=*, red] table[x=MeshSize,y=HcurlHdivError] {dat_revision/compl/Tetgen-Cube-0_2/data_rates.dat};
        \logLogSlopeTriangle{0.90}{0.4}{0.1}{3}{black};
        \addplot [mark=square*, cyan] table[x=MeshSize,y=HcurlHdivError] {dat_revision/compl/Tetgen-Cube-0_3/data_rates.dat};
        \logLogSlopeTriangle{0.90}{0.4}{0.1}{4}{black};        
        \addplot [mark=*, mark options=solid, blue, dashed] table[x=MeshSize,y=HcurlHdivError] {dat_revision/orth/Tetgen-Cube-0_0/data_rates.dat};
        \addplot [mark=square*, mark options=solid, black, dashed] table[x=MeshSize,y=HcurlHdivError] {dat_revision/orth/Tetgen-Cube-0_1/data_rates.dat};
        \addplot [mark=*, mark options=solid, red, dashed] table[x=MeshSize,y=HcurlHdivError] {dat_revision/orth/Tetgen-Cube-0_2/data_rates.dat};
        \addplot [mark=square*, mark options=solid, cyan, dashed] table[x=MeshSize,y=HcurlHdivError] {dat_revision/orth/Tetgen-Cube-0_3/data_rates.dat};
      \end{loglogaxis}                     
    \end{tikzpicture}
    \subcaption{``Tetgen-Cube-0'' mesh}
  \end{minipage}
  \vspace{0.25cm}\\
  \begin{minipage}{0.45\textwidth}
    \begin{tikzpicture}[scale=0.70]      
      \begin{loglogaxis}[legend columns=4] 
        \addplot [mark=*, blue] table[x=MeshSize,y=HcurlHdivError] {dat_revision/compl/Voro-small-0_0/data_rates.dat};
        \logLogSlopeTriangle{0.90}{0.4}{0.1}{1}{black};
        \addplot [mark=square*, black] table[x=MeshSize,y=HcurlHdivError] {dat_revision/compl/Voro-small-0_1/data_rates.dat};
        \logLogSlopeTriangle{0.90}{0.4}{0.1}{2}{black};
        \addplot [mark=*, red] table[x=MeshSize,y=HcurlHdivError] {dat_revision/compl/Voro-small-0_2/data_rates.dat};
        \logLogSlopeTriangle{0.90}{0.4}{0.1}{3}{black};
        \addplot [mark=*, mark options=solid, blue, dashed] table[x=MeshSize,y=HcurlHdivError] {dat_revision/orth/Voro-small-0_0/data_rates.dat};
        \addplot [mark=square*, mark options=solid, black, dashed] table[x=MeshSize,y=HcurlHdivError] {dat_revision/orth/Voro-small-0_1/data_rates.dat};
        \addplot [mark=*, mark options=solid, red, dashed] table[x=MeshSize,y=HcurlHdivError] {dat_revision/orth/Voro-small-0_2/data_rates.dat};
      \end{loglogaxis}
    \end{tikzpicture}
    \subcaption{``Voro-small-0'' mesh}
  \end{minipage}
  \begin{minipage}{0.45\textwidth}
    \begin{tikzpicture}[scale=0.70]
      \begin{loglogaxis}[legend columns=4, legend to name=conv.cubic-cells]        
        \addplot [mark=*, blue] table[x=MeshSize,y=HcurlHdivError] {dat_revision/compl/Voro-small-1_0/data_rates.dat};
        \logLogSlopeTriangle{0.90}{0.4}{0.1}{1}{black};
        \addplot [mark=square*, black] table[x=MeshSize,y=HcurlHdivError] {dat_revision/compl/Voro-small-1_1/data_rates.dat};
        \logLogSlopeTriangle{0.90}{0.4}{0.1}{2}{black};
        \addplot [mark=*, red] table[x=MeshSize,y=HcurlHdivError] {dat_revision/compl/Voro-small-1_2/data_rates.dat};
        \logLogSlopeTriangle{0.90}{0.4}{0.1}{3}{black};
        \addplot [mark=square*, cyan] table[x=MeshSize,y=HcurlHdivError] {dat_revision/compl/Voro-small-1_3/data_rates.dat};
        \logLogSlopeTriangle{0.90}{0.4}{0.1}{4}{black};        
        \addplot [mark=*, mark options=solid, blue, dashed] table[x=MeshSize,y=HcurlHdivError] {dat_revision/orth/Voro-small-1_0/data_rates.dat};
        \addplot [mark=square*, mark options=solid, black, dashed] table[x=MeshSize,y=HcurlHdivError] {dat_revision/orth/Voro-small-1_1/data_rates.dat};
        \addplot [mark=*, mark options=solid, red, dashed] table[x=MeshSize,y=HcurlHdivError] {dat_revision/orth/Voro-small-1_2/data_rates.dat};
        \addplot [mark=square*, mark options=solid, cyan, dashed] table[x=MeshSize,y=HcurlHdivError] {dat_revision/orth/Voro-small-1_3/data_rates.dat};
        \legend{$k=0$ (K), $k=1$ (K), $k=2$ (K), $k=3$ (K), $k=0$ ($\perp$), $k=1$ ($\perp$), $k=2$ ($\perp$), $k=3$ ($\perp$)};
      \end{loglogaxis}        
    \end{tikzpicture}
    \subcaption{``Voro-small-1'' mesh}
  \end{minipage}
  \caption{Relative error estimates in discrete $\Hcurl{\Omega}\times\Hdiv{\Omega}$ norm vs.\ $h$, for the Koszul complements of \eqref{eq:spaces.T} [(K), continuous lines], and the orthogonal complements of \cite{Di-Pietro.Droniou.ea:20}[($\perp$), dashed lines].
  \label{fig:conv}}
\end{figure}


The assembly of the ($\perp$)-DDR scheme requires, for any $Y\in\Th\cup\Fh$, to compute bases for the $\Leb$-orthogonal complements in $\vPoly{\ell}(Y)$ of $\Goly{\ell}(Y)$ and $\Roly{\ell}(Y)$, which is done by computing the kernels of local matrices through a full pivot LU algorithm \cite[Section 5.1]{Di-Pietro.Droniou:20*1}. On the contrary, in the (K) version, explicit bases for $\cGoly{\ell}(Y)$ and $\cRoly{\ell}(Y)$ can be devised; even though these bases are then orthonormalised to ensure a better numerical stability of the scheme (especially on non-isotropic elements, see the discussion in \cite[Section B.1.1]{Di-Pietro.Droniou:20} on this topic), the computational cost of creating the polynomial bases in ($\perp$) can be expected to be larger than in (K). Figure \ref{fig:times} compares the processor times for the two DDR schemes required for
\begin{inparaenum}[(a)]
\item the creation of the bases for local polynomial spaces and
\item the model construction (computation of the discrete operators, potentials, and $\Leb$-products, and global system assembly).
\end{inparaenum}
We do not compare the linear system resolution times as they are very close for both schemes.
In all the cases, the finest mesh of each sequence is considered; see Table \ref{tab:meshes}.
A profiling of the code shows that numerical integration is by far the most expensive operation.
We therefore include in Figure \ref{fig:times} also a comparison between two integration strategies on general meshes:
on one hand, the Homogeneous Numerical Integration of \cite{Chin.Lasserre.ea:15};
on the other hand, the use of standard quadratures on a simplicial subdivision of (nonsimplicial) elements.
\begin{table}\centering
  \begin{minipage}{0.95\textwidth}\centering
    \begin{tabular}{cccc}
      \toprule
      Mesh & $\card(\Th)$ & $\card(\Fh)$ & $\card(\Eh)$ \\
      \midrule
      Cubic\_Cells & 4~096 & 13~056 & 13~872 \\
      Tetgen\_Cube-0 & 2~925 & 6~228 & 3~965 \\
      Voro-small-0 & 2~197 & 15~969 & 27~546 \\
      Voro-small-1 & 356 & 2~376 & 4~042 \\
      \bottomrule
    \end{tabular}
    \subcaption{Number of relevant mesh entities}
  \end{minipage}
  \medskip\\
  \begin{minipage}{0.95\textwidth}\centering
    \begin{tabular}{ccccc}
      \toprule
      Mesh & $\dim(\underline{\bvec{X}}_{\CURL,h}^0)$ & $\dim(\underline{\bvec{X}}_{\CURL,h}^1)$ & $\dim(\underline{\bvec{X}}_{\CURL,h}^2)$ & $\dim(\underline{\bvec{X}}_{\CURL,h}^3)$ \\
      \midrule
      Cubic\_Cells & 13~872 & 83~296 & 207~504 & 398~784 \\
      Tetgen\_Cube-0 & 3~956 & 38~314 & 105~594 & 214~580 \\
      Voro-small-0 & 27~546 & 111~787 & 243~345 & --- \\
      Voro-small-1 & 4~042 & 16~636 & 36~474 & 64~624 \\
      \bottomrule
    \end{tabular}
    \subcaption{Dimension of the space $\Xcurl{h}$ for $k\in\{0,\ldots,3\}$}
  \end{minipage}
  \medskip\\
  \begin{minipage}{0.95\textwidth}\centering
    \begin{tabular}{ccccc}
      \toprule
      Mesh & $\dim(\underline{\bvec{X}}_{\DIV,h}^0)$ & $\dim(\underline{\bvec{X}}_{\DIV,h}^1)$ & $\dim(\underline{\bvec{X}}_{\DIV,h}^2)$ & $\dim(\underline{\bvec{X}}_{\DIV,h}^3)$ \\
      \midrule
      Cubic\_Cells & 13~056 & 63~744 & 160~256 & 314~880 \\
      Tetgen\_Cube-0 & 6~228 & 36~234 & 95~868 & 193~905 \\
      Voro-small-0 & 15~969 & 61~089 & 139~754 & --- \\
      Voro-small-1 & 2~376 & 9~264 & 21~376 & 39~780 \\
      \bottomrule
    \end{tabular}
    \subcaption{Dimension of the space $\Xdiv{h}$ for $k\in\{0,\ldots,3\}$}
  \end{minipage}
  \caption{Dimension of meshes and spaces considered for the evaluation of computational times in the numerical tests of Section \ref{sec:ddr:numerical.tests}.\label{tab:meshes}}
\end{table}
In the left column of Figure \ref{fig:times} we report the total CPU time, which constitutes the most reliable measure to assess performance.
Since our code makes use of multi-threading, we also report, in the right column, wall-clock times, which are more representative of real-life performance on the selected architecture.
Wall-clock times are subject to outside influences, such as the impact of other processes, and should therefore be regarded with caution.

As expected, when considering standard quadratures on element subdivisions, (K) polynomial bases are faster to create than ($\perp$) polynomial bases, but not by a large factor (this factor however becomes very large when $\Pgrad$ is required, which is not the case for the scheme \eqref{eq:discrete}, as the computation of $(\Roly{k+2}(T))^\perp$ (see \eqref{eq:PgradT}) necessitates to integrate polynomials of degree $2k+4$ over the elements). There is a more pronounced difference when comparing the time for model construction, which is mostly dedicated to the creation of the discrete vector calculus operators and potentials in $\Xcurl{h}$ and $\Xdiv{h}$ (once these are created, assembling the global linear system itself takes only a small fraction of the total model construction time).
Basis construction and model assembly times, on the other hand, basically even out between (K) and ($\perp$) when considering Homogeneous Numerical Integration, thereby showing the importance of efficient integral computation.
Drawing more definitive conclusions is always difficult, as running times highly depend on specific implementation choices, and our implementation is designed for flexibility rather than for efficiency on one given model.
The results presented in this section seem to show, however, that the DDR complex using Koszul complements is not only theoretically better (as it allows for complete consistency analysis and error estimates), but also requires less computational resources, at least when efficient integration is not available in the codes at hand.
The comparison of CPU times and wall clock times also confirms that the assembly step strongly benefits from parallel implementations.

To close this section, we briefly assess the evolution of construction and solution times with mesh refinement.
On a linear problem such as the one considered here, it is expected that the solution time be larger than the construction time starting from a certain number of elements.
To check whether this is the case, we consider the Voronoi mesh family ``Voro-small-0'' and the polynomial degree $k=2$ for the sequence.
This test is representative of a worst-case scenario for the construction time, since mesh elements are genuinely polyhedral and it is not possible to optimise the construction using standard (reference element) techniques.
The plots in Figure \ref{fig:times.vs.elements} show that the asymptotic behaviour of the construction times (``Bases'' and ``Model'') scale linearly with the number of elements (with the former being essentially negligible with respect to the latter), whereas the solution time (``Solve'') has a quadratic scaling.
The solution time exceeds the construction time starting from the fourth mesh in the sequence, which has 729 elements (a small number for a three-dimensional computation).
For meshes of real-life geometries, one can thus expect that the solution time will be the dominating cost (even if more efficient linear solvers become available at some point).


\begin{figure}\centering
  \begin{tabular}{ccc}
    & CPU time (s) & Wall time (s) \\
    \rotatebox[origin=c]{90}{$k=0$}
    &
    \begin{minipage}{0.45\textwidth}
      \begin{tikzpicture}[scale=0.70]
        \begin{axis}[
            ybar stacked,
            bar width=10pt,
            bar shift=-15pt,
            xticklabels={Cubic-Cells,Tetgen-Cube-0,Voro-small-0,Voro-small-1},
            xtick=data,
            xticklabel style={rotate=15,anchor=north east},
            enlarge x limits=0.20,
            ymin=0,
            ymax=30,
            height=5.6cm,
            width=8cm,
            legend style={
              legend columns=-1,
              draw=none,
              font=\tiny
            },
            legend to name=leg:times:nhi.K
          ]
          \addplot[fill=blue!40] table[x expr=\coordindex,y=TprocDDRCore] {dat_revision/times-by-degree/degree0_compl_times.dat};
          \addplot[fill=red!40] table[x expr=\coordindex,y=TprocModel] {dat_revision/times-by-degree/degree0_compl_times.dat};
          
          \legend{\parbox{24ex}{Bases (K, HNI)}, \parbox{24ex}{Model (K, HNI)}}
        \end{axis}
        \begin{axis}[
            ybar stacked,      
            hide axis,
            bar width=10pt,
            bar shift=-5pt,
            enlarge x limits=0.20,          
            ymin=0,
            ymax=30,
            height=5.6cm,
            width=8cm,
            legend style={
              legend columns=-1,
              draw=none,
              font=\tiny
            },
            legend to name=leg:times:nhi.orth
          ]
          \addplot[fill=blue!40, postaction={pattern={north east lines}, pattern color=blue}] table[x expr=\coordindex,y=TprocDDRCore] {dat_revision/times-by-degree/degree0_orth_times.dat};
          \addplot[fill=red!40, postaction={pattern={north east lines}, pattern color=red}] table[x expr=\coordindex,y=TprocModel] {dat_revision/times-by-degree/degree0_orth_times.dat};
          \legend{\parbox{24ex}{Bases ($\perp$, HNI)}, \parbox{24ex}{Model ($\perp$, HNI)}}
        \end{axis}
        \begin{axis}[
            ybar stacked,      
            hide axis,
            bar width=10pt,
            bar shift=5pt,
            enlarge x limits=0.20,
            ymin=0,
            ymax=30,
            height=5.6cm,
            width=8cm,          
            legend style={
              legend columns=-1,
              draw=none,
              font=\tiny
            },
            legend to name=leg:times:quad.K
          ]
          \addplot[fill=green!40] table[x expr=\coordindex,y=TprocDDRCore] {dat/times-by-degree/degree0_compl_times.dat};
          \addplot[fill=violet!40] table[x expr=\coordindex,y=TprocModel] {dat/times-by-degree/degree0_compl_times.dat};
          \legend{\parbox{24ex}{Bases (K, quad.)}, \parbox{24ex}{Model (K, quad.)}}
        \end{axis}
        \begin{axis}[
            ybar stacked,      
            hide axis,
            bar width=10pt,
            bar shift=15pt,
            enlarge x limits=0.20,
            ymin=0,
            ymax=30,
            height=5.6cm,
            width=8cm,          
            legend style={
              legend columns=-1,
              draw=none,
              font=\tiny
            },
            legend to name=leg:times:quad.orth
          ]
          \addplot[fill=green!40, postaction={pattern={north east lines}, pattern color=green!25!black}] table[x expr=\coordindex,y=TprocDDRCore] {dat/times-by-degree/degree0_orth_times.dat};
          \addplot[fill=violet!40, postaction={pattern={north east lines}, pattern color=violet}] table[x expr=\coordindex,y=TprocModel] {dat/times-by-degree/degree0_orth_times.dat};
          \legend{\parbox{24ex}{Bases ($\perp$, quad.)}, \parbox{24ex}{Model ($\perp$, quad.)}}
        \end{axis}      
      \end{tikzpicture}
    \end{minipage}
    &
    \begin{minipage}{0.45\textwidth}
      \begin{tikzpicture}[scale=0.70]
        \begin{axis}[
            ybar stacked,
            bar width=10pt,
            bar shift=-15pt,
            xticklabels={Cubic-Cells,Tetgen-Cube-0,Voro-small-0,Voro-small-1},
            xtick=data,
            xticklabel style={rotate=15,anchor=north east},
            enlarge x limits=0.20,
            ymin=0,
            ymax=5,
            height=5.6cm,
            width=8cm
          ]
          \addplot[fill=blue!40] table[x expr=\coordindex,y=TwallDDRCore] {dat_revision/times-by-degree/degree0_compl_times.dat};
          \addplot[fill=red!40] table[x expr=\coordindex,y=TwallModel] {dat_revision/times-by-degree/degree0_compl_times.dat};          
        \end{axis}
        \begin{axis}[
            ybar stacked,      
            hide axis,
            bar width=10pt,
            bar shift=-5pt,
            enlarge x limits=0.20,          
            ymin=0,
            ymax=5,
            height=5.6cm,
            width=8cm
          ]
          \addplot[fill=blue!40, postaction={pattern={north east lines}, pattern color=blue}] table[x expr=\coordindex,y=TwallDDRCore] {dat_revision/times-by-degree/degree0_orth_times.dat};
          \addplot[fill=red!40, postaction={pattern={north east lines}, pattern color=red}] table[x expr=\coordindex,y=TwallModel] {dat_revision/times-by-degree/degree0_orth_times.dat};
        \end{axis}
        \begin{axis}[
            ybar stacked,      
            hide axis,
            bar width=10pt,
            bar shift=5pt,
            enlarge x limits=0.20,
            ymin=0,
            ymax=5,
            height=5.6cm,
            width=8cm
          ]
          \addplot[fill=green!40] table[x expr=\coordindex,y=TwallDDRCore] {dat/times-by-degree/degree0_compl_times.dat};
          \addplot[fill=violet!40] table[x expr=\coordindex,y=TwallModel] {dat/times-by-degree/degree0_compl_times.dat};
        \end{axis}
        \begin{axis}[
            ybar stacked,      
            hide axis,
            bar width=10pt,
            bar shift=15pt,
            enlarge x limits=0.20,
            ymin=0,
            ymax=5,
            height=5.6cm,
            width=8cm
          ]
          \addplot[fill=green!40, postaction={pattern={north east lines}, pattern color=green!25!black}] table[x expr=\coordindex,y=TwallDDRCore] {dat/times-by-degree/degree0_orth_times.dat};
          \addplot[fill=violet!40, postaction={pattern={north east lines}, pattern color=violet}] table[x expr=\coordindex,y=TwallModel] {dat/times-by-degree/degree0_orth_times.dat};
        \end{axis}      
      \end{tikzpicture}
    \end{minipage}
    \\
    \rotatebox[origin=c]{90}{$k=1$}
    &
    \begin{minipage}{0.45\textwidth}
      \begin{tikzpicture}[scale=0.70]
        \begin{axis}[
            ybar stacked,
            bar width=10pt,
            bar shift=-15pt,
            xticklabels={Cubic-Cells,Tetgen-Cube-0,Voro-small-0,Voro-small-1},
            xtick=data,
            xticklabel style={rotate=15,anchor=north east},
            enlarge x limits=0.20,
            ymin=0,
            ymax=170,
            height=5.6cm,
            width=8cm
          ]
          \addplot[fill=blue!40] table[x expr=\coordindex,y=TprocDDRCore] {dat_revision/times-by-degree/degree1_compl_times.dat};
          \addplot[fill=red!40] table[x expr=\coordindex,y=TprocModel] {dat_revision/times-by-degree/degree1_compl_times.dat};
        \end{axis}
        \begin{axis}[
            ybar stacked,      
            hide axis,
            bar width=10pt,
            bar shift=-5pt,
            enlarge x limits=0.20,          
            ymin=0,
            ymax=170,
            height=5.6cm,
            width=8cm
          ]
          \addplot[fill=blue!40, postaction={pattern={north east lines}, pattern color=blue}] table[x expr=\coordindex,y=TprocDDRCore] {dat_revision/times-by-degree/degree1_orth_times.dat};
          \addplot[fill=red!40, postaction={pattern={north east lines}, pattern color=red}] table[x expr=\coordindex,y=TprocModel] {dat_revision/times-by-degree/degree1_orth_times.dat};
        \end{axis}
        \begin{axis}[
            ybar stacked,      
            hide axis,
            bar width=10pt,
            bar shift=5pt,
            enlarge x limits=0.20,
            ymin=0,
            ymax=170,
            height=5.6cm,
            width=8cm
          ]
          \addplot[fill=green!40] table[x expr=\coordindex,y=TprocDDRCore] {dat/times-by-degree/degree1_compl_times.dat};
          \addplot[fill=violet!40] table[x expr=\coordindex,y=TprocModel] {dat/times-by-degree/degree1_compl_times.dat};
        \end{axis}
        \begin{axis}[
            ybar stacked,      
            hide axis,
            bar width=10pt,
            bar shift=15pt,
            enlarge x limits=0.20,
            ymin=0,
            ymax=170,
            height=5.6cm,
            width=8cm
          ]
          \addplot[fill=green!40, postaction={pattern={north east lines}, pattern color=green!25!black}] table[x expr=\coordindex,y=TprocDDRCore] {dat/times-by-degree/degree1_orth_times.dat};
          \addplot[fill=violet!40, postaction={pattern={north east lines}, pattern color=violet}] table[x expr=\coordindex,y=TprocModel] {dat/times-by-degree/degree1_orth_times.dat};
        \end{axis}      
      \end{tikzpicture}
    \end{minipage}
    &
    \begin{minipage}{0.45\textwidth}
      \begin{tikzpicture}[scale=0.70]
        \begin{axis}[
            ybar stacked,
            bar width=10pt,
            bar shift=-15pt,
            xticklabels={Cubic-Cells,Tetgen-Cube-0,Voro-small-0,Voro-small-1},
            xtick=data,
            xticklabel style={rotate=15,anchor=north east},
            enlarge x limits=0.20,
            ymin=0,
            ymax=38,
            height=5.6cm,
            width=8cm
          ]
          \addplot[fill=blue!40] table[x expr=\coordindex,y=TwallDDRCore] {dat_revision/times-by-degree/degree1_compl_times.dat};
          \addplot[fill=red!40] table[x expr=\coordindex,y=TwallModel] {dat_revision/times-by-degree/degree1_compl_times.dat};
        \end{axis}
        \begin{axis}[
            ybar stacked,      
            hide axis,
            bar width=10pt,
            bar shift=-5pt,
            enlarge x limits=0.20,          
            ymin=0,
            ymax=38,
            height=5.6cm,
            width=8cm
          ]
          \addplot[fill=blue!40, postaction={pattern={north east lines}, pattern color=blue}] table[x expr=\coordindex,y=TwallDDRCore] {dat_revision/times-by-degree/degree1_orth_times.dat};
          \addplot[fill=red!40, postaction={pattern={north east lines}, pattern color=red}] table[x expr=\coordindex,y=TwallModel] {dat_revision/times-by-degree/degree1_orth_times.dat};
        \end{axis}
        \begin{axis}[
            ybar stacked,      
            hide axis,
            bar width=10pt,
            bar shift=5pt,
            enlarge x limits=0.20,
            ymin=0,
            ymax=38,
            height=5.6cm,
            width=8cm
                      ]
          \addplot[fill=green!40] table[x expr=\coordindex,y=TwallDDRCore] {dat/times-by-degree/degree1_compl_times.dat};
          \addplot[fill=violet!40] table[x expr=\coordindex,y=TwallModel] {dat/times-by-degree/degree1_compl_times.dat};
        \end{axis}
        \begin{axis}[
            ybar stacked,      
            hide axis,
            bar width=10pt,
            bar shift=15pt,
            enlarge x limits=0.20,
            ymin=0,
            ymax=38,
            height=5.6cm,
            width=8cm
          ]
          \addplot[fill=green!40, postaction={pattern={north east lines}, pattern color=green!25!black}] table[x expr=\coordindex,y=TwallDDRCore] {dat/times-by-degree/degree1_orth_times.dat};
          \addplot[fill=violet!40, postaction={pattern={north east lines}, pattern color=violet}] table[x expr=\coordindex,y=TwallModel] {dat/times-by-degree/degree1_orth_times.dat};
        \end{axis}      
      \end{tikzpicture}
    \end{minipage}
    \\
    \rotatebox[origin=c]{90}{$k=2$}
    &
    \begin{minipage}{0.45\textwidth}
      \begin{tikzpicture}[scale=0.70]
        \begin{axis}[
            ybar stacked,
            bar width=10pt,
            bar shift=-15pt,
            xticklabels={Cubic-Cells,Tetgen-Cube-0,Voro-small-0,Voro-small-1},
            xtick=data,
            xticklabel style={rotate=15,anchor=north east},
            enlarge x limits=0.20,
            ymin=0,
            ymax=1500,
            height=5.6cm,
            width=8cm
          ]
          \addplot[fill=blue!40] table[x expr=\coordindex,y=TprocDDRCore] {dat_revision/times-by-degree/degree2_compl_times.dat};
          \addplot[fill=red!40] table[x expr=\coordindex,y=TprocModel] {dat_revision/times-by-degree/degree2_compl_times.dat};
        \end{axis}
        \begin{axis}[
            ybar stacked,      
            hide axis,
            bar width=10pt,
            bar shift=-5pt,
            enlarge x limits=0.20,          
            ymin=0,
            ymax=1500,
            height=5.6cm,
            width=8cm
          ]
          \addplot[fill=blue!40, postaction={pattern={north east lines}, pattern color=blue}] table[x expr=\coordindex,y=TprocDDRCore] {dat_revision/times-by-degree/degree2_orth_times.dat};
          \addplot[fill=red!40, postaction={pattern={north east lines}, pattern color=red}] table[x expr=\coordindex,y=TprocModel] {dat_revision/times-by-degree/degree2_orth_times.dat};
        \end{axis}
        \begin{axis}[
            ybar stacked,      
            hide axis,
            bar width=10pt,
            bar shift=5pt,
            enlarge x limits=0.20,
            ymin=0,
            ymax=1500,
            height=5.6cm,
            width=8cm
          ]
          \addplot[fill=green!40] table[x expr=\coordindex,y=TprocDDRCore] {dat/times-by-degree/degree2_compl_times.dat};
          \addplot[fill=violet!40] table[x expr=\coordindex,y=TprocModel] {dat/times-by-degree/degree2_compl_times.dat};
        \end{axis}
        \begin{axis}[
            ybar stacked,      
            hide axis,
            bar width=10pt,
            bar shift=15pt,
            enlarge x limits=0.20,
            ymin=0,
            ymax=1500,
            height=5.6cm,
            width=8cm
          ]
          \addplot[fill=green!40, postaction={pattern={north east lines}, pattern color=green!25!black}] table[x expr=\coordindex,y=TprocDDRCore] {dat/times-by-degree/degree2_orth_times.dat};
          \addplot[fill=violet!40, postaction={pattern={north east lines}, pattern color=violet}] table[x expr=\coordindex,y=TprocModel] {dat/times-by-degree/degree2_orth_times.dat};
        \end{axis}      
      \end{tikzpicture}
    \end{minipage}
    &
    \begin{minipage}{0.45\textwidth}
      \begin{tikzpicture}[scale=0.70]
        \begin{axis}[
            ybar stacked,
            bar width=10pt,
            bar shift=-15pt,
            xticklabels={Cubic-Cells,Tetgen-Cube-0,Voro-small-0,Voro-small-1},
            xtick=data,
            xticklabel style={rotate=15,anchor=north east},
            enlarge x limits=0.20,
            ymin=0,
            ymax=200,
            height=5.6cm,
            width=8cm
          ]
          \addplot[fill=blue!40] table[x expr=\coordindex,y=TwallDDRCore] {dat_revision/times-by-degree/degree2_compl_times.dat};
          \addplot[fill=red!40] table[x expr=\coordindex,y=TwallModel] {dat_revision/times-by-degree/degree2_compl_times.dat};
        \end{axis}
        \begin{axis}[
            ybar stacked,      
            hide axis,
            bar width=10pt,
            bar shift=-5pt,
            enlarge x limits=0.20,          
            ymin=0,
            ymax=200,
            height=5.6cm,
            width=8cm
          ]
          \addplot[fill=blue!40, postaction={pattern={north east lines}, pattern color=blue}] table[x expr=\coordindex,y=TwallDDRCore] {dat_revision/times-by-degree/degree2_orth_times.dat};
          \addplot[fill=red!40, postaction={pattern={north east lines}, pattern color=red}] table[x expr=\coordindex,y=TwallModel] {dat_revision/times-by-degree/degree2_orth_times.dat};
        \end{axis}
        \begin{axis}[
            ybar stacked,      
            hide axis,
            bar width=10pt,
            bar shift=5pt,
            enlarge x limits=0.20,
            ymin=0,
            ymax=200,
            height=5.6cm,
            width=8cm
          ]
          \addplot[fill=green!40] table[x expr=\coordindex,y=TwallDDRCore] {dat/times-by-degree/degree2_compl_times.dat};
          \addplot[fill=violet!40] table[x expr=\coordindex,y=TwallModel] {dat/times-by-degree/degree2_compl_times.dat};
        \end{axis}
        \begin{axis}[
            ybar stacked,      
            hide axis,
            bar width=10pt,
            bar shift=15pt,
            enlarge x limits=0.20,
            ymin=0,
            ymax=200,
            height=5.6cm,
            width=8cm
          ]
          \addplot[fill=green!40, postaction={pattern={north east lines}, pattern color=green!25!black}] table[x expr=\coordindex,y=TwallDDRCore] {dat/times-by-degree/degree2_orth_times.dat};
          \addplot[fill=violet!40, postaction={pattern={north east lines}, pattern color=violet}] table[x expr=\coordindex,y=TwallModel] {dat/times-by-degree/degree2_orth_times.dat};
        \end{axis}      
      \end{tikzpicture}
    \end{minipage}
    \\
    \rotatebox[origin=c]{90}{$k=3$}
    &
    \begin{minipage}{0.45\textwidth}
      \begin{tikzpicture}[scale=0.70]
        \begin{axis}[
            ybar stacked,
            bar width=10pt,
            bar shift=-15pt,
            xticklabels={Cubic-Cells,Tetgen-Cube-0,Voro-small-1},
            xtick=data,
            xticklabel style={rotate=15,anchor=north east},
            enlarge x limits=0.20,
            ymin=0,
            ymax=8000,
            height=5.6cm,
            width=8cm
          ]
          \addplot[fill=blue!40] table[x expr=\coordindex,y=TprocDDRCore] {dat_revision/times-by-degree/degree3_compl_times.dat};
          \addplot[fill=red!40] table[x expr=\coordindex,y=TprocModel] {dat_revision/times-by-degree/degree3_compl_times.dat};
        \end{axis}
        \begin{axis}[
            ybar stacked,      
            hide axis,
            bar width=10pt,
            bar shift=-5pt,
            enlarge x limits=0.20,          
            ymin=0,
            ymax=8000,
            height=5.6cm,
            width=8cm
          ]
          \addplot[fill=blue!40, postaction={pattern={north east lines}, pattern color=blue}] table[x expr=\coordindex,y=TprocDDRCore] {dat_revision/times-by-degree/degree3_orth_times.dat};
          \addplot[fill=red!40, postaction={pattern={north east lines}, pattern color=red}] table[x expr=\coordindex,y=TprocModel] {dat_revision/times-by-degree/degree3_orth_times.dat};
        \end{axis}
        \begin{axis}[
            ybar stacked,      
            hide axis,
            bar width=10pt,
            bar shift=5pt,
            enlarge x limits=0.20,
            ymin=0,
            ymax=8000,
            height=5.6cm,
            width=8cm
          ]
          \addplot[fill=green!40] table[x expr=\coordindex,y=TprocDDRCore] {dat/times-by-degree/degree3_compl_times.dat};
          \addplot[fill=violet!40] table[x expr=\coordindex,y=TprocModel] {dat/times-by-degree/degree3_compl_times.dat};
        \end{axis}
        \begin{axis}[
            ybar stacked,      
            hide axis,
            bar width=10pt,
            bar shift=15pt,
            enlarge x limits=0.20,
            ymin=0,
            ymax=8000,
            height=5.6cm,
            width=8cm
          ]
          \addplot[fill=green!40, postaction={pattern={north east lines}, pattern color=green!25!black}] table[x expr=\coordindex,y=TprocDDRCore] {dat/times-by-degree/degree3_orth_times.dat};
          \addplot[fill=violet!40, postaction={pattern={north east lines}, pattern color=violet}] table[x expr=\coordindex,y=TprocModel] {dat/times-by-degree/degree3_orth_times.dat};
        \end{axis}      
      \end{tikzpicture}
    \end{minipage}
    &
    \begin{minipage}{0.45\textwidth}
      \begin{tikzpicture}[scale=0.70]
        \begin{axis}[
            ybar stacked,
            bar width=10pt,
            bar shift=-15pt,
            xticklabels={Cubic-Cells,Tetgen-Cube-0,Voro-small-1},
            xtick=data,
            xticklabel style={rotate=15,anchor=north east},
            enlarge x limits=0.20,
            ymin=0,
            ymax=600,
            height=5.6cm,
            width=8cm
          ]
          \addplot[fill=blue!40] table[x expr=\coordindex,y=TwallDDRCore] {dat_revision/times-by-degree/degree3_compl_times.dat};
          \addplot[fill=red!40] table[x expr=\coordindex,y=TwallModel] {dat_revision/times-by-degree/degree3_compl_times.dat};
        \end{axis}
        \begin{axis}[
            ybar stacked,      
            hide axis,
            bar width=10pt,
            bar shift=-5pt,
            enlarge x limits=0.20,          
            ymin=0,
            ymax=600,
            height=5.6cm,
            width=8cm
          ]
          \addplot[fill=blue!40, postaction={pattern={north east lines}, pattern color=blue}] table[x expr=\coordindex,y=TwallDDRCore] {dat_revision/times-by-degree/degree3_orth_times.dat};
          \addplot[fill=red!40, postaction={pattern={north east lines}, pattern color=red}] table[x expr=\coordindex,y=TwallModel] {dat_revision/times-by-degree/degree3_orth_times.dat};
        \end{axis}
        \begin{axis}[
            ybar stacked,      
            hide axis,
            bar width=10pt,
            bar shift=5pt,
            enlarge x limits=0.20,
            ymin=0,
            ymax=600,
            height=5.6cm,
            width=8cm
          ]
          \addplot[fill=green!40] table[x expr=\coordindex,y=TwallDDRCore] {dat/times-by-degree/degree3_compl_times.dat};
          \addplot[fill=violet!40] table[x expr=\coordindex,y=TwallModel] {dat/times-by-degree/degree3_compl_times.dat};
        \end{axis}
        \begin{axis}[
            ybar stacked,      
            hide axis,
            bar width=10pt,
            bar shift=15pt,
            enlarge x limits=0.20,
            ymin=0,
            ymax=600,
            height=5.6cm,
            width=8cm
          ]
          \addplot[fill=green!40, postaction={pattern={north east lines}, pattern color=green!25!black}] table[x expr=\coordindex,y=TwallDDRCore] {dat/times-by-degree/degree3_orth_times.dat};
          \addplot[fill=violet!40, postaction={pattern={north east lines}, pattern color=violet}] table[x expr=\coordindex,y=TwallModel] {dat/times-by-degree/degree3_orth_times.dat};
        \end{axis}      
      \end{tikzpicture}
    \end{minipage}
  \end{tabular}  
      \fbox{\hypersetup{hidelinks}
        \begin{tabular}{{r}{l}}
          \ref{leg:times:nhi.K} & \ref{leg:times:nhi.orth}
          \\
          \ref{leg:times:quad.K} & \ref{leg:times:quad.orth}
        \end{tabular}}
  \caption{CPU (left column) and wall times (right column), both measured in seconds, for the computation of the DDR bases (``Bases'') and of the model construction (``Model'') for Koszul (K) and orthogonal ($\perp$) complements using Homogeneous Numerical Integration (HNI) or quadratures on an element subdivision (quad.) on the finest mesh of each sequence; see Table \ref{tab:meshes}.}\label{fig:times}
\end{figure}
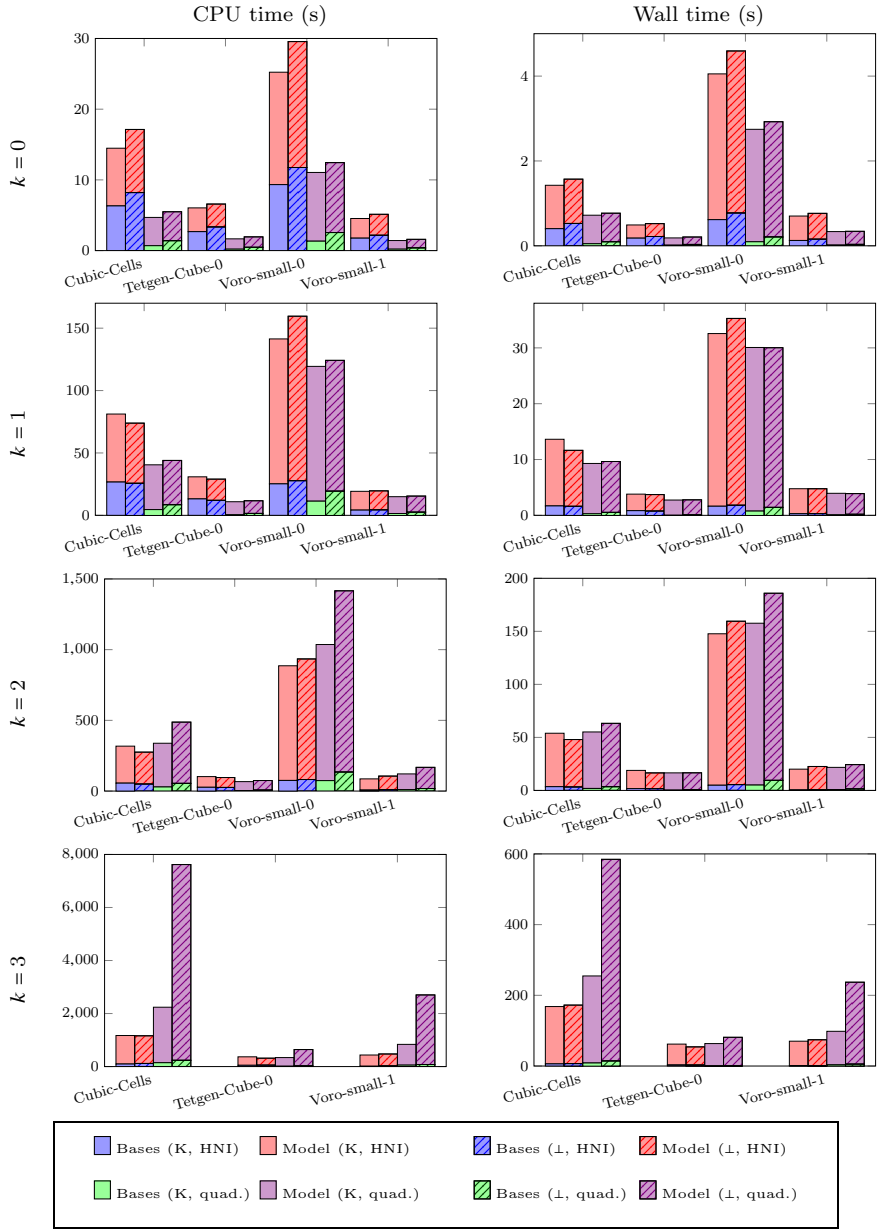

\begin{figure}\centering
  \begin{minipage}{0.45\textwidth}
    \begin{tikzpicture}[scale=0.70]
      \begin{loglogaxis}[
          legend pos = north west,
          ymax=1100,
          xlabel = {$\card(\Th)$},
          ylabel = {Time (s)}
        ]
        \addplot table[x=NbCells,y=TwallDDRCore] {dat_revision/compl/Voro-small-0_2/times.dat};
        \addplot table[x=NbCells,y=TwallModel] {dat_revision/compl/Voro-small-0_2/times.dat};
        \addplot table[x=NbCells,y=TwallSolve] {dat_revision/compl/Voro-small-0_2/times.dat};
        \logLogSlopeTriangle{0.925}{0.3}{0.1}{2}{black};
        \logLogSlopeTriangle{0.925}{0.3}{0.1}{1}{black};
        \legend{Bases, Model, Solve}
      \end{loglogaxis}
    \end{tikzpicture}
    \subcaption{Homogeneous Numerical Integration}
  \end{minipage}
  \begin{minipage}{0.45\textwidth}
    \begin{tikzpicture}[scale=0.70]
      \begin{loglogaxis}[
          legend pos = north west,
          ymax=1100,
          xlabel = {$\card(\Th)$},
          ylabel = {Time (s)}
        ]
        \addplot table[x=NbCells,y=TwallDDRCore] {dat/compl/Voro-small-0_2/times.dat};
        \addplot table[x=NbCells,y=TwallModel] {dat/compl/Voro-small-0_2/times.dat};
        \addplot table[x=NbCells,y=TwallSolve] {dat/compl/Voro-small-0_2/times.dat};
        \logLogSlopeTriangle{0.925}{0.3}{0.1}{2}{black};
        \logLogSlopeTriangle{0.925}{0.3}{0.1}{1}{black};
        \legend{Bases, Model, Solve}
      \end{loglogaxis}
    \end{tikzpicture}
    \subcaption{Quadratures on element subdivisions}
  \end{minipage}
  \caption{Comparison of the wall times, all measured in seconds, for the computation of the DDR bases (``Bases''), the model construction (``Model'') and the resolution of the linear system using the PARDISO direct solver (``Solve'') for the Voro-small-0 mesh sequence for $k=2$.}\label{fig:times.vs.elements}  
\end{figure}
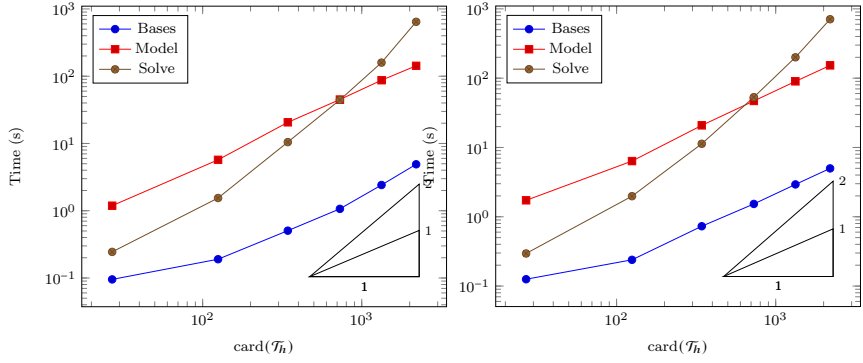


\appendix
\normalsize

\section{Results on local spaces}\label{sec:results}

This section collects miscellaneous results on the Koszul complements defined in \eqref{eq:spaces.F} and \eqref{eq:spaces.T}, as well as on the trimmed spaces \eqref{eq:NE.RT} obtained from the latter.
The first result on traces of Raviart--Thomas and N\'ed\'elec functions is known on simplices, see e.g. \cite[Proposition 2.3.3]{Boffi.Brezzi.ea:13}; we however provide a proof on general polyhedra for sake of completeness.

\begin{proposition}[Traces of N\'ed\'elec and Raviart--Thomas functions]\label{prop:traces.NE.RT}
  It holds, for all $F\in\Fh$,
  \begin{alignat}{4}
    \label{eq:NE.F.trace}
    \forall E\in\EF&\qquad&(\bvec{v}_F)_{|E}\cdot\tangent_E&\in\Poly{\ell-1}(E)
    &\qquad
    &\forall\bvec{v}_F\in\NE{\ell}(F),
    \\
    \label{eq:RT.F.trace}
    \forall E\in\EF&\qquad&(\bvec{w}_F)_{|E}\cdot\normal_{FE}&\in\Poly{\ell-1}(E)
    &\qquad
    &\forall\bvec{w}_F\in\RT{\ell}(F)
  \end{alignat}
  and, for all $T\in\Th$,
  \begin{alignat}{4} \label{eq:NE.T.edge.trace}
    &\forall E\in\ET&\qquad
    (\bvec{v}_T)_{|E}\cdot\tangent_E&\in\Poly{\ell-1}(E)&\qquad
    &\forall\bvec{v}_T\in\NE{\ell}(T),
    \\ \label{eq:RT.T.trace}
    &\forall F\in\FT&\qquad
    (\bvec{w}_T)_{|F}\cdot\normal_F&\in\Poly{\ell-1}(F)&\qquad
    &\forall\bvec{w}_T\in\RT{\ell}(T),
    \\ \label{eq:NE.T.trace}
    &\forall F\in\FT&\qquad
    (\bvec{v}_T)_{|F}\times\normal_F&\in\RT{\ell}(F)&\qquad
    &\forall\bvec{v}_T\in\NE{\ell}(T).
  \end{alignat}
\end{proposition}

\begin{proof}
  \underline{1. \emph{Proof of \eqref{eq:NE.F.trace} and \eqref{eq:NE.T.edge.trace}}.}
  The tangent edge traces of functions in $\Goly{\ell-1}(F)$ (resp.\ $\Goly{\ell-1}(T)$) are in $\Poly{\ell-1}(E)$ for all $E\in\EF$ (resp.\ $E\in\ET$).
  Let, for any $E\in\EF$, $\bvec{x}_E$ denote the middle point of $E$.
  To prove \eqref{eq:NE.F.trace}, it then suffices to recall the definition \eqref{eq:Goly.cGoly.F} and observe that the quantity $(\bvec{x}-\bvec{x}_F)^\perp\cdot\tangent_E=\cancel{(\bvec{x}-\bvec{x}_E)^\perp\cdot\tangent_E}+ (\bvec{x}_E-\bvec{x}_F)^\perp\cdot\tangent_E$ is constant over $E$, the cancellation coming from the fact that $(\bvec{x}-\bvec{x}_E)$ and $\tangent_E$ are parallel for all $\bvec{x}\in E$.
  To prove \eqref{eq:NE.T.edge.trace}, recall the definition \eqref{eq:Goly.cGoly} of $\cGoly{\ell}(T)$ and observe that, for all $\bvec{v}\in\vPoly{\ell-1}(T)$ and all $\bvec{x}\in E$, $\left[(\bvec{x}-\bvec{x}_T)\times\bvec{v}\right]\cdot\tangent_E = \cancel{\left[(\bvec{x}-\bvec{x}_E)\times\bvec{v}\right]\cdot\tangent_E} + \left[(\bvec{x}_E-\bvec{x}_T)\times\bvec{v}\right]\cdot\tangent_E\in\Poly{\ell-1}(E)$, where the cancellation follows observing, as before, that the vectors $(\bvec{x}-\bvec{x}_E)$ and $\tangent_E$ are parallel.
  \medskip\\
  \underline{2. \emph{Proof of \eqref{eq:RT.F.trace} and \eqref{eq:RT.T.trace}}.} The normal traces of functions in $\Roly{\ell-1}(F)$ (resp.\ $\Roly{\ell-1}(T)$) are in $\Poly{\ell-1}(E)$ (resp.\ $\Poly{\ell-1}(F)$) for all $E\in\EF$ (resp.\ $F\in\FT$).
  To conclude, recall the definition \eqref{eq:Roly.cRoly.F} (resp.\ \eqref{eq:Roly.cRoly}) of $\cRoly{\ell}(F)$ (resp.\ $\cRoly{\ell}(T)$) and observe, using similar arguments as above, that the quantity $(\bvec{x}-\bvec{x}_F)\cdot\normal_{FE}$ (resp.\ $(\bvec{x}-\bvec{x}_T)\cdot\normal_F$) is constant for any $\bvec{x}\in E$ (resp.\ $\bvec{x}\in F$).
  This implies, in particular, for all $T\in\Th$,
  \[
    \forall F\in\FT,\qquad
    (\bvec{z}_T)_{|F}\cdot\normal_F\in\Poly{\ell-1}(F)\qquad
    \forall\bvec{z}_T\in\cRoly{\ell}(T).
  \]  
  \underline{3. \emph{Proof of \eqref{eq:NE.T.trace}}.}
  For all $\bvec{v}_T\in \Goly{\ell-1}(T)\subset \vPoly{\ell-1}(T)$ we have $(\bvec{v}_T)_{|F}\times\normal_F\in\vPoly{\ell-1}(F)\subset \RT{\ell}(F)$ since $\normal_F$ is constant.
  It therefore suffices to prove \eqref{eq:NE.T.trace} for $\bvec{v}_T\in\cGoly{\ell}(T)$.
  Recalling \eqref{eq:Goly.cGoly}, there is $\bvec{z}_T\in\vPoly{\ell-1}(T)$ such that $\bvec{v}_T = (\bvec{x}-\bvec{x}_T)\times\bvec{z}_T$.
  Thus, we can write
  \[
  \begin{aligned}
    (\bvec{v}_T)_{|F}\times\normal_F
    &= ( (\bvec{x}-\bvec{x}_T)\times\bvec{z}_T)_{|F}\times\normal_F
    \\
    &= ( (\bvec{x}-\bvec{x}_T)_{|F}\cdot\normal_F )(\bvec{z}_T)_{|F}
    + ( (\bvec{z}_T)_{|F}\cdot\normal_F )(\bvec{x}_T-\bvec{x})_{|F}    
    \\
    &=
    ( (\bvec{x}_F-\bvec{x}_T)_{|F}\cdot\normal_F )(\bvec{z}_T)_{|F}
    + ( (\bvec{z}_T)_{|F}\cdot\normal_F )(\bvec{x}_F-\bvec{x})_{|F}\\
    &\quad- ( (\bvec{z}_T)_{|F}\cdot\normal_F )(\bvec{x}_F-\bvec{x}_T)
    \\
    &=
    \underbrace{ ( (\bvec{x}_F-\bvec{x}_T)\times\bvec{z}_T)_{|F}\times\normal_F }_{\in\vPoly{\ell-1}(F)}
    + \underbrace{ ( (\bvec{z}_T)_{|F}\cdot\normal_F )(\bvec{x}_F-\bvec{x})_{|F} }_{\in\cRoly{\ell}(F)},
  \end{aligned}
  \]
  where we have used the vector algebra identity
  \begin{equation}\label{eq:vector.identity}
    (\bvec{A}\times\bvec{B})\times\bvec{C}
    = (\bvec{A}\cdot\bvec{C})\bvec{B} - (\bvec{B}\cdot\bvec{C})\bvec{A}
    \qquad\forall\bvec{A},\bvec{B},\bvec{C}\in\Real^3
  \end{equation}
  with $\bvec{A}=\bvec{x}-\bvec{x}_T$, $\bvec{B}=\bvec{z}_T$, and $\bvec{C}=\bvec{n}_F$ to pass to the second line;
  to pass to the third line, we have noticed that $(\bvec{x}-\bvec{x}_T)\cdot\normal_F$ is constant on $F$ for the first term, and we have added $\pm\bvec{x}_F$ inside the last parentheses and developed;
  the last line follows from an application of \eqref{eq:vector.identity} with $\bvec{A}=\bvec{x}_F-\bvec{x}_T$, $\bvec{B}=\bvec{z}_T$, and $\bvec{C}=\bvec{n}_F$.
  Since $\Poly{\ell-1}(F)=\Roly{\ell-1}(F)\oplus\cRoly{\ell-1}(F)\subset \Roly{\ell-1}(F)\oplus\cRoly{\ell}(F)=\RT{\ell}(F)$, this concludes the proof.\qed
\end{proof}

\begin{lemma}[Norms of the inverses of local differential isomorphisms]\label{lem:norm.isomorphisms}
  The norms of the inverses of the isomorphisms defined in \eqref{eq:iso:VROTF.GRAD}--\eqref{eq:iso:CURL} satisfy, for all $F\in\Fh$ or $T\in\Th$,
  \[
  \norm{(\VROT_F)^{-1}}\lesssim h_F\,,\ \norm{(\DIV_F)^{-1}}\lesssim h_F\,,\ 
  \norm{(\DIV)^{-1}}\lesssim h_T, \mbox{ and } \norm{(\CURL)^{-1}}\lesssim h_T
  \]
  where, above, $\norm{{\cdot}}$ denotes the norm of the corresponding isomorphism when its domain and co-domains are endowed with
  their $\Leb$-norms, and $a\lesssim b$ means that $a\le Cb$ with $C$ depending only on the polynomial degree $\ell$ and on the mesh regularity parameter.
\end{lemma}

\begin{proof}
  We only prove the estimate on $\norm{(\CURL)^{-1}}$, since the other ones follow from similar arguments. The idea is to use the transport $T\ni\bvec{x}\mapsto \widehat{\bvec{x}}=h_T^{-1}(\bvec{x}-\bvec{x}_T)\in\widehat{T}$ as in the proof of Lemma \ref{lem:piS.piSc}. We recall that $B(\rho)\subset \widehat{T}\subset B(1)$, where $\rho$ is the mesh regularity parameter and $B(r)\coloneq\{\bvec{y}\in\Real^3\,:\,|\bvec{y}|<r\}$.

  Let $\bvec{v}\in\Roly{\ell-1}(T)$ and set $\widehat{\bvec{v}}(\widehat{\bvec{x}})\coloneq\bvec{v}(\bvec{x})$. Given the definition of the change of variable $\bvec{x}\mapsto\widehat{\bvec{x}}$, $\widehat{\bvec{v}}$ belongs to $\Roly{\ell-1}(\widehat{T})$, and can be considered as a polynomial in $\Roly{\ell-1}(\Real^3)$. As $\CURL:\widehat{\bvec{x}}\times\vPoly{\ell-1}(\Real^3)\to \Roly{\ell-1}(\Real^3)$ is an isomorphism, it has a continuous inverse for any pair of norms we choose on the domain and co-domain; we endow $\widehat{\bvec{x}}\times\vPoly{\ell-1}(\Real^3)$ with the $\vLeb(B(\rho))$-norm and $\Roly{\ell-1}(\Real^3)$ with the $\vLeb(B(1))$-norm. The continuity of the inverse of this $\CURL$ operator gives $\widehat{\bvec{w}}\in\widehat{\bvec{x}}\times\vPoly{\ell-1}(\Real^3)$ such that $\CURL\widehat{\bvec{w}}=\widehat{\bvec{v}}$ on $\Real^3$ and $\norm[\vLeb(B(1))]{\widehat{\bvec{w}}}\lesssim \norm[\vLeb(B(\rho))]{\widehat{\bvec{v}}}$, where the hidden constant depends only on the spaces and their norms, that is, on $\ell$ and $\rho$. Since $B(\rho)\subset\widehat{T}\subset B(1)$, this shows that $\CURL\widehat{\bvec{w}}=\widehat{\bvec{v}}$ on $\widehat{T}$ and $\norm[\vLeb(\widehat{T})]{\widehat{\bvec{w}}}\lesssim\norm[\vLeb(\widehat{T})]{\widehat{\bvec{v}}}$.

  For $\bvec{x}\in T$, define $\bvec{w}(\bvec{x})\coloneq h_T\widehat{\bvec{w}}(\widehat{\bvec{x}})$. Then, $\bvec{w}\in (\bvec{x}-\bvec{x}_T)\times\Poly{\ell-1}(T)$, $\CURL\bvec{w}=\bvec{v}$ (the scaling by $h_T$ cancels out the factor $h_T^{-1}$ which appears when differentiating $\bvec{x}\mapsto \widehat{\bvec{w}}(h_T^{-1}(\bvec{x}-\bvec{x}_T))$), and, denoting by $\bvec{J}_T$ the Jacobian of the transport $\widehat{T}\to T$, we have
  \[
  \norm[\vLeb(T)]{\bvec{w}}=h_T |\bvec{J}_T|^{\frac12}\norm[\vLeb(\widehat{T})]{\widehat{\bvec{w}}}\lesssim h_T |\bvec{J}_T|^{\frac12}\norm[\vLeb(\widehat{T})]{\widehat{\bvec{v}}}=h_T\norm[\vLeb(T)]{\bvec{v}},
  \]
  which concludes the proof.\qed
\end{proof}

\section{Curl lifting}\label{appen:RcurlT}

We prove here that the face $\RcurlF$ and element $\RcurlT$ liftings, detailed in Section \ref{sec:consistency.adjoint.curl}, are well defined and satisfy the key properties \eqref{eq:RcurlF:orth} and \eqref{eq:RcurlT:bound}.

\subsection{Face lifting $\RcurlF$}

\subsubsection{Existence of $\bvec{\phi}_{\uvec{v}_F}$}

Owing to \eqref{eq:RcurlF:phi:div}, we look for $\bvec{\phi}_{\uvec{v}_F}=\VROT_F q_F$ for some $q_F\in \Sob{1}(F)$. Using the property $\ROT_F(\VROT_F)=-\DIV_F(\GRAD_F)=-\Delta_F$ (which stems from \eqref{eq:def:VROTF:ROTF}) and that $\VROT_F q_F$ (resp. $\tangent_E$) is $\GRAD_F q_F$ (resp. $\normal_{FE}$) rotated by $-\pi/2$ in the plane spanned by $F$, we see that \eqref{eq:RcurlF:phi} reduces to the following Neumann problem on $q_F$:
\begin{equation}\label{eq:RcurlF.q}
  \begin{alignedat}{3}
    -\Delta_F q_F&=\CF\uvec{v}_F&\quad&\mbox{ in $F$},\\
    \GRAD_F q_F\cdot(\omega_{FE}\normal_{FE})&=\omega_{FE}v_E&\quad&\forall E\in\EF.
  \end{alignedat}
\end{equation}
Recalling that $\omega_{FE}\normal_{FE}$ is the outer normal, in the plane spanned by $F$, to $F$ on $E$, we see that 
the compatibility condition of this Neumann problem simply amounts to the definition \eqref{eq:CF} of $\CF$ with $r_F=1$.
There exists therefore a unique $q_F\in \Sob{1}(F)$ solution of this problem with $\int_F q_F=0$. Using $q_F$ as a test function in 
the weak formulation and applying Cauchy--Schwarz inequalities leads to 
\[
\begin{aligned}
  &\norm[\vLeb(F)]{\GRAD_F q_F}^2\le \norm[\Leb(F)]{\CF\uvec{v}_F}\norm[\Leb(F)]{q_F}+\sum_{E\in\EF}\norm[\Leb(E)]{v_E}\norm[\Leb(E)]{q_F}\\
  &\ \lesssim h_F\norm[\Leb(F)]{\CF\uvec{v}_F}\norm[\vLeb(F)]{\GRAD_F q_F} + \left(
  \sum_{E\in\EF}h_E\norm[\Leb(E)]{v_E}^2
  \right)^{\frac12}
  \norm[\vLeb(F)]{\GRAD_F q_F},
\end{aligned}
\]
where the second line follows from the Poincar\'e--Wirtinger inequality 
\[
\norm[\Leb(F)]{q_F}\lesssim h_F\norm[\vLeb(F)]{\GRAD_F q_F}
\]
together with the continuous trace inequality (see \cite[Remark 1.46 and Lemma 1.31]{Di-Pietro.Droniou:20})
\[
\norm[\Leb(E)]{q_F}\lesssim h_E^{-\nicefrac12}\norm[\Leb(F)]{q_F}+h_E^{\nicefrac12}\norm[\vLeb(F)]{\GRAD_F q_F}.
\]
As a consequence,
\begin{equation}\label{eq:est.phiv}
  \norm[\vLeb(F)]{\bvec{\phi}_{\uvec{v}_F}}\lesssim \norm[\Leb(F)]{\CF\uvec{v}_F}+\tnorm[\CURL,F]{\uvec{v}_F}.
\end{equation}

\subsubsection{Existence of $\psi_{\uvec{v}_F}$}

Fix $\varpi_F\in \rC{\infty}_c(F)$ such that $\varpi_F=1$ on a ball $B_F\subset F$ of radius $\simeq h_F$ (the existence of such a ball follows from the mesh regularity assumption) and $0\le \varpi_F\le 1$. We look for $\psi_{\uvec{v}_F}$ under the form $\varpi_F r_F$ with $r_F\in\Poly{k}(F)$. Since $\DIV_F:\cRoly{k+1}(F)\to \Poly{k}(F)$ is an isomorphism, denoting as in Lemma \ref{lem:norm.isomorphisms} its inverse by $(\DIV_F)^{-1}$, the relation \eqref{eq:RcurlF:psi} is equivalent to
\[
\int_F \varpi_F r_F w_F = \int_F (\trFt\uvec{v}_F-\bvec{\phi}_{\uvec{v}_F})\cdot(\DIV_F)^{-1}w_F\qquad\forall w_F\in\Poly{k}(F).
\]
Since $\varpi_F\ge 0$ is strictly positive on a ball, the mapping $(r_F,w_F)\mapsto\int_F \varpi_F r_F w_F$ is an inner product on $\Poly{k}(F)$ and there exists therefore a unique $r_F\in\Poly{k}(F)$ that satisfies this property. This establishes the existence of $\psi_{\uvec{v}_F}$. 

Moreover, since $\varpi_F= 1$ on $B_F$ and $\norm[\Leb(B_F)]{{\cdot}}$ and $\norm[\Leb(F)]{{\cdot}}$ are uniformly equivalent on $\Poly{k}(F)$ (see the proof of \cite[Lemma 1.25]{Di-Pietro.Droniou:20}), using $w_F=r_F$ above leads to
\begin{align*}
  \norm[\Leb(F)]{r_F}^2\lesssim \int_F\varpi_F r_F^2
  \le{}& \norm[\vLeb(F)]{\trFt\uvec{v}_F-\bvec{\phi}_{\uvec{v}_F}}\norm[\vLeb(F)]{(\DIV_F)^{-1}r_F}\\
  \lesssim{}& \big(\tnorm[\CURL,F]{\uvec{v}_F}+\norm[\Leb(F)]{\CF\uvec{v}_F}\big)h_F\norm[\Leb(F)]{r_F},
\end{align*}
where the conclusion follows from a triangle inequality along with the boundedness \eqref{eq:bound.trFt.Pcurl} of $\trFt$ and the estimate \eqref{eq:est.phiv} for the first factor, and Lemma \ref{lem:norm.isomorphisms} for the second factor.
Simplifying, we obtain
\begin{equation}\label{eq:RcurlF.est.rF}
  \norm[\Leb(F)]{r_F}\lesssim h_F\big(\tnorm[\CURL,F]{\uvec{v}_F}+\norm[\Leb(F)]{\CF\uvec{v}_F}\big).
\end{equation}

\subsubsection{Orthogonality property of $\RcurlF$}

We prove here \eqref{eq:RcurlF:orth}. Notice first that, since $\psi_{\uvec{v}_F}$ vanishes on $\partial F$ and $\ROT_F\GRAD_F=0$, by \eqref{eq:RcurlF:phi} it holds
\begin{equation}\label{eq:Rcurl:rot.tE}
  \ROT_F(\RcurlF\uvec{v}_F) = \CF\uvec{v}_F\quad\mbox{ and }\quad
  (\RcurlF\uvec{v}_F)\cdot\tangent_E=v_E\quad\forall E\in\EF.
\end{equation}
Let $\bvec{z}_F\in\Roly{k}(F)$ and write $\bvec{z}_F=\VROT_F r_F$ with $r_F\in\Poly{0,k+1}(F)$. By \eqref{eq:trFt:Roly.k} and Remark \ref{rem:validity:trFt:Roly.k}, we have
\begin{align*}
  \int_F\trFt\uvec{v}_F\cdot\bvec{z}_F={}&\int_F\CF\uvec{v}_F r_F + \sum_{E\in\FE}\omega_{FE}\int_E v_E r_F\\
  ={}&\int_F\ROT_F(\RcurlF\uvec{v}_F) r_F
  + \sum_{E\in\FE}\omega_{FE}\int_E(\RcurlF\uvec{v}_F)\cdot\tangent_E~ r_F\\
  ={}&\int_F \RcurlF\uvec{v}_F\cdot\bvec{z}_F,
\end{align*}
where the second equality follows from \eqref{eq:Rcurl:rot.tE}, and the conclusion has been obtained using
an integration by parts. This proves that \eqref{eq:RcurlF:orth} holds for $\bvec{z}_F\in\Roly{k}(F)$.

Let us now take $\bvec{z}_F\in\cRoly{k+1}(F)$. Integrating the left-hand side of \eqref{eq:RcurlF:psi} by parts yields
\[
\int_F \GRAD_F\psi_{\uvec{v}_F}\cdot\bvec{z}_F=\int_F(\trFt\uvec{v}_F-\bvec{\phi}_{\uvec{v}_F})\cdot\bvec{z}_F.
\]
Since $\RcurlF\uvec{v}_F=\bvec{\phi}_{\uvec{v}_F}+\GRAD_F\psi_{\uvec{v}_F}$, this establishes that \eqref{eq:RcurlF:orth} also holds for $\bvec{z}_F\in\cRoly{k+1}(F)$, and completes the proof of this orthogonality relation since $\RT{k+1}(F)=\Roly{k}(F)\oplus\cRoly{k+1}(F)$.

\subsection{Element lifting $\RcurlT$} 

\subsubsection{Existence of $\CurlCorr\uvec{v}_T$}

Owing to \eqref{eq:deltaT:curl}, we look for $\CurlCorr\uvec{v}_T$ under the form of a potential gradient $\GRAD q_T$ with $q_T\in \Sob{1}(T)$. Equations \eqref{eq:deltaT:div} and \eqref{eq:deltaT:bc} then show that $q_T$ must solve the Neumann problem
\begin{equation}\label{eq:deltaT.neumann}
  \begin{alignedat}{3}
    \Delta q_T ={}& -\DIV\cCT\uvec{v}_T&\quad&\mbox{ in $T$},\\
    \GRAD q_T\cdot(\omega_{TF}\normal_{F}) ={}&  \omega_{TF}(\CF\uvec{v}_F-\cCT\uvec{v}_T\cdot\normal_F)&\quad&\forall F\in\FT,
  \end{alignedat}
\end{equation}
where we recall that $\omega_{TF}\normal_F$ is the outer normal to $T$ on $F$. The compatibility condition of this problem is
\[
\sum_{F\in\FT}\omega_{TF}\int_F (\CF\uvec{v}_F-\cCT\uvec{v}_T\cdot\normal_F)=-\int_T\DIV\cCT\uvec{v}_T=-\sum_{T\in\FT}\omega_{TF}\int_F
\cCT\uvec{v}_T\cdot\normal_F,
\]
which holds true owing to \eqref{eq:cCT.CF} with $r_T=1$. 
There exists therefore a unique $q_T\in \Sob{1}(T)$ with $\int_T q_T=0$ solution to \eqref{eq:deltaT.neumann}. Using $q_T$ as a test function in the weak formulation of \eqref{eq:deltaT.neumann} yields
\[
\norm[\vLeb(T)]{\GRAD q_T}^2\le \norm[\vLeb(T)]{\cCT\uvec{v}_T}\norm[\vLeb(T)]{\GRAD q_T}+\sum_{F\in\FT}\norm[\Leb(F)]{\CF\uvec{v}_F}\norm[\Leb(F)]{q_T}.
\]
Using the Poincar\'e--Wirtinger and continuous trace inequalities as we did to obtain \eqref{eq:est.phiv}, and recalling that $\CurlCorr\uvec{v}_T=\GRAD q_T$, we infer
\begin{equation}\label{eq:deltaT.bound}
  \norm[\vLeb(T)]{\CurlCorr\uvec{v}_T}\lesssim \norm[\vLeb(T)]{\cCT\uvec{v}_T}+\left(\sum_{F\in\FT}h_F\norm[\Leb(F)]{\CF\uvec{v}_F}^2\right)^{\frac12}\lesssim\tnorm[\DIV,T]{\uCT\uvec{v}_T},
\end{equation}
where the conclusion follows from \eqref{eq:bound:cCT}.

\subsubsection{Existence of $\RcurlT\uvec{v}_T$}

The equation \eqref{eq:RcurlT:div} suggests to look for $\RcurlT\uvec{v}_T=\CURL \bvec{z}_T$. Since adding a gradient to $\bvec{z}_T$ does not change its curl, we can look for $\bvec{z}_T$ in the space
\begin{equation}\label{def:Im.grad.perp}
  \bvec{z}_T\in (\GRAD \Sob{1}(T))^\perp\coloneq\left\{\bvec{w}\in \Hcurl{T}\,:\,\int_T \bvec{w}\cdot\GRAD r=0\quad\forall r\in \Sob{1}(T)\right\}.
\end{equation}
The equations \eqref{eq:RcurlT:curl} and \eqref{eq:RcurlT:bc} then lead to a curl-curl problem on $\bvec{z}_T$, whose variational form is: Find $\bvec{z}_T\in (\GRAD \Sob{1}(T))^\perp$ such that
\begin{multline}\label{eq:curlcurl.zT}
  \int_T \CURL\bvec{z}_T\cdot\CURL\bvec{w}=\int_T (\cCT\uvec{v}_T+\CurlCorr\uvec{v}_T)\cdot\bvec{w} \\
  -\langle \omega_{T\partial T}\RcurlF[\partial T]\uvec{v}_{\partial T}, \bvec{w}\times\normal_{\partial T}\rangle_{\Sob{\nicefrac12}_{\varparallel}(\partial T),\Sob{-\nicefrac12}_{\varparallel}(\partial T)}
  \\
  \forall \bvec{w}\in(\GRAD \Sob{1}(T))^\perp,
\end{multline}
where $\omega_{T,\partial T}\RcurlF[\partial T]\uvec{v}_{\partial T}$ and $\bvec{w}\times\normal_{\partial T}$ are the functions defined on $\partial T$ by setting, respectively,
$(\omega_{T,\partial T}\RcurlF[\partial T]\uvec{v}_{\partial T})_{|F} \coloneq (\omega_{TF}\RcurlF\uvec{v}_F)_{{\rm t},F}$
and $(\bvec{w}\times\normal_{\partial T})_{|F} \coloneq \bvec{w}_{|F}\times\normal_F$ for all $F\in\FT$, $\Sob{\nicefrac12}_{\varparallel}(\partial T)$ is the set of functions on $\partial T$ whose restriction to each face $F\in\FT$ belongs to $\Sob{\nicefrac12}(F)$, and whose tangential traces on the edges are weakly continuous (see \cite[Definition 3.1.2]{Assous.Ciarlet.ea:18} for details), and $\Sob{-\nicefrac12}_{\varparallel}(\partial T)$ is its dual space. Since the solution to \eqref{eq:RcurlF.q} belongs to $\Sob{\nicefrac32}(F)$ (see \cite[Corollary 23.5]{Dauge:88}), the edge tangential trace property in \eqref{eq:Rcurl:rot.tE} ensures that $\omega_{T,\partial T}\RcurlF[\partial T]\uvec{v}_{\partial T}$ indeed belongs to $\Sob{\nicefrac12}_{\varparallel}(\partial T)$.

Owing to the Poincar\'e inequality \eqref{eq:Poincare.curlT} and to the fact that $(\GRAD \Sob{1}(T))^\perp$ is a closed subspace of $\Hcurl{T}$, there exists a unique solution to \eqref{eq:curlcurl.zT}. We now prove that $\bvec{z}_T$ satisfies \eqref{eq:curlcurl.zT} for all $\bvec{w}\in\Hcurl{T}=\GRAD \Sob{1}(T) \oplus (\GRAD\Sob{1}(T))^\perp$, which amounts to showing that the right-hand side vanishes whenever $\bvec{w}=\GRAD r$ for some $r\in \Sob{1}(T)$. By density of smooth functions in $\Sob{1}(T)$, we only need to prove this result for $r\in \rC{\infty}(\overline{T})$. Plugging $\bvec{w}=\GRAD r$ in the right-hand side of \eqref{eq:curlcurl.zT}, the duality product can be written as standard integrals (since $\RcurlF\uvec{v}_F\in\vLeb(F)$ for all $F\in\FT$) and, integrating by parts, we obtain
\[
\begin{aligned}
  &\int_T (\cCT\uvec{v}_T+\CurlCorr\uvec{v}_T)\cdot\GRAD r - \sum_{F\in\FT}\omega_{TF}\int_F \RcurlF\uvec{v}_F\cdot(\GRAD r\times\normal_F)\\
  &\quad=
  -\int_T \cancel{\DIV (\cCT\uvec{v}_T+\CurlCorr\uvec{v}_T)}~r +\sum_{F\in\FT}\omega_{TF}\int_F (\cCT\uvec{v}_T+\CurlCorr\uvec{v}_T)\cdot\normal_F~r\\
  &\qquad
  - \sum_{F\in\FT}\omega_{TF}\int_F \RcurlF\uvec{v}_F\cdot \VROT_F (r_{|F})\\
  &\quad=
  \sum_{F\in\FT}\omega_{TF}\int_F \CF\uvec{v}_F~r- \sum_{F\in\FT}\omega_{TF}\int_F \ROT_F(\RcurlF\uvec{v}_F)~r_{|F}\\
  &\qquad-\sum_{F\in\FT}\sum_{E\in\FE}\omega_{TF}\omega_{FE}\int_E (\RcurlF\uvec{v}_F\cdot\tangent_E)~r_{|F},
\end{aligned}
\]
where we have used \eqref{eq:deltaT:div} to cancel the term in the first equality, and \eqref{eq:deltaT:bc} together with integrations by parts on each face in the second equality. Recalling \eqref{eq:Rcurl:rot.tE} and that $\omega_{TF_1}\omega_{F_1E}+\omega_{TF_2}\omega_{F_2E}=0$ if $F_1,F_2$ are the two faces of $T$ that share the edge $E$, the right-hand side above vanishes, which shows that \eqref{eq:curlcurl.zT} indeed holds for $\bvec{w}=\GRAD r$, and thus for all $\bvec{w}\in\Hcurl{T}$.

Since $\RcurlT\uvec{v}_T=\CURL\bvec{z}_T$, applying this relation to a generic $\bvec{w}\in \vC{\infty}_c(T)$ and integrating by parts yields \eqref{eq:RcurlT:curl}; using then a generic $\bvec{w}\in \vC{\infty}(\overline{T})$ and again integrating by parts, we infer \eqref{eq:RcurlT:bc}. 

\subsubsection{Bound on $\RcurlT$}

We prove here the estimate \eqref{eq:RcurlT:bound}. The estimate on $\CURL\RcurlT\uvec{v}_T$ follows from \eqref{eq:RcurlT:curl}, \eqref{eq:bound:cCT} and \eqref{eq:deltaT.bound}. It remains to bound the $\Leb$-norm of $\RcurlT\uvec{v}_T$.
To do so, we use $\bvec{g}_{\uvec{v}_T}$ provided by Lemma \ref{lem:lift.g} below and an integration by parts \cite[Eq.\ (2.27)]{Assous.Ciarlet.ea:18} to re-cast \eqref{eq:curlcurl.zT} as
\[
\int_T \CURL\bvec{z}_T\cdot\CURL\bvec{w}=\int_T (\cCT\uvec{v}_T+\CurlCorr\uvec{v}_T)\cdot\bvec{w} 
+\int_T\CURL\bvec{w}\cdot \bvec{g}_{\uvec{v}_T}-\int_T \bvec{w}\cdot\CURL\bvec{g}_{\uvec{v}_T}.
\]
Making $\bvec{w}=\bvec{z}_T$, we deduce
\begin{multline*}
\norm[\vLeb(T)]{\CURL \bvec{z}_T}^2\\
\lesssim \tnorm[\DIV,T]{\uCT\uvec{v}_T}h_T\norm[\vLeb(T)]{\CURL\bvec{z}_T}+
\norm[\vLeb(T)]{\CURL \bvec{z}_T}\left(\tnorm[\CURL,T]{\uvec{v}_T}+\tnorm[\DIV,T]{\uCT\uvec{v}_T}\right),
\end{multline*}
where we have invoked \eqref{eq:bound:cCT}, \eqref{eq:deltaT.bound}, the Poincar\'e inequality \eqref{eq:Poincare.curlT}, and \eqref{eq:RcurlT.bound.g} below.
Simplifying, using the norm equivalences \eqref{eq:equiv.norms}, and recalling that $\RcurlT\uvec{v}_T=\CURL\bvec{z}_T$ concludes the proof of the $\Leb$-estimate on $\RcurlT\uvec{v}_T$ stated in \eqref{eq:RcurlT:bound}.

\begin{lemma}[Lifting in $\vSob{1}(T)$]\label{lem:lift.g}
  There exists $\bvec{g}_{\uvec{v}_T}\in \vSob{1}(T)$ such that the tangential trace of $\bvec{g}_{\uvec{v}_T}$ on $\partial T$ is $\RcurlF[\partial T]\uvec{v}_{\partial T}$, and
  \begin{equation}\label{eq:RcurlT.bound.g}
    \norm[\vLeb(T)]{\bvec{g}_{\uvec{v}_T}}+h_T\norm[\vLeb(T)]{\CURL\bvec{g}_{\uvec{v}_T}}\lesssim \tnorm[\CURL,T]{\uvec{v}_T}+\tnorm[\DIV,T]{\uCT\uvec{v}_T}.
  \end{equation}
\end{lemma}

\begin{proof}
  Recalling that
  \begin{equation}\label{eq:Rcurl.pT.dec}
    \RcurlF[\partial T]\uvec{v}_{\partial T}=\bvec{\phi}_{\uvec{v}_{\partial T}}+\GRAD_{\partial T}\psi_{\uvec{v}_{\partial T}},
  \end{equation}
  with obvious notations (each of these functions, restricted to a face $F\in\FT$, corresponds to the function obtained replacing $\partial T$ by $F$), we construct $\bvec{g}_{\uvec{v}_T}=\bvec{g}_{\uvec{v}_T,\bvec{\phi}}+\bvec{g}_{\uvec{v}_T,\psi}$, each addend corresponding to the addends in the decomposition \eqref{eq:Rcurl.pT.dec} of $\RcurlF[\partial T]\uvec{v}_{\partial T}$.
  \medskip\\
  \underline{1. \emph{Construction of $\bvec{g}_{\uvec{v}_T,\bvec{\phi}}$}.}
    We assume, for the moment, that $h_T=1$. By \cite[Corollary 23.5]{Dauge:88} and inverse inequalities on the polynomials $\CF\uvec{v}_F$ and $(v_E)_{E\in\FE}$ (recalling that $1=h_T\simeq h_F\simeq h_E$ for all $F\in\FT$ and $E\in\ET$), there exists $\epsilon\in (0,\nicefrac12)$ such that $\GRAD_F q_F\in \vSob{\nicefrac12+\epsilon}(F)$ and
  \[
  \begin{aligned}
  \norm[\vSob{\nicefrac12+\epsilon}(F)]{\GRAD_F q_F}\lesssim{}&
  \norm[\Leb(F)]{\CF\uvec{v}_F}+\sum_{E\in\EF}\norm[\Leb(E)]{v_E}\\
  \lesssim{}& \norm[\Leb(F)]{\CF\uvec{v}_F}+\tnorm[\CURL,F]{\uvec{v}_F}.
  \end{aligned}
  \]
  Above, when invoking \cite[Corollary 23.5]{Dauge:88}, we have used the fact that, since $\epsilon<\nicefrac12$, the $\Sob{\epsilon}(\partial F)$-norm is equivalent to the sum of the $\Sob{\epsilon}(E)$-norms over $E\in\FE$.
  By construction, $\bvec{\phi}_{\uvec{v}_{\partial T}}$ has strongly continuous tangential traces on the edges of $T$ so
  \[
  \begin{aligned}
    \seminorm[\vSob{\nicefrac12}_{\varparallel}(\partial T)]{\bvec{\phi}_{\uvec{v}_{\partial T}}}^2
    &\lesssim\sum_{F\in\FT}\seminorm[\vSob{\nicefrac12}(F)]{\bvec{\phi}_{\uvec{v}_F}}^2
    =\sum_{F\in\FT}\seminorm[\vSob{\nicefrac12}(F)]{\GRAD_F q_F}^2
    \\
    &\lesssim\sum_{F\in\FT}\norm[\vSob{\nicefrac12+\epsilon}(F)]{\GRAD_F q_F}^2
    \lesssim\sum_{F\in\FT}\left(
    \norm[\Leb(F)]{\CF\uvec{v}_F}+\tnorm[\CURL,F]{\uvec{v}_F}
    \right)^2.
  \end{aligned}
  \]
  Combined with \eqref{eq:est.phiv} and recalling that the local length scales are $\simeq 1$, this leads to
  \[
  \norm[\vLeb(\partial T)]{\bvec{\phi}_{\uvec{v}_{\partial T}}}
  +
  \seminorm[\vSob{\nicefrac12}_{\varparallel}(\partial T)]{\bvec{\phi}_{\uvec{v}_{\partial T}}}
  \lesssim \tnorm[\DIV,T]{\uCT\uvec{v}_T}+\tnorm[\CURL,T]{\uvec{v}_T}.
  \]
  Since $\bvec{\phi}_{\uvec{v}_{\partial T}}$ belongs to $\vSob{\nicefrac12}_{\varparallel}(\partial T)$, by \cite[Theorem 3.1.3]{Assous.Ciarlet.ea:18} there exists $\bvec{g}_{\uvec{v}_T,\bvec{\phi}}\in \vSob{1}(T)$ such that the tangential trace of $\bvec{g}_{\uvec{v}_T,\bvec{\phi}}$ is $\bvec{\phi}_{\uvec{v}_{\partial T}}$ and
  \begin{align*}
  \norm[\vLeb(T)]{\bvec{g}_{\uvec{v}_T,\bvec{\phi}}}+\norm[\vLeb(T)]{\CURL \bvec{g}_{\uvec{v}_T,\bvec{\phi}}}\lesssim{}&
  \norm[\vLeb(\partial T)]{\bvec{\phi}_{\uvec{v}_{\partial T}}}
  +\seminorm[\vSob{\nicefrac12}_{\varparallel}(\partial T)]{\bvec{\phi}_{\uvec{v}_{\partial T}}}\nonumber\\
  \lesssim{}&\tnorm[\DIV,T]{\uCT\uvec{v}_T}+\tnorm[\CURL,T]{\uvec{v}_T}.
  \end{align*}
  This was done under the assumption that $h_T=1$. Using a scaling argument, we infer from the estimate above that, for an element $T$ of generic diameter $h_T$,
  \begin{equation}
    \norm[\vLeb(T)]{\bvec{g}_{\uvec{v}_T,\bvec{\phi}}}+h_T\norm[\vLeb(T)]{\CURL \bvec{g}_{\uvec{v}_T,\bvec{\phi}}}
    \lesssim\tnorm[\DIV,T]{\uCT\uvec{v}_T}+\tnorm[\CURL,T]{\uvec{v}_T}.
    \label{eq:RcurlT.bound.gphi}
  \end{equation}
  \\
  \underline{2. \emph{Construction of $\bvec{g}_{\uvec{v}_T,\psi}$}.}
  By definition, $\bvec{g}_{\uvec{v}_T,\psi}$ is the lifting of $\GRAD_{\partial T}\psi_{\uvec{v}_{\partial T}}$.
  Recalling the construction of each $\psi_{\uvec{v}_F}=\varpi_F r_F$, for $F\in\FT$, we can extend $r_F$ into a polynomial $r_{TF}\in \Poly{k}(T)$ (for example, by making $r_{TF}$ independent of the coordinate perpendicular to $F$). We then have, by \eqref{eq:RcurlF.est.rF},
  \begin{equation}\label{eq:lift.psi.r} 
    \norm[\Leb(T)]{r_{TF}}\lesssim h_T^{\frac12}\norm[\Leb(F)]{r_F}\lesssim h_T \left(h_F^{\frac12}\tnorm[\CURL,F]{\uvec{v}_F}+h_F^{\frac12}\norm[\Leb(F)]{\CF\uvec{v}_F}\right).
  \end{equation}
  The smooth, compactly supported function $\varpi_F$ can be extended in $T$ into $\varpi_{TF}$ such that $0\le \varpi_{TF}\le 1$, $\varpi_{TF}$ has a compact support in a ball of radius $\simeq h_T$ that does not touch the faces in $\FT\backslash\{F\}$, and $|\GRAD\varpi_{TF}|\lesssim h_T^{-1}$. Then, for each $F\in\FT$, the chain rule yields
  \begin{equation}\label{eq:lift.psi.boundTF}
    \begin{aligned}
      \norm[\vLeb(T)]{\GRAD(\varpi_{TF}r_{TF})}
      &\lesssim \norm[\vLeb(T)]{\GRAD r_{TF}}+h_T^{-1}\norm[\Leb(T)]{r_{TF}}
      \\
      &\lesssim h_F^{\frac12}\tnorm[\CURL,F]{\uvec{v}_F}+h_F^{\frac12}\norm[\Leb(F)]{\CF\uvec{v}_F},
    \end{aligned}
  \end{equation}
  where the second inequality follows from an inverse inequality and \eqref{eq:lift.psi.r}. We then set
  \[
  \bvec{g}_{\uvec{v}_T,\psi}=\sum_{F\in\FT} \GRAD (\varpi_{TF}r_{TF})\in \vC{\infty}(\overline{T}).
  \]
  By choice of the supports of $(\varpi_{TF})_{F\in\FT}$, the tangential trace of $\bvec{g}_{\uvec{v}_T,\psi}$ on each face $F\in \FT$ is $\GRAD_F(\varpi_{TF}r_{TF})_{|F}=\GRAD_F\psi_{\uvec{v}_F}$. Moreover, the estimate \eqref{eq:lift.psi.boundTF} gives
  \begin{equation}\label{eq:RcurlT.bound.gpsi}
    \begin{aligned}
    \norm[\vLeb(T)]{\bvec{g}_{\uvec{v}_T,\psi}}\lesssim{}& \left[
      \sum_{F\in\FT}\left(
      h_F\tnorm[\CURL,F]{\uvec{v}_F}^2+h_F\norm[\Leb(F)]{\CF\uvec{v}_F}^2
      \right)
      \right]^{\frac12}\\
      \lesssim{}& \tnorm[\CURL,T]{\uvec{v}_T}+\tnorm[\DIV,T]{\uCT\uvec{v}_T}.
    \end{aligned}
  \end{equation}
  Since $\bvec{g}_{\uvec{v}_T,\psi}$ is a gradient, we also have $\CURL\bvec{g}_{\uvec{v}_T,\psi}=\bvec{0}$ and thus, combining \eqref{eq:RcurlT.bound.gphi} and \eqref{eq:RcurlT.bound.gpsi} yields the estimate \eqref{eq:RcurlT.bound.g} on $\bvec{g}_{\uvec{v}_T}=\bvec{g}_{\uvec{v}_T,\bvec{\phi}}+\bvec{g}_{\uvec{v}_T,\psi}$.\qed
\end{proof}

\begin{lemma}[Local Poincar\'e inequality for $\CURL$]
  With $(\GRAD \Sob{1}(T))^\perp$ defined by \eqref{def:Im.grad.perp}, it holds
  \begin{equation}\label{eq:Poincare.curlT}
    \norm[\vLeb(T)]{\bvec{w}}\lesssim h_T\norm[\vLeb(T)]{\CURL\bvec{w}}\qquad\forall \bvec{w}\in(\GRAD \Sob{1}(T))^\perp.
  \end{equation}
\end{lemma}

\begin{proof}
  By \cite[Theorem 3.4.1]{Assous.Ciarlet.ea:18}, for all $\bvec{v}\in \Hdiv{T}$ such that $\DIV\bvec{v}=0$ and $\langle \bvec{v}\cdot\normal_T,1\rangle_{\partial T}=0$ (where $\langle\cdot,\cdot\rangle_{\partial T}$ is the $\Sob{-\frac12}(\partial T)$--$\Sob{\frac12}(\partial T)$ duality product and $\normal_T$ is the outer normal to $T$), there exists $\bvec{z}\in \Hcurl{T}$ such that $\int_T\bvec{z}=0$ and $\bvec{v}=\CURL\bvec{z}$. Moreover, $\norm[\vLeb(T)]{\bvec{z}}\le C_0\norm[\vLeb(T)]{\bvec{v}}=C_0\norm[\vLeb(T)]{\CURL \bvec{z}}$ and an inspection of the proof shows that $C_0\lesssim h_T$
  (this estimate is obtained via a scaling argument, and noticing that, if $h_T=1$, the constants appearing in the proof of \cite[Theorem 3.4.1]{Assous.Ciarlet.ea:18} do not depend on $T$ under our mesh regularity assumptions).

  Take $\bvec{w}\in (\GRAD \Sob{1}(T))^\perp$ and let $(\bvec{w}_m)_{m\in\Natural}$ be a sequence in $\vC{\infty}(\overline{T})$ which converges to $\bvec{w}$ in $\Hcurl{T}$, see \cite[Proposition 2.2.12]{Assous.Ciarlet.ea:18}. Apply the result above to $\bvec{v}=\CURL\bvec{w}_m$, which satisfies the requirements since, on each $F\in\FT$, we have $\CURL\bvec{w}_m\cdot\normal_{TF}=\ROT_F((\bvec{w}_m)_{t,F})$ (where $\normal_{TF}=(\normal_T)_{|F}$ and, as before, $(\bvec{w}_m)_{t,F}$ is the tangential trace of $\bvec{w}_m$ on $F$, oriented here according to $\normal_{TF}$), and $\bvec{w}_m$ is continuous on $\partial T$. This yields $\bvec{z}_m\in\Hcurl{T}$ such that $\CURL(\bvec{w}_m-\bvec{z}_m)=0$ and $\norm[\vLeb(T)]{\bvec{z}_m}\lesssim h_T \norm[\vLeb(T)]{\CURL\bvec{w}_m}$.
  In particular, since the second Betti number of $T$ is zero, $\bvec{w}_m-\bvec{z}_m\in \GRAD \Sob{1}(T)$, and thus $\int_T (\bvec{w}_m-\bvec{z}_m)\cdot\bvec{w}=0$. Hence,
  \[
  \int_T \bvec{w}_m\cdot\bvec{w}=\int_T \bvec{z}_m\cdot\bvec{w}\lesssim \norm[\vLeb(T)]{\bvec{w}}\norm[\vLeb(T)]{\bvec{z}_m}\lesssim \norm[\vLeb(T)]{\bvec{w}}h_T\norm[\vLeb(T)]{\CURL\bvec{w}_m}.
  \]
  The conclusion follows by letting $m\to\infty$ and simplifying by $\norm[\vLeb(T)]{\bvec{w}}$.\qed
\end{proof}

\section{Notations}\label{appen:notations}

The notations used in the paper follow these rules: polynomial spaces, subspaces and projections are in curly letters; functions and operators with values in $\Real^2$ or $\Real^3$ are in boldface; the exponents indicate the maximum polynomial degree of the space or operator; full discrete gradient and curl, which need to be projected to define the operators in the DDR sequence, are in sans serif (and boldface since they are $\Real^3$-valued); spaces, vectors and operators made of components attached to mesh entities of different dimensions are underlined.
Table \ref{tab:notations} lists the main notations used in the design and analysis of the DDR complex.

\begin{table}\centering
  \renewcommand{\arraystretch}{1.3}
  \begin{tabular}{>{\centering}m{.25\linewidth}|m{.45\linewidth}|m{.2\linewidth}}
    \toprule
    \textbf{Notation} & \multicolumn{1}{c|}{\textbf{Meaning}} & \textbf{Reference} \\
    \hline
    $\Goly{\ell}$, $\Roly{\ell}$, $\cGoly{\ell}$, $\cRoly{\ell}$ & Polynomial ranges of gradient/curl, and complements & \eqref{eq:spaces.F}, \eqref{eq:spaces.T}\\
    \hline
    $\Gproj{\ell}{X}$, $\Rproj{\ell}{X}$, $\Gcproj{\ell}{X}$, $\Rcproj{\ell}{X}$ & $\Leb$-projections on $\Goly{\ell}(X)$, $\Roly{\ell}(X)$, $\cGoly{\ell}(X)$, $\cRoly{\ell}(X)$& Section \ref{sec:polynomial.spaces}\\
    \hline
    $\NE{\ell}$, $\RT{\ell}$ & N\'ed\'elec and Raviart--Thomas spaces & \eqref{eq:NE.RT} \\
    \hline
    $\rec{S,S^\compl}{\cdot}{\cdot}$, $\Xrec{X}{Y}{\ell}{\cdot}{\cdot}$
    & Recovery operator & \eqref{eq:recovery.operator}, \eqref{eq:def.Xrec}\\
    \hline
    $\Xgrad{h}$, $\Xcurl{h}$, $\Xdiv{h}$ & Spaces of the DDR sequence & \eqref{eq:Xgrad.h}--\eqref{eq:Xdiv.h}\\
    \hline
    $\Igrad{h}$, $\Icurl{h}$, $\Idiv{h}$ & Interpolators on the DDR spaces & \eqref{eq:Igradh}--\eqref{eq:Idivh}\\
    \hline
    $G_E$, $\cGF$, $\cGT$ & Edge, face and element full gradients & \eqref{eq:GE}, \eqref{eq:cGF}, \eqref{eq:cGT}\\
    \hline
    $\CF$, $\cCT$ & Face and element full curl & \eqref{eq:CF}, \eqref{eq:cCT}\\
    \hline
    $\DT$ & Element divergence & \eqref{eq:DT}\\
    \hline
    $\uGh$, $\uCh$, $\Dh$ & Global discrete gradient, curl and divergence & \eqref{eq:uGh}, \eqref{eq:uCh}, \eqref{eq:Dh}\\
    \hline
    $\trF$, $\Pgrad$ & Scalar trace and potential on $\Xgrad{T}$ & \eqref{eq:trF}, \eqref{eq:PgradT}\\
    \hline
    $\trFt$, $\Pcurl$ & Tangential trace and potential on $\Xcurl{T}$ & \eqref{eq:trFt}, \eqref{eq:Pcurl}\\
    \hline
    $\Pdiv$ & Potential on $\Xdiv{T}$ & \eqref{eq:Pdiv}\\
    \hline
    $(\cdot,\cdot)_{\bullet,T}$, $(\cdot,\cdot)_{\bullet,h}$, $\norm[\bullet,h]{{\cdot}}$ & \makecell[l]{Discrete $\Leb$-products and norms\\ ($\bullet\in\{\GRAD,\CURL,\DIV\}$)} & Section \ref{sec:discrete.L2}\\
    \hline
    $\tnorm[\bullet,h]{{\cdot}}$ & Components norm ($\bullet\in\{\GRAD,\CURL,\DIV\}$) & Section \ref{sec:components.norms}\\
    \bottomrule
  \end{tabular}
  \caption{Notations for spaces and operators.}
  \label{tab:notations}
\end{table}


\bibliographystyle{plain}
\bibliography{ddr-variant}

\end{document}